\documentclass[captions=topbeside, 10pt]{article}

\usepackage[normalem]{ulem}
\usepackage{etex}
\usepackage[left=2cm,top=2cm,right=2cm,bottom=2cm]{geometry}
\usepackage{amsmath, amsthm, amssymb}
\usepackage{pst-all}
\usepackage[textwidth=0.5in]{todonotes}
\usepackage[rightcaption]{sidecap}

\usepackage{float}
\usepackage[caption = false]{subfig}

\usepackage[scaled=1]{helvet}
\usepackage{titlesec}
\usepackage{multido}
\usepackage{graphicx}
\usepackage{multirow}
\usepackage{blindtext}
\usepackage{color}
\usepackage{lipsum}
\usepackage{mathtools}
\usepackage[vcentermath]{youngtab}
\usepackage{young}
\usepackage[all,ps,dvips,graph]{xy}
\usepackage[misc,geometry]{ifsym}
\usepackage{genyoungtabtikz}

\usepackage{titlesec}
\usepackage{lineno,hyperref}

\let\OLDthebibliography\thebibliography
\renewcommand\thebibliography[1]{
  \OLDthebibliography{#1}
  \setlength{\parskip}{0pt}
  \setlength{\itemsep}{0pt plus 0.3ex}
}

\titleformat{\section} {\normalfont\scshape \large \centering}{ \thesection}{1em}{}
\definecolor{morado}{rgb}{0.5,0,0.5}

\newcommand{\cupdot}{\mathbin{\mathaccent\cdot\cup}}

\newcommand{\once}{11}

\newcommand{\II}{I_e}
\newcommand{\shape}{\text{shape}}

\newcommand{\F}{ { \mathbb F}}

\newcommand{\Z}{\mathbb{Z}}

\newcommand{\B}{\mathbb{B}_n}

\newcommand{\Basis}{\mathcal{C}_n}

\newcommand{\res}{ \textrm{res} }

\newcommand{\End}{{\rm End}}

\newcommand\bt{\mathbf{t}}

\newcommand{\s}{\mathfrak{s}}

\newcommand{\T}{  \mathfrak{t}}

\newcommand{\bb}{  \mathfrak{b}}

\newcommand\bi{\boldsymbol{i}}

\newcommand\bj{\boldsymbol{j}}

\newcommand{\OnePar}{{ \rm Par}^1_{n}}
\newcommand{\OneParSing}{{ \rm Par}^1_{\overline{n}}}

\newcommand{\Si}{\mathfrak{S}}
\newcommand{\std}{{\rm Std}}

\newcommand{\tab}{{\rm Tab}}

\newcommand\blambda{{\boldsymbol\lambda}}

\newcommand\bnu{{\boldsymbol\nu}}

\newcommand\bmu{{\boldsymbol\mu}}
\newcommand{\YL}{{\mathcal Y}^\blambda}

\newcommand{\bT}{\pmb{\mathfrak{t}}}

\newcommand{\trunc}{ {\B} (\blambda ) }
\newcommand{\truncPrime}{ {\mathbb B}_n^{\prime} (\blambda ) }

\newcommand{\sign}{-}
\newcommand{\UU}{\mathbb{U}}
\newcommand{\VV}{\mathbb{V}}
\newcommand{\YY}{\mathbb{Y}}
\newcommand{\LL}{\mathbb{L}}

\newcommand{\Exp}{ {\rm \bf exp} }

\newcommand{\botts}[2]{ \mbox{Hom}_{{\mathcal D}} ( \underline{#1}, \underline{#2} ) }

\newcommand{\ese}{  {\color{red} s}}
\newcommand{\te}{ {\color{blue} t}}

\modulolinenumbers[5]

\newtheorem{theorem}{Theorem}[section]
\newtheorem{lemma}[theorem]{Lemma}

\newtheorem{definition}[theorem]{Definition}
\newtheorem{corollary}[theorem]{Corollary}

\newtheorem{remark}[theorem]{Remark}

\newtheorem{algorithm}[theorem]{Algorithm}

\newenvironment{dem}{\noindent \textit{Proof:} }{\quad \hfill $\square$}

\numberwithin{equation}{section}

\newgray{plomoclaro}{.90}
\newgray{plomooscuro}{.70}
\newgray{blanco}{1}

\begin{document}
\Yvcentermath1
\sidecaptionvpos{figure}{lc}


\title{The nil-blob algebra: An incarnation of type $\tilde{A}_1$ Soergel calculus and of the truncated blob algebra.}

\author{Diego Lobos{\thanks{Supported in part by CONICYT-PCHA/Doctorado Nacional/2016-21160722}}, 
  \, David Plaza{\thanks{Supported in part by FONDECYT-Iniciaci\'on grant 11160154 and FONDECYT grant 1200341}} \, and  Steen Ryom-Hansen{\thanks{Supported in part by FONDECYT grant 1171379  }}}





\maketitle
\begin{abstract}
  We introduce a type $B$ analogue of the nil Temperley-Lieb algebra in terms of generators and relations,
  that we call the
  (extended) nil-blob algebra. We show that this algebra is isomorphic to the endomorphism
  algebra of a Bott-Samelson bimodule in type $\tilde{A}_1$. We also prove that it is isomorphic
  to an idempotent truncation of the classical blob algebra. 
\end{abstract}

\section{Introduction}

\titleformat{\subsection}[runin]
{\scshape\bfseries}
{ \S \thesubsection .} {1ex}{}[.\quad] \subsection{Motivation}

The study of diagram algebras and categories is currently one of the most active and vibrant areas of representation theory. Among these diagrammatically defined objects of study there are two that play a prominent role: KLR algebras (categories) and Elias and Williamson's diagrammatic Soergel category. Despite being defined with different
motivations, in recent years we have seen close connections emerging between these two worlds. For example, Riche and Williamson showed in \cite{RicheWill}
that the diagrammatic Soergel category acts on the category of tilting modules for $ GL_n$, via 
an action of the KLR-category. Similar, but not equivalent, ideas were exploited by Elias and Losev in \cite{elias-losev}. In that paper the authors work in the opposite direction; they equip a Soergel-type category with a KLR action.   

\medskip
In the same vein, Libedinsky and the second author introduced in \cite{LiPl}
the \emph{Categorical Blob v/s Soergel conjecture} (B(v/s)S-conjecture for short) which posits an equivalence between full subcategories, one for each element in the affine Weyl group of type $\tilde{A}_n$, of the diagrammatic
Soergel category in type $\tilde{A}_n$ and a certain category obtained from a quotient of cyclotomic KLR-algebras: 
the generalized blob algebras of Martin and Woodcock \cite{MW}. Although similar, this conjecture has a
fundamental  difference with the two aforementioned works: there is no action involved in it! Roughly speaking,
we think of Riche and Williamson's and Elias and Losev's works as a kind of 'categorical Schur-Weyl duality'.
Just like classical Schur-Weyl duality allows us to pass representation theoretical information between the symmetric group and the general linear group, these works allow us to transfer representation theoretical information between the KLR world and the Soergel world. On the other hand, the B(v/s)S-conjecture tells us that the two worlds are
really the 'same'. We stress that both categories involved in this conjecture are skeletal and therefore the equivalence is indeed an isomorphism of categories.    

\medskip
Five months after this paper was finished we learnt from Bowman, Cox and Hazi that they have
obtained a proof of the B(v/s)S-conjecture, see \cite{BCH}. As a matter of fact, their results are more general
and the conjecture follows as a particular case. Despite being more general, their work keeps the original spirit of the B(v/s)S-conjecture: the two worlds are the same.

\medskip 
A common feature of the works mentioned so far is the fact that they imply that the relevant decomposition numbers in the counterpart of the Soergel-like category are controlled by the  $p$-Kazhdan-Lusztig basis. For this and other reasons, understanding the $p$-Kazhdan-Lusztig basis has become one of the most important problems in representation theory. For example, in type $ \tilde{A}_n$ a solution
to this problem would give a solution to the longstanding problem of finding the decomposition numbers
for the symmetric groups in characteristic $p$. Unfortunately, we are far from a full understanding of the $p$-Kazhdan-Lusztig basis. This poor understanding of the $p$-Kazhdan-Lusztig basis is due in part to a lack of a good understanding of the multiplicative structure of the diagrammatic Soergel category and its KLR counterparts. It is with the intention of unraveling these multiplicative structures that this article comes into existence. 

\subsection{Algebras}
In this paper we investigate three (more precisely five) different, although well-known, diagram algebras.  

\medskip
The first algebra of our paper is a variation 
of the blob algebra $\mathbb{B}_n$. The blob algebra was introduced by Martin and Saleur in
\cite{Mat-Sal} via motivations in statistical mechanics. It is a generalization of the Temperley-Lieb algebra
and in fact its diagram basis consists of certain marked Temperley-Lieb diagrams. 
{\color{black}{The first diagram algebra of our paper has the same diagram basis as $\mathbb{B}_n$, but
 we endow it  
with a different multiplication rule.}}

\medskip
For our second diagram algebra we choose $ (W,S) $ of type $ \tilde{A}_1$ and consider a
diagrammatically defined \textbf{subalgebra} of the endomorphism
algebra $ {\rm End}_{\cal D}(\underline{w}) $, where $ \underline{w} $ is any reduced  expression over $ S$ and  $\cal D$ denotes the diagrammatic Elias and Williamson's category.


\medskip

Our third diagram algebra comes from the  KLR world. The second and the third author showed in \cite{PlazaRyom} that a {\color{black}{quotient}}
of the KLR algebra is isomorphic
to the blob algebra $\mathbb{B}_n$, but our third diagram algebra is a slightly different variation of this algebra,
given by idempotent truncation with respect to a \textbf{singular} weight in the associated alcove
geometry.


\medskip

In our paper we provide a presentation for each of the three algebras in terms of generators and relations. The three presentations turn out to be identical. Indeed, our first main theorem is the following.

\begin{theorem} \label{teo intro}
The three aforementioned algebras have a presentation with generators 
$\UU_0,\UU_1, \ldots , \UU_{n-1} $
  subject to the relations 
	\begin{align}
\label{eq oneintro}	\UU_i^2	& =  \sign 2 \UU_i,  &   &  \mbox{if }  1\leq i < n; \\
\label{eq twointro}	\UU_i\UU_j\UU_i & =\UU_i,    &  & \mbox{if } |i-j|=1 \mbox{ and } i,j >0;  \\
\label{eq threeintro}	\UU_i\UU_j& = \UU_j\UU_i,    &  &   \mbox{if } |i-j|>1; \\
\label{eq fourintro}	\UU_1\UU_0\UU_1& =  0,  &  &   \\
\label{eq fiveintro}			\UU_0^2& = 0.   &  &  
	\end{align}	
\end{theorem}

As far as we know, the abstract algebra
defined by the common presentation of the three algebras has not appeared before in the literature; it 
is the nil-blob algebra $\mathbb{NB}_n$ of the title of the paper. 

\medskip 
We also provide a 'regular' version of Theorem \ref{teo intro}. On the one hand, consider the \textbf{full} endomorphism algebra $\End_{\cal D} (\underline{w})$. On the other hand, consider the truncated blob algebra 
with respect to a \textbf{regular} weight in the associated alcove
geometry. Our second main theorem is the following.

\begin{theorem} \label{teo intro two}
	These two algebras have a presentation by generators $\UU_0,\UU_1, \ldots , \UU_{n-1} $ and $\mathbb{J}_n $
        subject to the relations \ref{eq oneintro}--\ref{eq fiveintro} and  $\mathbb{J}_n^2=0 $, together with
        relations saying that $\mathbb{J}_n$ is central. 
\end{theorem}


\medskip
We would like to point out 
that our results are by no means consequences of general principles.
Indeed, in general a presentation for an associative algebra $ \mathcal A$ does not automatically induce 
a presentation for an (idempotent truncated) subalgebra of $ \mathcal A $,
and in fact our generators for the idempotent truncation of $\mathbb{B}_n$
are highly non-trivial expressions in the KLR-generators for $\mathbb{B}_n$. Similarly, 
a presentation for a category $ \mathcal C $ does not automatically induce a 
presentation for $\mbox{End}_{\mathcal{C}}(M)$, where $ M $ is an object of $ \mathcal C$, 
and in fact our generators for 
$\mbox{End}_{\mathcal{D}}(\underline{w})$ are non-trivial expressions in  
Elias and Williamson's generators for $ \cal D $. By similar reasons, our results do not follow from the work of Bowman, Cox and Hazi \cite{BCH}, since we cannot obtain in any direct way  a presentation for the relevant algebras from their isomorphism's theorem.

\medskip
For type  $\tilde{A}_1$ we consider Theorem \ref{teo intro} and Theorem \ref{teo intro two} as a
satisfactory answer to the question raised at the end of the previous section: understanding the multiplicative structure of the diagrammatic Soergel category and its KLR counterpart.  In this setting, we feel that similar presentations for the analogous  algebras in type $\tilde{A}_n$  would give us some hope of being able to calculate the  $p$-Kazhdan-Lusztig basis. We are not claiming, of course, that this will be enough to compute the $p$-Kazhdan-Lusztig basis, 
but it would be a significant step towards this goal.  

\medskip
We conclude this section by  highlighting the simplicity of the relations of the nil-blob algebra which 
should be contrasted with the much more complicated
relations in the definitions of the Soergel and KLR diagrams. In short,
once the correct point of view is found, 
complicated diagrams manipulations become easier ones.
The above give us some reasons to be optimistic with respect to a possible generalization of our results for type $\tilde{A}_n$. We expect to consider this problem elsewhere in the future.

\subsection{Structure of the paper}
Let us briefly indicate the layout of the paper.
Throughout the paper we fix a ground field
$\mathbb{F}$ with $\mbox{char}(\mathbb{F})\neq 2$.
In the following section 2 we introduce the
main object of our paper, namely the nil-blob algebra $\mathbb{NB}_n$.
 We also introduce the extended nil-blob algebra $\widetilde{\mathbb{NB}}_n$
by adding an extra generator $\mathbb{J}_n$ which is central in $\widetilde{\mathbb{NB}}_n$.
We next go on to prove that $\mathbb{NB}_n$ is a diagram algebra where the diagram
basis is the same as the one used for the original blob algebra, but where the multiplication rule is modified.
The candidates for the diagrammatic counterparts of the generators $ \UU_{i}$'s are
the obvious ones, but the fact 
that these diagrams generate the diagram algebra is not so obvious. We establish it 
in Theorem \ref{blob diagram realization}. From this Theorem we obtain
the dimensions of $\mathbb{NB}_n$ and $\widetilde{\mathbb{NB}}_n$ and we also deduce from it that 
$\mathbb{NB}_n$ is a cellular algebra in the sense of Graham and Lehrer. Finally, we indicate that
this cellular structure is endowed with a family of JM-elements, in the sense of Mathas.

\medskip
Section 3 of our paper is devoted to the diagrammatic Soergel category $ \mathcal D $.
We begin the section by recalling the relevant notations and definitions concerning $ \mathcal D $.
This part of the section is valid for general Coxeter systems $ (W,S) $, but we soon
focus on type $\tilde{A}_1$, with $ S = \{ {\color{red} s}, {\color{blue} t} \} $.
The objects of $ \mathcal D $ are expression over $ S $. We fix the expression
$ \underline{w}:= \underbrace{{\color{red} s} {\color{blue} t}
{\color{red} s}   \ldots }_{n  \scriptsize \mbox{-times}}$ and consider throughout the section
the corresponding endomorphism algebra $ \tilde{A}_{w}:={\rm End}_{\mathcal D}(\underline{w})$.
This is the second diagram algebra of our paper. We find diagrammatic counterparts of the $ \UU_{i}$'s and
$\mathbb{J}_n$ and
obtain from this homomorphisms from
$\mathbb{NB}_n$ and $\widetilde{\mathbb{NB}}_n$ to $ \tilde{A}_{w}$.
The diagram basis for $ \tilde{A}_{w}$ is 
Elias and Williamson's diagrammatic version of Libedinsky's light leaves
and it is a cellular basis for $ \tilde{A}_{w}$. 
For general $(W,S)$ the combinatorics of this basis is quite complicated, 
but in type $\tilde{A}_1$ it is much easier, in particular there is a non-recursive
description of it, due to Libedinsky. Using this we obtain in Theorem {\ref{mainTheoremSection3}}
and Corollary \ref{corollary short presentation} the main results of 
this section, stating that there is a diagrammatically defined subalgebra $ {A}_{w}$ of
$ \tilde{A}_{w}$ and that the above homomorphisms induce isomorphisms $ \mathbb{NB}_n \cong {A}_{w} $
and $\widetilde{\mathbb{NB}}_n \cong \tilde{A}_{w}$. Similarly to the situation in section 2, 
the most difficult part of these results is the fact that the diagrammatic counterparts of the 
$ \UU_{i}$'s and $\mathbb{J}_n$ generate the algebras in question. The proof of this generation result relies on
long calculations with Soergel calculus, and is quite different from the proof of the generation result of the
previous section 2.

\medskip
In the rest of the paper, that is in sections 4, 5 and 6, we consider the idempotent
truncation of the blob algebra. This is technically the most difficult part of our paper, but in fact we first discovered
our results in this setting. 

\medskip
In section 4 we fix the notation and give the necessary background for the KLR-approach to
the representation of the blob algebra. In particular we recall the graded cellular basis for
$\mathbb{B}_n$, introduced in \cite{PlazaRyom}, the relevant alcove geometry, which is of type $\tilde{A}_1$,
and the
idempotent truncated subalgebra $ \trunc $ of $\mathbb{B}_n$.
This is the diagram algebra that is studied in the rest of the paper. 
We use the alcove geometry to distinguish between
the regular and the singular cases for $\trunc$. We also recall the indexation of the
cellular basis in terms of paths in this geometry. Finally, in
Algorithm \ref{algorithm paths} and Theorem \ref{Algorithm}
we explain how to obtain
reduced expressions for the group elements associated with these paths, in the symmetric group $ \Si_n $.
We remark that Algorithm \ref{algorithm paths}
has certain flexibility built in, which is of importance for the following sections.

\medskip
In section 5 we consider $ \trunc $ in the singular case. The main result is our Theorem 
\ref{theoremsing}, establishing an isomorphism $ \trunc \cong \mathbb{NB}_n$.
The idea behind this isomorphism Theorem is essentially the same as the idea behind the previous two
isomorphism Theorems, but once again the technical details are very different.
The diagrammatic counterparts of the generators $ \UU_1, \ldots, \UU_{n-1} $ are
here the 'diamond' diagrams found recently by Libedinsky and the second author in \cite{LiPl}, whereas
the diagrammatic counterpart of $  \UU_0$ is given directly by the KLR-type presentation. 
Once again, the most difficult part of the isomorphism Theorem is the fact that
these elements actually generate the whole diagram algebra.
We obtain this fact by showing that the graded cellular basis elements for $\trunc $ can all be written in terms of
them. This involves calculations with the KLR-relations.


\medskip
Finally, in section 6 we consider the regular case which is slightly more
complicated than the singular case. Our main result is here Theorem
\ref{main theorem regular}, establishing the isomorphism $ \trunc \cong 
\widetilde{\mathbb{NB}}_n $. The proof involves more calculations with the KLR-relations, in the same spirit
as the ones in section 5.

\section{The nil-blob algebra}
Throughout the paper we fix a field $\mathbb{F}$ with $\mbox{char}(\mathbb{F})\neq 2$.
All our algebras are associative and unital $\mathbb{F}$-algebras.

\medskip
In this section we introduce and study the basic properties of the nil-blob algebra.
Let us first recall the definition of the classical blob algebra $\mathbb{B}_n$. It was introduced by Martin and Saleur in
\cite{Mat-Sal}. We fix $ q \in \mathbb{F}^{\times} $ and define for any $ k \in \Z$
the usual Gaussian integer
\begin{equation}\label{Gaussian}
[k] := q^{ k-1} +  q^{ k-3} +  \ldots  +q^{- k+3} +  q^{ -k+1}.
\end{equation}

\begin{definition}\label{definition original blob}
Let $ m \in \Z$ with $[m]\neq 0$. 
  The blob algebra $\mathbb{B}_n(m) = \mathbb{B}_n$ is the algebra generated by $\VV_0,\VV_1, \ldots , \VV_{n-1} $
  subject to the relations 
	\begin{align}
\label{eq one}	\VV_i^2	& =  \sign [2] \VV_i,  &   &  \mbox{if }  1\leq i < n; \\
\label{eq two}	\VV_i\VV_j\VV_i & =\VV_i,   &  & \mbox{if } |i-j|=1 \mbox{ and } i,j >0;  \\
\label{eq three}	\VV_i\VV_j& = \VV_j\VV_i ,    &  &   \mbox{if } |i-j|>1 ;\\
\label{eq four}	\VV_1\VV_0\VV_1& =  [m-1] \VV_1,   &  &   \\
\label{eq five}			\VV_0^2& = -[m] \VV_0.   &  &  
	\end{align}
\end{definition}

An important feature of $\mathbb{B}_n$ is the fact that it is a diagram algebra. 
The diagram basis consists of blobbed (marked) Temperley-Lieb diagrams on $ n $ points where
only arcs exposed to the left side of the diagram may be marked and at most once.
The multiplication  
{\color{black}{$ D_1 D_2 $ of two diagrams $ D_1 $ and $ D_2 $ 
is given by concatenation of them, with $ D_1 $ on top of $ D_2$.}} This concatenation process
may give rise to internal 
marked or unmarked loops, as well as arcs with more than one mark.  The internal unmarked loops are removed from a diagram
by multiplying it by $  \sign [ 2]$, whereas the internal marked loops are removed
from a diagram by multiplying it by $ -[m-1]/[m]$. Finally, any diagram with $r>1$ marks on an arc is set equal to the same diagram with the $(r-1)$ extra marks removed.
These marked Temperley-Lieb diagrams are called blob diagrams.
In Figure \ref{fig:an example with $ n=20$} we give an example with $ n=20$.
{\color{black}{The color red in Figure \ref{fig:an example with $ n=20$} is only used to indicate those arcs that are not exposed to the left side of the diagram and therefore
cannot be marked. For any of the black arcs the blob is optional.}}


\begin{figure}[h]
\raisebox{-.5\height}{\includegraphics{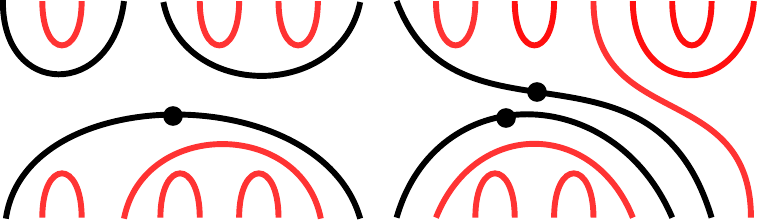}}
\centering
\caption{Blob diagram example with $n=20$.}
\label{fig:an example with $ n=20$}
\end{figure}

\medskip
Motivated in part by $\mathbb{B}_n$ we now define the nil-blob algebra $\mathbb{NB}_n$ and
its extended version $\widetilde{\mathbb{NB}}_n$. They are the main objects of
study of this paper.
\begin{definition} \label{definition blob}
  The nil-blob algebra $\mathbb{NB}_n$ is the algebra on the generators $\UU_0,\UU_1, \ldots , \UU_{n-1} $
  subject to the relations 
	\begin{align}
\label{eq one}	\UU_i^2	& =  \sign 2 \UU_i,  &   &  \mbox{if }  1\leq i < n; \\
\label{eq two}	\UU_i\UU_j\UU_i & =\UU_i,    &  & \mbox{if } |i-j|=1 \mbox{ and } i,j >0;  \\
\label{eq three}	\UU_i\UU_j& = \UU_j\UU_i,    &  &   \mbox{if } |i-j|>1; \\
\label{eq four}	\UU_1\UU_0\UU_1& =  0,  &  &   \\
\label{eq five}			\UU_0^2& = 0.   &  &  
	\end{align}
The extended nil-blob algebra $\widetilde{\mathbb{NB}}_n$ is the algebra obtained from $\mathbb{NB}_n$
by adding an extra generator $\mathbb{J}_n$ which is central and satisfies $ \mathbb{J}_n^2=0 $. 
\end{definition}

\begin{remark} \rm
  Note that the sign in \ref{eq one} is unimportant. Indeed, replacing
  $ \UU_i $ with $- \UU_i $ we get a presentation as in Definition \ref{definition blob} but with the sign in 
\ref{eq one} positive.
\end{remark}

It is known from \cite{PlazaRyom} that $\mathbb{B}_n$ is a $ \Z$-graded algebra.
This is also the case for $\mathbb{NB}_n$ and $\widetilde{\mathbb{NB}}_n$ but is actually much easier to prove.
\begin{lemma}\label{referenceLemma}
  The rules $ {\rm deg}(\UU_i) = 0 $ for $ i > 0 $ and $ {\rm deg}(\UU_0) ={\rm deg}(\mathbb{J}_n) = 2 $ 
  define (positive) $ \Z$-gradings on $\mathbb{NB}_n$ and $\widetilde{\mathbb{NB}}_n$.
\end{lemma}  
\begin{proof}
One checks easily that the relations are homogeneous with respect to ${\rm deg} $.
\end{proof}  

Our first goal is to show that $\mathbb{NB}_n$ 
is a diagram algebra with the same diagram basis as for $\mathbb{B}_n$,
but with a slightly different multiplication rule.
Indeed, in $\mathbb{NB}_n$ 
an internal unmarked loop is removed from a diagram
by multiplying it with $  \sign 2 $, whereas diagrams in $\mathbb{NB}_n$ with a marked loop are set to zero.
Moreover, in $\mathbb{NB}_n$ diagrams with a multiple marked arc are also set equal to zero.
{\color{black}{This defines an associative multiplication}} with 
identity element given as 


\begin{equation} \! \!\! \!\! \! \! \!
  1 = \raisebox{-.5\height}{\includegraphics{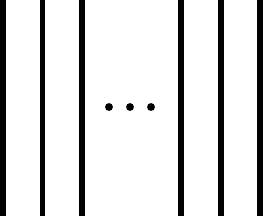}}
\end{equation}




That $\mathbb{NB}_n$ has this diagram realization follows from the results presented in
the Appendix of \cite{blob positive}, 
but for the reader's convenience we here present a different more self-contained proof of this fact,
avoiding the theory of projection algebras.
Let us denote by $\mathbb{NB}_n^{diag}$ the diagram algebra indicated above, with basis given by blob diagrams
and multiplication rule as explained in the previous paragraph. We then prove the following Theorem:
\begin{theorem}\label{blob diagram realization}
There is an isomorphism between $\mathbb{NB}_n$ and $\mathbb{NB}_n^{diag}$ induced by 


\begin{equation}{\label{isomorphism}} \! \!\! \!\! \! \! \!
  \UU_0 \mapsto \quad \raisebox{-.5\height}{\includegraphics{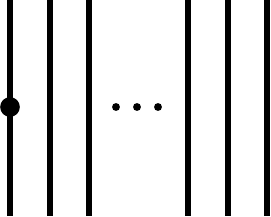}} \,,  
   \, \,\,\, \,
   \UU_i \mapsto \quad \raisebox{-.5\height}{\includegraphics{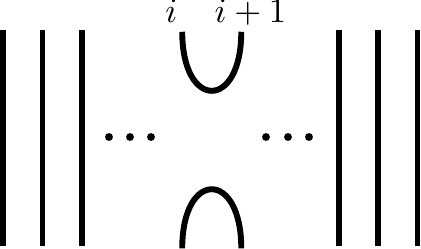}} 
\end{equation}

\noindent
In particular,
$\mathbb{NB}_n$ has the same dimension as $\mathbb{B}_n$, in other words
\medskip 
\begin{equation}
  \dim_{\mathbb{F}}(\mathbb{NB}_n) = \binom{2n}{n}.
\end{equation}	
\end{theorem}
\begin{proof}
 One easily checks that the diagrams in {\ref{isomorphism}} satisfy the
   relations for the $ \UU_i$'s in Definition \ref{definition blob} and so at least
   {\ref{isomorphism}} induces an algebra homomorphism $ \varphi: \mathbb{NB}_n \rightarrow
 \mathbb{NB}_n^{diag} $.   

 Although it is not possible to determine the dimension of $ \mathbb{NB}_n $ directly, we can still get an
 upper bound for it using normal forms as follows. For $ 0 \le  j \le i \le n-1 $ we define
 \begin{equation} \UU_{ i j} := \UU_i \UU_{i-1} \cdots   \UU_{j+1} \UU_{j}  \in  \mathbb{NB}_n.
 \end{equation}   
 We consider ordered pairs  $ (I,J) $ formed by sequences of  numbers in $ \{ 0,1,2,\ldots, n-1 \} $
 of the same length $ k $ such that 
 $ I = (i_1, i_2 , \ldots, i_k ) $ is strictly increasing, such that $ J = (j_1, j_2 , \ldots, j_k ) $ 
 is strictly increasing too, except that there may be repetitions of $ 0$, and such that $ j_s \le i_s $ for
 all $1\leq  s \leq k $. For such pairs we define 
 \begin{equation}
\UU_{ I J} := \UU_{ i_1  j_1} \UU_{ i_2  j_2} \cdots \UU_{ i_k j_k}.
 \end{equation}
 A monomial of this form is called  \emph{normal}. We denote by $\mathcal{NM}_{n}$ the set formed by all normal monomials in $\mathbb{NB}_n$ together with $1$. For
 {\color{black}{{$n=2$}}}
 we have
 \begin{equation}
\mathcal{NM}_{1} =\{ 1, \UU_{ 0},  \UU_{ 1}, \UU_{ 1} \UU_{ 0}, \UU_{ 0} \UU_{ 1}, \UU_{ 0} \UU_{ 1} \UU_{ 0} \},  
 \end{equation}
 whereas for 
 {\color{black}{{$n=3$}}}
 \begin{equation}
 \begin{array}{l}
 \mathcal{NM}_{2} =\{  1, \UU_{ 0}, \UU_{ 1} \UU_{ 0} , \UU_{ 1} ,  \UU_{ 2}  \UU_{ 1} \UU_{ 0} , \UU_{ 2} \UU_{ 1} , \UU_{ 2} ,
   \UU_{ 0} \UU_{ 1}  \UU_{ 0} , \UU_{ 0} \UU_{ 1},
   \UU_{ 0} \UU_{ 2} \UU_{ 1} \UU_{ 0},  \UU_{ 0} \UU_{ 2} \UU_{ 1}, \UU_{ 0} \UU_{ 2} , 
   \UU_{ 1} \UU_{ 0} \UU_{ 2} \UU_{ 1} \UU_{ 0},  \\ \qquad \qquad   \UU_{ 1} \UU_{ 0} \UU_{ 2} \UU_{ 1} , 
   \UU_{ 1} \UU_{ 0} \UU_{ 2}, \UU_{ 1} \UU_{ 2} , 
      \UU_{ 0} \UU_{ 1} \UU_{ 0} \UU_{ 2} \UU_{ 1} \UU_{ 0},
   \UU_{ 0} \UU_{ 1} \UU_{ 0} \UU_{ 2} \UU_{ 1},
         \UU_{ 0} \UU_{ 1} \UU_{ 0} \UU_{ 2},  \UU_{ 0} \UU_{ 1}  \UU_{ 2} \}.
 \end{array}
 \end{equation} 
In general, using the relations given in Definition \ref{definition blob} one easily checks that
$\mathcal{NM}_{n}$ spans $ \mathbb{NB}_n $.
Indeed, we have that {\color{black}{{$ \{ \UU_0, \UU_1, \ldots, \UU_{n-1} \} \subseteq \mathcal{NM}_{n}$}}}
and that
any product of the form $ \UU_i \UU_{ I J}$
can be written as a linear combination of elements of  $\mathcal{NM}_{n}$.
On the other hand, the set    $\mathcal{NM}_{n}$ is in bijection with the set of positive fully commutative
 elements of the Coxeter group of type $ B_n $. In particular, the cardinality of  $\mathcal{NM}_{n}$ is known to be 
$ \binom{2n}{n} $, see for example \cite{AlHar-Gon-Pl}. Hence we deduce that 
\begin{equation}
\dim  \mathbb{NB}_n  \le \dim  \mathbb{NB}_n^{diag}  
\end{equation}  
since $ \dim \mathbb{NB}_n^{diag}  =  \dim \mathbb{B}_n = \binom{2n}{n}$.
Thus, in order to show the Theorem we must check that $ \varphi $ is surjective, or equivalently that
the diagrams in \ref{isomorphism} generate $ \mathbb{NB}_n^{diag}$.

\medskip
Let us first focus on the `Temperley-Lieb part' of $ \mathbb{NB}_n^{diag}$, that is
the subalgebra of $ \mathbb{NB}_n^{diag}$ consisting of the linear combinations of 
Temperley-Lieb diagrams, the unmarked diagrams from $ \mathbb{NB}_n^{diag}$.
There is a concrete algorithm for obtaining any
Temperley-Lieb diagram as a product of the $ \varphi(\UU_i ) $'s, where $i > 0$,
and so these diagrams generate the subalgebra. Although it is well known, we 
still explain how it works since we need a small variation of it. 

\medskip
In the following, whenever $  \UU  \in \mathbb{NB}_n $ we shall often write
$  \UU \in  \mathbb{NB}_n^{diag}$ for $ \varphi(\UU ) $. This should not cause confusion.

\medskip

Let $ D $ be a Temperley-Lieb diagram on $ n  $ points with $ l $ through lines and let
$ k = (n-l)/2$. We 
associate with $ D $ two standard tableaux $ top(D) $ and $ bot(D) $ of shape $ \lambda = (1^{l+k}, 1^k) $
as follows. For $ top(D) $ we go through the upper points of $ D $, placing $ 1 $ in position $ (1,1) $ of 
$ top(D) $, then $ 2 $ in position $ (1,2) $ if $ 2$ is the right end point of a horizontal arc, otherwise in
position $ (2,1)$, and so on recursively. Thus, having placed
$ 1, 2 \ldots, i-1 $ in $ top(D) $ we place 
$ i $ in the first vacant position of the second column if 
$ i $ is the right end point of a horizontal arc, otherwise in the first vacant position of the first column.
The standard tableau $  bot(D) $ is constructed the same way, using the bottom points of $ D$.
{\color{black}{An example of
    this process is illustrated in the Figures \ref{fig:blobD} and \ref{fig:tableauxex1}.}}

\begin{figure}[h]
  \centering
   $D:=$ \, \, \, \,  \raisebox{-.5\height}{\includegraphics{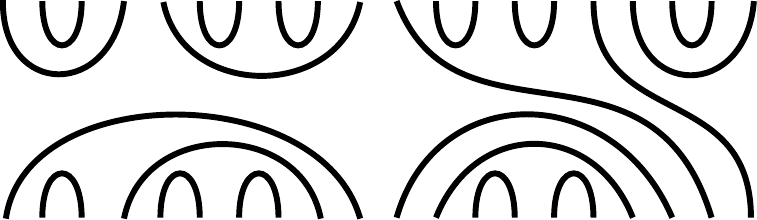}}
  \caption{A Temperley-Lieb diagram. }
\label{fig:blobD}
\end{figure}
\noindent

\begin{figure}[h]
  \centering
 $ top(D)$\, =\, \gyoung(;1;3,;2;4,;5;7,;6;9,;8;<10>,;\once;<13>,;<12>;<15>,;<14>;<19>,;<16>;<20>,;<17>,;<18>)\, 
\, \, \, \, \, \, \, \, \, \, \, \,
    $bot(D)$\, =\, \gyoung(;1;3,;2;6,;4;8,;5;9,;7;<10>,;\once;<14>,;<12>;<16>,;<13>;<17>,;<15>;<18>,;<19>,;<20>)
    \caption{The tableaux $   top(D) $ and $   bot(D)$ associated with the diagram $D$ in Figure \ref{fig:blobD}.}
\label{fig:tableauxex1}
\end{figure}

\medskip
It is well known, and easy to see, that the map $ D \mapsto (top(D), bot(D)) $ is a bijection between
Temperley-Lieb diagrams and pairs of two column standard tableaux of the same shape.

\medskip
For $ \bT $ any Young tableau and $ 1 \le k \le n $  we define $ \bT \! \mid_k$ as the restriction of $\bT$
 to the set $ \{1,2, \ldots, k \}$. 
We may then consider a two-column standard tableaux $ \bT $  
as a sequence of pairs $ (i, {\rm{diff}}( \bT \! \mid_i) ) $ for $ i=0,1,2 \ldots, n $, 
where $ {\rm{diff}}( \bT \! \mid_i) $ is the difference between the lengths of the first and the second column of the
underlying shape of $  \bT \! \mid_i $ (here $ i=0$ corresponds to the pair $ (0,0)$).
We then plot these pairs in a coordinate system, using matrix convention for
the coordinates.

This may be viewed as a walk in this coordinate system, where at level $ i $ we step once to the left
if $ i+1 $ is in the second column of $ \bT $ and otherwise once to the right. In Figure
\ref{fig:two walks}
we have indicated the 
corresponding walks for $ top(D) $ and $ bot(D) $ where $ D $ is as above in Figure \ref{fig:blobD}. 

\medskip

A Temperley-Lieb diagram $ D $ is given uniquely by $ (top(D), bot(D)) $ and so we  
introduce{\color{black}{ {\it half-diagrams} $ T(D) $ and $ B(D) $ corresponding to $ top(D) $ and 
$  bot(D) $.}}
For example the half-diagrams $ T(D) $ and $ B(D) $ for
{\color{black}{$ D $ in Figure \ref{fig:blobD} are given below in Figure \ref{fig:half diagramsA}}}

\begin{figure}[h]
  \centering
 \! \! \! \! \! \! \! 
$   top(D)  =  \, \, \,\, $\raisebox{-.5\height}{\includegraphics[scale=0.6]{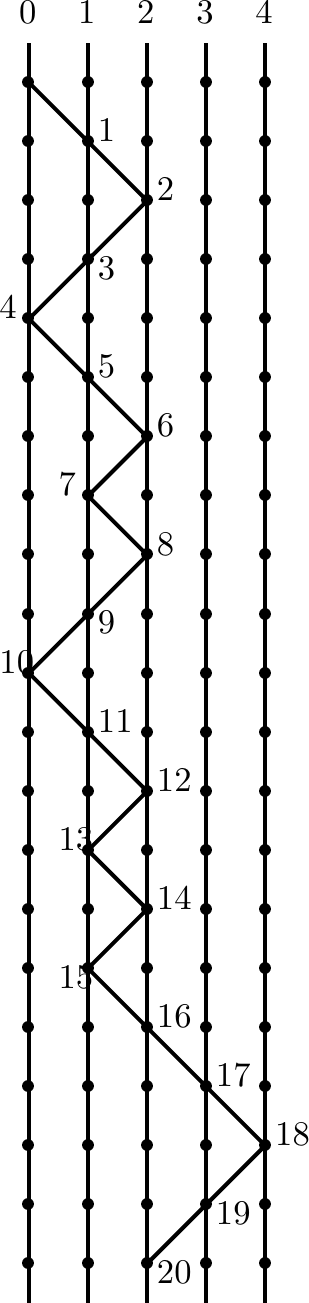}}
 \, \, \, \! \! \! \! 
$  bot(D)  = \, \, $ \raisebox{-.5\height}{\includegraphics[scale=0.6]{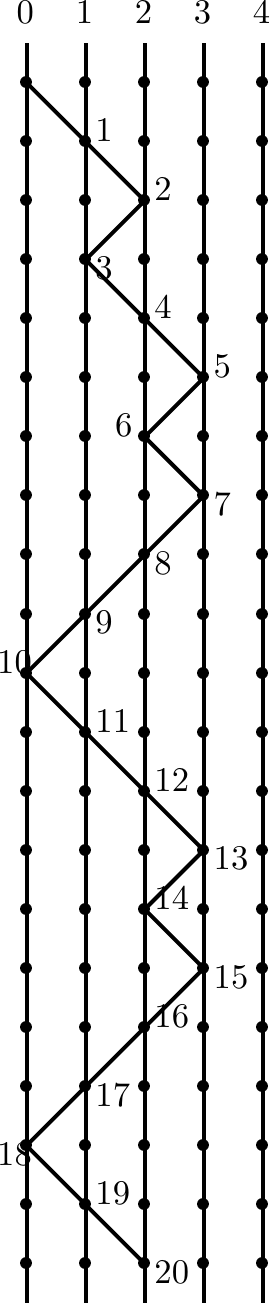}}
  \caption{The two walks associated with $D$ of Figure \ref{fig:blobD}. }
\label{fig:two walks}
\end{figure}


\vspace{4cm}


\begin{figure}[h]
  \centering
  $T(D)=\, \,  $\raisebox{-.5\height}{\includegraphics[scale=0.7]{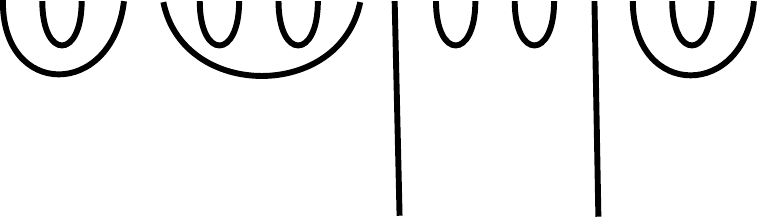}}  \, \,\, \, \, \, \, \,\, \,
  $ B(D) = \, \, \, \,  $\raisebox{-.5\height}{\includegraphics[scale=0.7]{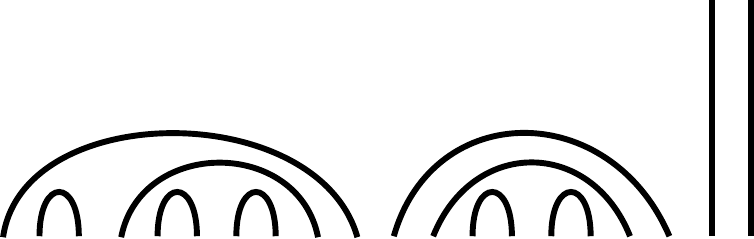}}
  \caption{The diagrams $ T(D) $ and $ B(D) $ for $D$ as in Figure \ref{fig:blobD}.}
\label{fig:half diagramsA}
\end{figure}
\noindent

\vspace{0.5cm}
Recall that for any two column partition $ \lambda $ there is unique maximal $ \lambda$-tableau $ \bT^{\lambda} $
under the dominance order. It is constructed as the row reading of $ \lambda$. For example, for 
$ \lambda = (1^{11},1^9 ) $ we give in {\color{black}{Figure \ref{fig:tableauxex} the tableau
$ \bT^{\lambda} $ and its corresponding bottom half-diagram.}}


\begin{figure}[h]
  \centering
$  \bT^{\lambda}\, =\,\, $\gyoung(;1;2,;3;4,;5;6,;7;8,;9;<10>,;\once;<12>,;<13>;<14>,;<15>;<16>,;<17>;<18>,;<19>,;<20>) \, , \, \, \, 
$     B^{\lambda} = $ \, \, \, \,       \raisebox{-.5\height}{\includegraphics[scale=1]{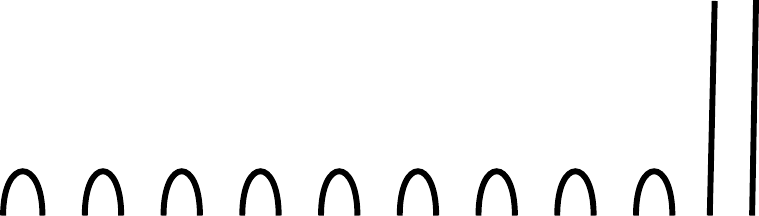}}
  \caption{The tableau $ \bt^{\lambda}$ and its associated bottom half-diagram for $ \lambda = (1^{11}, 1^9)$.}
\label{fig:tableauxex}
\end{figure}

\medskip
{\color{black}{For the same $ \lambda $}}, the walk corresponding to $ \bT^{\lambda}$ is 
indicated {\color{black}{three times in Figure \ref{fig:walkisalindicated}
where in the middle and on the right 
we have colored it red}} and have combined it with
the walks for $ top(D) $ and $bot(D) $ coming from Figure \ref{fig:two walks}. 

\begin{figure}[h]
  \centering
  \raisebox{-.5\height}{\includegraphics[scale=0.5]{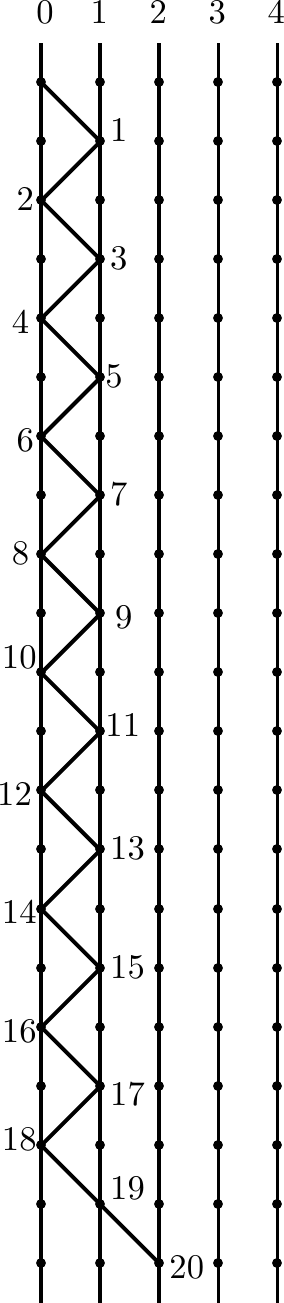}}
\, \, \, \, \, \,\, \, \, \,\, \,\, \, \, \,
\raisebox{-.5\height}{\includegraphics[scale=0.5]{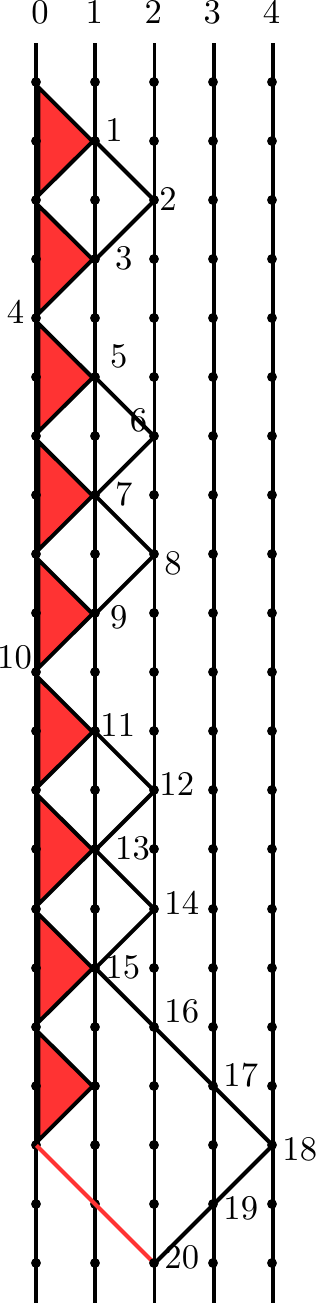}}
\, \, \, \, \, \,\, \, \, \, \,\, \, \, \,
\raisebox{-.5\height}{\includegraphics[scale=0.5]{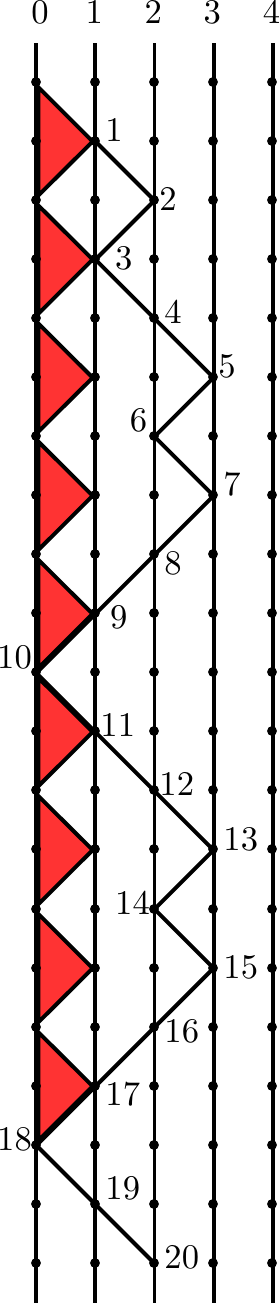}}
\caption{For $ \lambda = (1^{11}, 1^9) $ we give on the left the walk for $ \bt^{\lambda}$ and
  include in the middle and
  on the right the walks for $ top(D) $ and $ bot(D) $, coming from Figure \ref{fig:two walks}.}
\label{fig:walkisalindicated}
\end{figure}

\medskip
The algorithm for generating the Temperley-Lieb diagrams consists now in filling in
the area between the walks for $ \bT^{\lambda}$ and $ bot(D) $ (resp. $top(D) $) one column at the time,
and then multiplying with the corresponding $ \UU_i $'s. 
For example, using Figure \ref{fig:twocolorfill},
\begin{figure}[h]
  \centering
    \raisebox{-.5\height}{\includegraphics[scale=0.5]{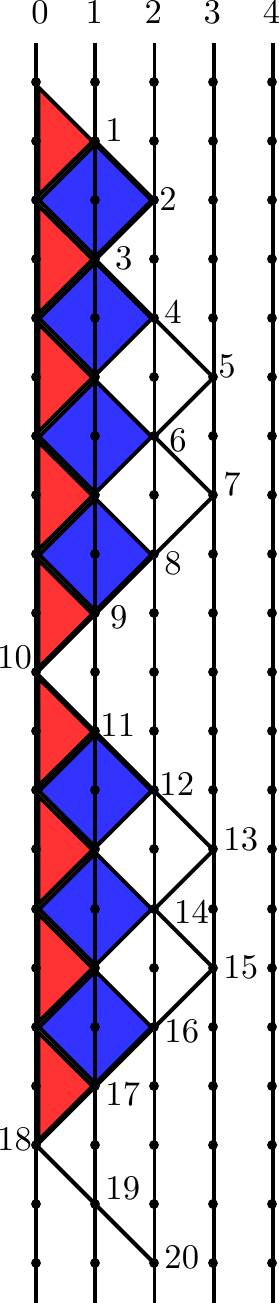}}
\, \, \, \, \, \,\, \, \, \, \, \, \, \, \,\, \, \,
\raisebox{-.5\height}{\includegraphics[scale=0.5]{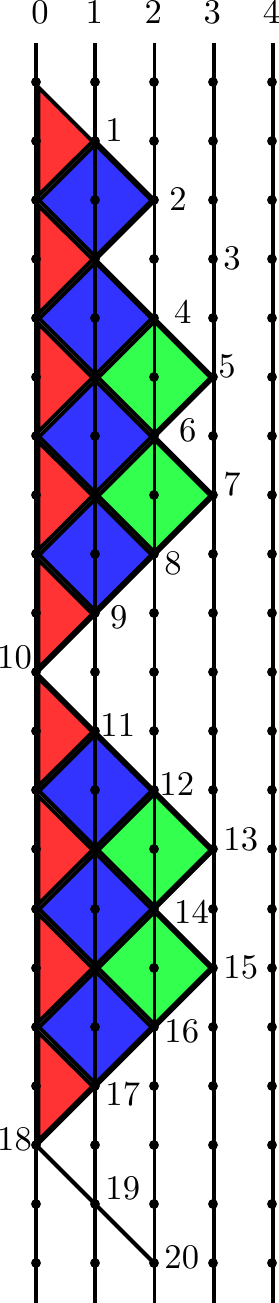}}
  \caption{The algorithm illustrated on $ bot(D) $ as in Figure \ref{fig:walkisalindicated}. }
\label{fig:twocolorfill}
\end{figure}
we find that to obtain $ bot(D) $
from the walk for $ \bT^{\lambda} $ we should first multiply by $ \UU_2 \UU_4 \UU_6 \UU_8 \UU_{12} \UU_{14} \UU_{16}$
corresponding to the blue area, and then with $ \UU_5 \UU_7 \UU_{13} \UU_{15}$, corresponding to the green area, 
that is we have that 
\begin{equation}
B(D) =    B^{\lambda}  ( \UU_2 \UU_4 \UU_6 \UU_8 \UU_{12} \UU_{14} \UU_{16}) (\UU_5 \UU_7 \UU_{13} \UU_{15}) 
\end{equation}
where $ B(D) $ is the half-diagram in Figure \ref{fig:half diagramsA} and $ B^{\lambda} $ is
the diagram defined in Figure {\ref{fig:tableauxex}}.
Similarly, we have that
\begin{equation}
  T(D)
  =  \UU_{18} ( \UU_{17} \UU_{19} )    ( \UU_2 \UU_6 \UU_8 \UU_{12} \UU_{14} \UU_{16} \UU_{18}      ) T^{\lambda}    
\end{equation}
where $ T(D) $ is the half-diagram in Figure \ref{fig:half diagramsA} and $ T^{\lambda} $ is
the reflection through a horizontal axis of $ B^{\lambda} $.
Since $T^{\lambda} B^{\lambda}   = \UU_{1} \UU_{3} \UU_{5} \UU_{7} \UU_{9} \UU_{11} \UU_{13} \UU_{15} \UU_{17} $ 
we get now $ D $ as a product of $ \UU_{i}$'s: 
\begin{equation}
  D = T(D) B(D) =
  \UU_{18} ( \UU_{17} \UU_{19} )    ( \UU_2 \UU_6 \UU_8 \UU_{12} \UU_{14} \UU_{16} \UU_{18}      )
T^{\lambda} B^{\lambda}
   ( \UU_2 \UU_4 \UU_6 \UU_8 \UU_{12} \UU_{14} \UU_{16}) (\UU_5 \UU_7 \UU_{13} \UU_{15}) .
\end{equation}

Summing up, we have shown that any unmarked blob diagram can be obtained as a product of the generators $\UU_i$'s, for $i>0$.

\medskip

We now explain how to obtain the marks on the arcs. In the case of 
$ B $ as before there are three arcs that may carry a mark, namely the black arcs below

\begin{equation}\label{half diagrams}
\raisebox{-.5\height}{\includegraphics[scale=0.8]{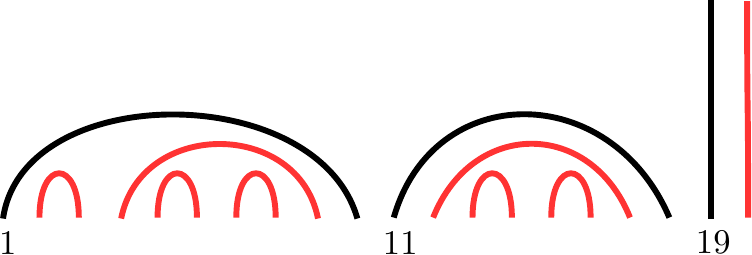}}
\end{equation}


A main general observation for what follows is  that these arcs are in correspondence with
the `contacts' between the associated walk and the vertical $0$-line. To be precise
for $ i =0,1,\ldots, n-1 $ we have that $ (i, 0) $ belongs to the walk for $ B $ if and only if
$ i+1 $ is the leftmost point of an arc that may be marked.
For instance, using the walk in Figure \ref{fig:twocolorfill} for the above $ B $ we see that these points are
$1, 11 $  and $ 19$, as one indeed observes in \ref{half diagrams}.

These contacts points induce a partition of the indices $ 1 \le i \le n$
in {\color{black}{subsets that we call {\it contact intervals}}.}
Thus in the example given in Figure \ref{fig:twocolorfill}, the first contact interval consists of the indices 
$ 1 \le i \le 10$, the second of $ 11 \le i \le 18$ and the third of $ 19$ and $20$. We stress that the smallest number in each contact interval is odd. 
On the other hand, under the above process of filling in the areas, the $ \UU_i $'s, where $ i $ corresponds to
the rightmost index of some contact interval, are not needed. But from this we deduce that
the indices corresponding to distinct contact intervals give
rise to commuting $ \UU_i$'s and hence we can in fact fill in one contact interval at the time. We choose to do so
going through the contact interval of each walk from bottom to the top.

\medskip
Our second observation is that any diagram of the form

\begin{equation}\label{oursecondobservation}
\raisebox{-.5\height}{\includegraphics[scale=0.8]{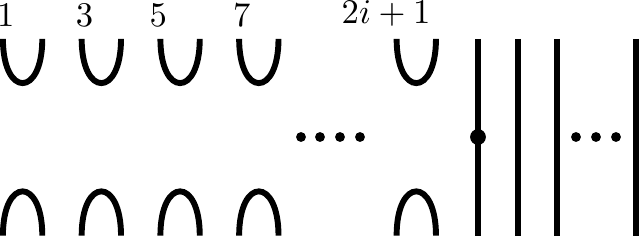}}
\end{equation}


\medskip
\noindent
can be generated by the $ \UU_i$'s since indeed it is equal to
\begin{equation}\label{sinceindeedblabla}
(\UU_1 \UU_3 \UU_5 \cdots \UU_{2i+1}) \UU_0 (\UU_2 \UU_4 \UU_6 \cdots \UU_{2i+2})( \UU_1 \UU_3 \UU_5 \cdots \UU_{2i+1}).
\end{equation}
In Figure \ref{fig:examplei=2andn=9} we give the case $ i = 2 $ and $ n = 9$.

\begin{figure}[h]
  \centering
\raisebox{-.5\height}{\includegraphics[scale=0.7]{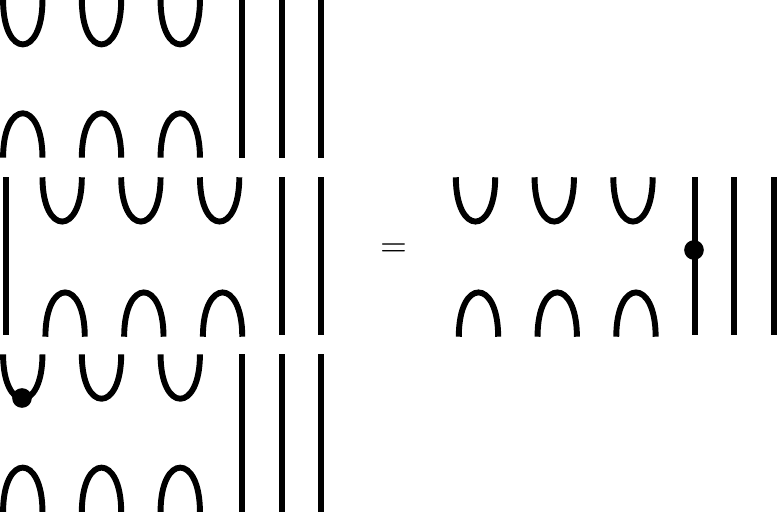}}   
\, \, \, \, \, \, \, 
\caption{Ilustration of equation \ref{sinceindeedblabla} for $ n=9$ and $i=2 $.  }
\label{fig:examplei=2andn=9}
\end{figure}

\medskip 

The algorithm for obtaining any marked diagram now consists in filling in the contact intervals,
from bottom to top, and
multiplying by a diagram of the form given in \ref{oursecondobservation}, for each contact
interval that requires a mark.
Let us illustrate a few step of it on the blob diagram given in Figure \ref{fig:an example with $ n=20$}.
Its bottom and top halves are given in Figure \ref{fig:half diagramsA}. Both of them have three contact
intervals.
The third contact interval is $ \{ 11,12,\ldots, 20\} $ for the top diagram and, as we have already seen, $ \{ 19,20 \} $
for the bottom diagram. Multiplying with the corresponding $ \UU_i$'s on $ T^{\lambda}B^{\lambda}  $ we get
the diagram

\begin{equation}
   \raisebox{-.5\height}{\includegraphics[scale=0.9]{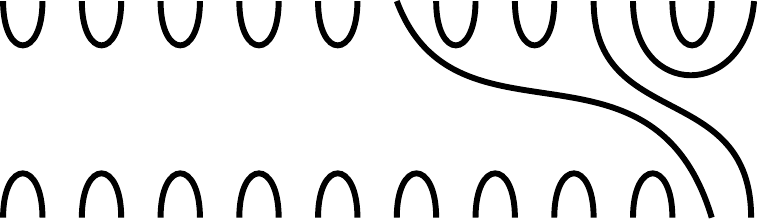}}   
\end{equation}

\medskip

Suppose now that we want to produce the blob diagram from Figure \ref{fig:an example with $ n=20$}.
Then we need a mark on the first through line and thus we multiply below with a diagram of the form 
\ref{oursecondobservation} with $ i = 8$ which gives us

\begin{equation}
      \raisebox{-.5\height}{\includegraphics[scale=0.9]{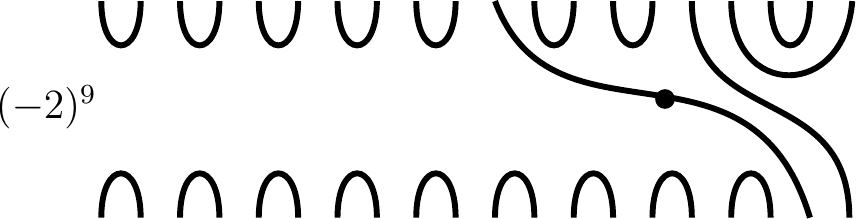}}   
\end{equation}

\noindent
settling the third contact interval, at least up to a unit in $ \F$. The algorithm now goes on with the second contact interval, etc.
The Theorem is proved.
\end{proof}

In view of the Theorem \ref{blob diagram realization} we shall write $  \mathbb{NB}_n   = \mathbb{NB}_n^{diag} $. Similarly
we shall in general write $ \UU  $ for $ \varphi(\UU) $. 

\medskip
The next  corollary is an immediate consequence of Theorem \ref{blob diagram realization}.

\begin{corollary}\label{cor dim}
The set $\mathcal{NM}_n$ is a basis for $ \mathbb{NB}_n $.  Similarly, the set 
\begin{equation}
	\widetilde{\mathcal{NM}}_n :=  \{ X\mathbb{J}_n^i \, |\,  X \in \mathcal{NM}_n, \, i\in \{0,1\} \}
\end{equation}	
is a basis for $\widetilde{\mathbb{NB}}_n$. Consequently,  $\dim  \widetilde{\mathbb{NB}}_n = 2 \binom{2n}{n}$.\\ We refer to the set $\mathcal{NM}_n$ (resp. $\widetilde{\mathcal{NM}}_n $)  as the normal basis of  $ \mathbb{NB}_n $ (resp. $\widetilde{\mathbb{NB}}_n$). 
 \end{corollary}

\begin{corollary}\label{cellularfirst}
  $ \mathbb{NB}_n $ is a cellular algebra {\color{black}{{in the sense of Graham and Lehrer, see \cite{GL}}},
with the same cellular datum as for $\mathbb{B}_n$, see for example \cite{GL1,PlazaRyom} for this cellular structure}.
\end{corollary}

\begin{proof}
In view of the diagrammatic description, given right after Lemma \ref{referenceLemma}, 
  of the multiplication in $ \mathbb{NB}_n $, 
this is essentially the same as the proof of cellularity of $\mathbb{B}_n$. 
  To be more precise, with a blob diagram $ D$, we first associate the number of through arcs of 
$D$, that is the number of arcs going
  from the top to the bottom of the diagram.
  The blob diagram in Figure \ref{fig:an example with $ n=20$} 
  has for example two through arcs.
  We would next also like to associate with $ D $ a top and a bottom blob half-diagram,
  $ T(D) $ and $ B(T)$ which should give rise to
  a bijection $ D \mapsto (T(D), B(D))$, as in the case of Temperley-Lieb diagrams. But
here we encounter the problem
  that the leftmost through arc may be blobbed, which makes it unclear whether the corresponding blob should
  belong to $ T(D) $ or to $ B(T)$.

  We resolve this problem as follows. We first consider only those blob diagrams whose leftmost through
  arc either does not exist or is unmarked. For each such diagram $ D $ there is no problem in
  defining $ T(D)$ and $  B(T)$ and we consider 
  the corresponding 
  map $\varphi_1: D \mapsto (T(D), B(D))$. 
For $ k $ a non-negative integer, we 
  define correspondingly $ \tab(k) := \{ T(D) \mid D \mbox{ has $k$ through arcs} \}$.
  We next consider the blob diagrams $ D $ that have through arcs such that the leftmost one of these is
  marked.
  For these $ D $ we first remove the mark on the through arc and next apply $\varphi_1$.
  This gives a map 
  $\varphi_2: D \mapsto (T(D), B(D))$ and for $ k $ a negative integer we define correspondingly
$ \tab(k) := \{ T(D) \mid D \mbox{ has $-k$ through arcs} \}$. With this notation we now have the following 
  description of our basis for $ \mathbb{NB}_n $. 
 \begin{equation}
\mathcal{NM}_n =\{ D \iota(D^{\prime}) \mid D \in \tab(k), D^{\prime} \in \tab(k), k\in \{n, n-2, \ldots, -n \}\}
 \end{equation}
 where $ \iota$ is the reflection along a horizontal axis.

 \medskip We define an order relation on $ \mathbb Z $ via
 $ k \prec l $ if $ | k | < |l | $ of if $ | k | = |l | $ and $ k < l $.
 Suppose now that $ D \iota(D^{\prime}) \in \mathcal{NM}_n $.
 Then if follows from the diagrammatic description of 
 the multiplication in $ \mathbb{NB}_n $
 that $ \UU_i D \iota(D^{\prime}) $ and $ D \iota(D^{\prime}) \UU_i$ are
 linear combinations of $\{ D \iota(D^{\prime}) \mid D, D^{\prime} \in \tab(l) $ where $ l \preceq k \}$.
 Moreover, it also follows from that description that 
if the expansion of $ \UU_i D \iota(D^{\prime}) $ has a nonzero term of the form 
$ D_1 \iota(D_1^{\prime}) $ where $ D_1, D_1^{\prime} \in  \tab(k) $ then
the expansion of $ \UU_i D \iota(D^{\prime}) $ in fact only involves terms of the form 
$ D_2 \iota(D^{\prime}) $; in other words with the same bottom half-diagram as $ D \iota(D^{\prime})$,
and similarly for $  D \iota(D^{\prime}) \UU_i$.
But these two statements amount to $\mathcal{NM}_n $ being a cellular basis for $\mathbb{NB}_n $, 
see Definition 1.1 of
\cite{GL}.
\end{proof}

In the next section we shall show that there is another cellular
structure on $ \mathbb{NB}_n $, given by Soergel calculus, 
That cellular structure is endowed with
a family of JM-elements, in the sense of \cite{Mat-So}:

\begin{definition} 
We define the JM-elements $ \LL_1, \LL_2, \ldots , \LL_{n} $ of $ \mathbb{NB}_n $ via
$ \LL_1 = \UU_0 $ and recursively
\begin{equation}
\LL_{i+1} = \UU_{i}  \LL_{i}  + \LL_{i} \UU_{i }  -2 \UU_{i} \sum_{j=1}^{i-1} \LL_{j }, \, \,  i \geq 1.
\end{equation}
\end{definition}

\begin{lemma}\label{JMB}
{\color{black}{The $\LL_{i}$'s have the following properties. }}
  \begin{description}
	\item[a)] $\color{black}{ \LL_{i} \LL_{j} = \LL_{j} \LL_{i}}$ for all $ i, j $. 
       	\item[b)] $\LL_1^2=0$ and that $ \LL_{i}^2 = -2 \LL_{i} \sum_{j=1}^{i-1} \LL_{j}  $ for all $ 1<i\leq n $. 
  \end{description}
  \end{lemma}  
\begin{proof}
{\color{black}{We give the proof in Remark \ref{Jucys-Murphy-proof}.}}
\end{proof}  

\noindent
In Figure \ref{fig:firstJMelements} we give the JM-elements for $ n = 3$.

\begin{figure}[h]
 \raisebox{-.5\height}{\includegraphics[scale=0.9]{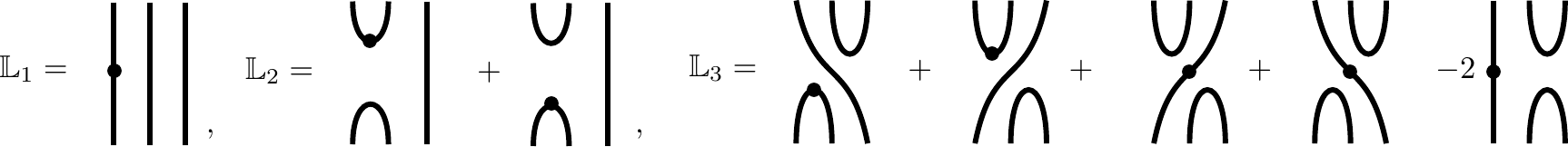}}   
\caption{The JM-elements $ \LL_i $ for $n=3$.}
\label{fig:firstJMelements}
\end{figure}

It follows from the results of section 5 of our paper that there is yet another cellular structure on 
$ \mathbb{NB}_n $, coming from the blob algebra. That cellular structure is also endowed with
a family of JM-elements: 

\begin{definition} 
We define the JM-elements $ \YY_1, \YY_2, \ldots , \YY_{n} $ of $ \mathbb{NB}_n $ via
$ \YY_1 = \UU_0 $ and recursively
\begin{equation}
\YY_{i+1} = \color{black}{(\UU_{i} +1) \YY_{i}}(\UU_{i} +1), \, \,  i \geq 1.
\end{equation}
\end{definition}

\noindent
In Figure \ref{fig:JMye}, we give these JM-elements for $ n = 3$.


\begin{lemma}\label{JMA}
{\color{black}{The $\YY_{i}$'s have the following properties. }}
  \begin{description}
	\item[a)]  $\color{black}{\YY_{i} \YY_{j} = \YY_{j} \YY_{i}}$ for all $ i, j $. 
        \item[b)]  $ \YY_{i}^2 = 0 $ for all $ i $.
  \end{description}          
\end{lemma}  
\begin{proof}
{\color{black}{We give the proof in Remark \ref{remarkconcerningJM}.}}
\end{proof}

\begin{figure}[H]
\begin{tabular}{l} 
  \raisebox{-.5\height}{\includegraphics[scale=0.9]{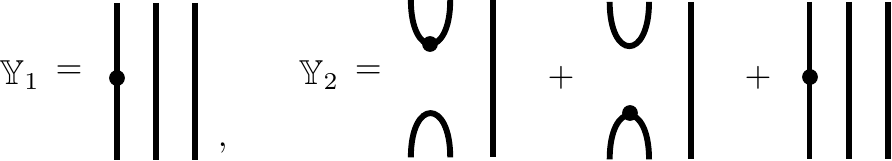}} \\ \\
\raisebox{-.5\height}{\includegraphics[scale=0.9]{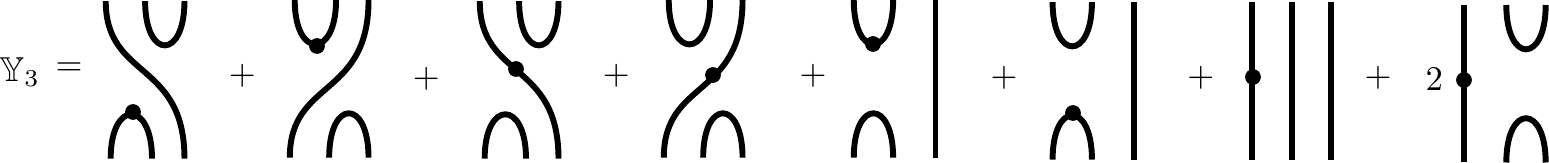}} 
\end{tabular}
\caption{The JM-elements $ \YY_i $ for $n=3$.}
\label{fig:JMye}
\end{figure}

As mentioned above, 
the $ \YY_{i}$'s are {\color{black}{(nilpotent)}} JM-elements 
{\color{black}{for $ \mathbb{NB}_n $ 
with respect to the cellular structure on $ \mathbb{NB}_n $ given in Corollary \ref{cellularfirst}}}.
Calculations for small $ n $ seem to indicate that
the various cellular structures on $ \mathbb{NB}_n $ are in fact equal. Of course, this does not
contradict the fact that the families of JM-elements are different, since there is no uniqueness 
statement for JM-elements.

\section{Soergel calculus for $\tilde{A}_1$.}

In this section, we start out by briefly recalling the diagrammatic Soergel category $ \mathcal D $ associated
with the affine Weyl group $ W $ of type $\tilde{A}_1$. This category $ \mathcal D  $ was introduced in \cite{EW},
in the complete generality of any Coxeter system $ (W, S) $.
The objects of $ \mathcal D $
are expressions $ \underline{w}$ over $ S $ and hence for any such $ \underline{w}$ we can
introduce an algebra
$  \tilde{A}_w : =\mbox{End}_{\mathcal D} (\underline{w}) $.
In the main result of this section we show that
$   \tilde{A}_w  $ and a natural subalgebra
$   {A}_w  \subset \tilde{A}_w $ of it are
isomorphic to the nil-blob algebras $\widetilde{\mathbb{NB}}_n$ and ${\mathbb{NB}}_n$
from the previous section.

\medskip
Let $ S  := \{ {\color{red} s}, {\color{blue} t} \} $ and let $ W $ be the Coxeter group on
$ S $ defined by 
\begin{equation} \label{presentation W}
W := \langle  {\color{red} s}, {\color{blue} t} \, |\,   {\color{red} s}^2 ={\color{blue} t}^2=e \rangle .
\end{equation}
Thus $ W $ is the infinite dihedral group or the affine Weyl group of type  $\tilde{A}_1$.
Given a non-negative integer $n$, we let
\begin{equation}\label{redexp}
  n_{ {\color{red} s}}:=  \underbrace{{\color{red} s} {\color{blue} t}
    {\color{red} s} \ldots }_{n  \scriptsize \mbox{-times}}  \qquad n_{ {\color{blue} t}}:=  \underbrace{ {\color{blue} t} {\color{red} s} {\color{blue} t}  \ldots }_{ n  \scriptsize \mbox{-times}}
 \end{equation}
with the conventions that $0_{ {\color{red} s}}:=0_{ {\color{blue} t}} :=e$.
It is easy to see from \ref{presentation W} that $  n_{ {\color{red} s}} $ and
$  n_{ {\color{blue} s}} $ are reduced expressions and that each 
element in $W$ is of the form $   n_{ {\color{red} s}} $ or $   n_{ {\color{blue} t}} $
{\color{black}{for a unique choice of $ n $ and $ {\color{red} s} $ or $ {\color{blue} t} $.
Note that 
the elements of $ W $ are \emph{rigid}, that is they have a unique reduced expression.}}

\medskip

The construction of  $\mathcal{D}$ depends on the choice of 
a \emph{realization} $\mathfrak{h}$ of $(W,S)$,
which by definition is a representation $ \mathfrak{h} $ of $ W $, with associated  
{\it roots} and {\it coroots}, see \cite[Section 3.1]{EW} for the precise definition.

In this paper, our $ \mathfrak{h} $ will be the  \emph{geometric representation}
of $W$ defined over $\mathbb{F}$, see  \cite[Section 5.3]{H1}. The coroots are the basis of 
$ \mathfrak{h} $, that is $\mathfrak{h}= \mathbb{F}\alpha^\vee_\ese \oplus \mathbb{F } \alpha^\vee_\te $ and
in terms of this basis the representation $\mathfrak{h} $ of $ W $ is given by  
\begin{equation}  \label{matrix realization}
\ese \rightarrow
\left(
\begin{array}{cc}
\! \! \! \! -1 & 2 \\
0  & 1
\end{array}
\right),
\qquad
\te \rightarrow
\left(
\begin{array}{cc}
1 & 0 \\
2 & \! \! \! \! -1
\end{array}
\right).
\end{equation}
The roots $\alpha_\ese , \alpha_\te \in \mathfrak{h}^\ast $ are now given by 
\begin{equation}  \label{equations realization}
\alpha_{\ese } ( \alpha^\vee_\ese )= 2,   \, \, \, \,  \alpha_\te (\alpha^\vee_\ese )= -2,     \, \, \, \,
\alpha_\ese (\alpha^\vee_\te ) = -2,  \, \, \, \,      \alpha_\te (\alpha^\vee_\te )= 2
\end{equation}
and so the Cartan matrix is
\begin{equation}
\left(
\begin{array}{cc}
2 & \! \! \! \!-2 \\
\! \! \! \!-2  & 2
\end{array}
\right).
\end{equation}
Note that we have 
\begin{equation} \label{roots sum A}
\alpha_{\ese} = -\alpha_\te .
\end{equation}

Let $R := S(\mathfrak{h}^\ast) = \oplus_{i\geq 0} S^i(\mathfrak{h}^\ast)$
be the symmetric algebra of $\mathfrak{h}^\ast$, or in view of \ref{roots sum A}
\begin{equation} \label{roots sum}
R = \F[\alpha_{\ese}] = \F[\alpha_{\te}].
\end{equation}
In other words, this is a just the usual one variable polynomial algebra. We consider it a $ \Z$-graded algebra 
by setting the degree of $ \alpha_{\ese} $ equal to 2.
Since $W $ acts on $\mathfrak{h} $
it also acts on $\mathfrak{h}^\ast$ and this action extends in a 
canonical way to $R$. 
We now introduce 
the \emph{Demazure operators} $\partial_\ese , \partial_\te :R \rightarrow R(-2)$ via 
\begin{equation}
\partial_\ese (f) = \frac{f-\ese f}{\alpha_\ese},  \qquad \qquad \partial_\te (f) = \frac{f-\te f}{\alpha_\te}.
\end{equation}
We have that
\begin{equation}
\ese \alpha_{\ese} = \alpha_{\te},  \, \, \,\te \alpha_{\te} = \alpha_{\ese}  \, \, \,
\end{equation}
and so we get 
\begin{equation}
  \partial_\ese ( \alpha_{\ese})= \partial_\te ( \alpha_{\te})= 2, \, \, \,
    \partial_\ese ( \alpha_{\te})= \partial_\te ( \alpha_{\ese})= -2 .\, \, \, 
\end{equation}

\medskip
We now come to the diagrammatic ingredients of $ \cal D$.
\begin{definition} 
  A Soergel graph for $ (W,S) $ is a finite and decorated graph 
  embedded in the planar strip $\mathbb{R}\times [0,1]$.
  The arcs of a Soergel graph are colored by $\ese $ and $\te $.
  The vertices of a Soergel graph are of two types as indicated below, univalent vertices (dots) and
trivalent vertices where all three incident arcs are of the same color. 
 \begin{equation}\label{Vertices}
\raisebox{-.5\height}{\includegraphics[scale=0.8]{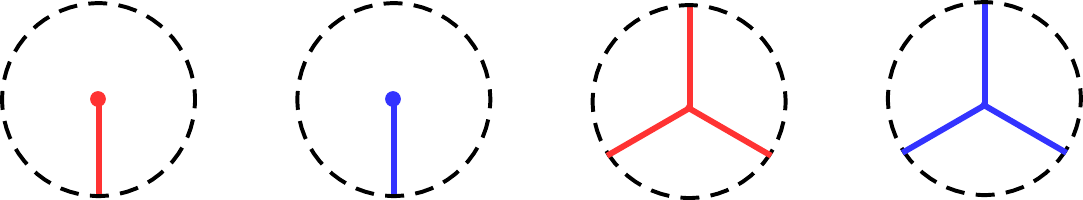}}   
 \end{equation}

\noindent
 A Soergel graph may have its regions, that is the connected components of the
complement of the graph in $\mathbb{R}\times [0,1] $,
decorated by elements of $R$. 
\end{definition}

In Figure \ref{fig:ejemploSoergel} we give an example of a Soergel graph. 
Shortly we shall give many more examples.

\begin{figure}[H]
  \centering
  \raisebox{-.5\height}{\includegraphics[scale=0.8]{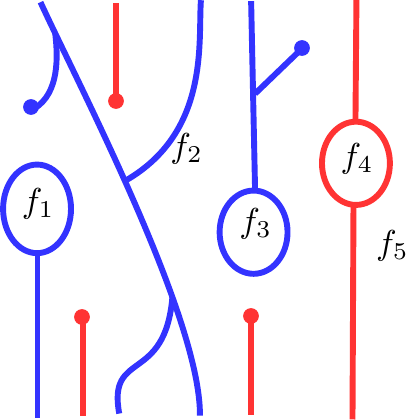}}   \, \, \, \, \, \, \, \, 
\caption{A Soergel graph for $(W,S)$. The $f_i$'s belong to $ R$.   }
\label{fig:ejemploSoergel}
\end{figure}

\medskip

We define
\begin{equation}
\Exp := \{ \underline{w} = (s_1, s_2 , \ldots, s_k)  \mid s_i  \in S, \,  k =1, 2, \ldots \}  \cup \emptyset .
\end{equation}
as the set of expressions over $ S$, that is words over the alphabet $ S$. 
The points where an arc of a Soergel graph intersects the boundary of the strip $\mathbb{R}\times [0,1] $
are called \emph{boundary points}. 
The boundary points provide two elements of $ \Exp$ 
called the \emph{bottom boundary} and \emph{top boundary}, respectively. 
In the above example the bottom boundary is $ ({\color{blue} t}, {\color{red} s},
{\color{blue} t}, {\color{blue} t}, {\color{red} s},{\color{red} s} ) $ and the top boundary is
$ ({\color{blue} t}, {\color{red} s},
{\color{blue} t}, {\color{blue} t}, {\color{red} s} ) $.

%

\begin{definition}  \label{defin endo BS}
  The diagrammatic Soergel category $\mathcal{D}$ is defined to be the
  monoidal category whose objects are the elements of $ \Exp$ and whose morphisms 
    $ \botts{x}{y} $ are the $\mathbb{F}$-vector space generated by all Soergel
  graphs with bottom boundary $\underline{x}$ and top boundary $\underline{y}$, modulo
isotopy and modulo the following local relations


\begin{equation} \label{sgraphA}
  \raisebox{-.5\height}{\includegraphics[scale=0.7]{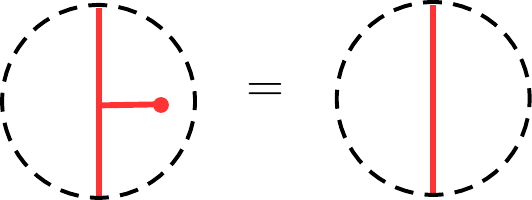}}   
\end{equation}

 \begin{equation}   \label{sgraphB}
 \raisebox{-.5\height}{\includegraphics[scale=0.7]{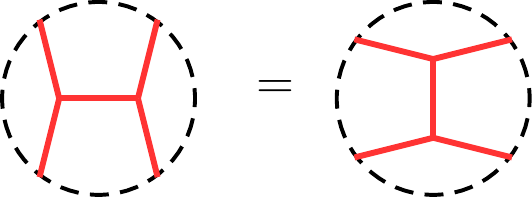}}   
 \end{equation}

\begin{equation}   \label{sgraphC}
 \raisebox{-.5\height}{\includegraphics[scale=0.7]{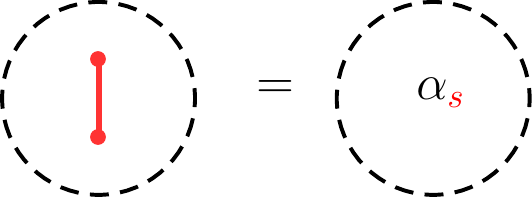}}
\end{equation}

\begin{equation}   \label{sgraphD}
    \raisebox{-.5\height}{\includegraphics[scale=0.7]{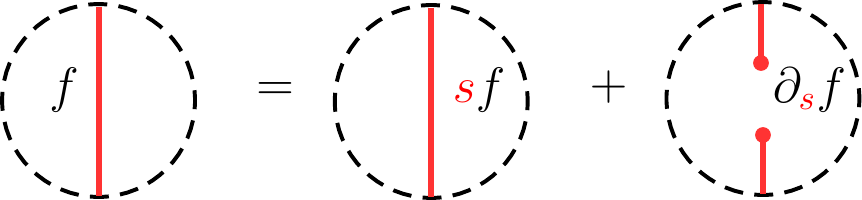}}
\end{equation}

 \begin{equation}   \label{sgraphE}
       \raisebox{-.5\height}{\includegraphics[scale=0.7]{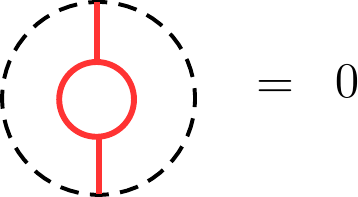}}
\end{equation}

There is a final relation saying that any Soergel graph $D$ which is
decorated in its leftmost region by an $f \in (\alpha_{\ese})$, that is a polynomial
with no constant term, is set equal to zero.
We depict it as follows

\begin{equation}  \label{rel muere a la izquierda}
         \raisebox{-.5\height}{\includegraphics[scale=0.7]{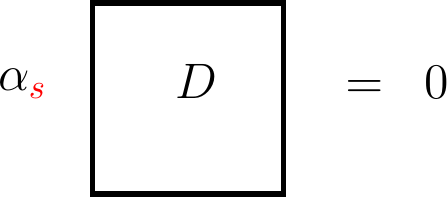}}
\end{equation}

\medskip
The relations  {\rm\ref{sgraphA}--\ref{rel muere a la izquierda}} also hold if red is replaced by blue,
of course.

\medskip
For $ \lambda \in \F $ and $ D $ a Soergel diagram, 
the scalar product $ \lambda D$ is identified 
with the multiplication by $ \lambda $ in any region of $ D $.
The multiplication {\color{black}{$ D_1 D_2 $}}
{\color{black}{of diagrams $ D_1 $ and $ D_2 $}}
is given by vertical concatenation 
{\color{black}{with $ D_1 $ on top of $ D_2$}}
and the monoidal structure by horizontal concatenation. There is natural $ \Z$-grading on
$ \cal D$, extending the
grading on $ R $, in which the dots, that is the first two diagrams in \ref{Vertices} have degree $ 1$,  and
the trivalents, that is 
the last two diagrams in \ref{Vertices}, have degree $-1$.
\end{definition}

\begin{remark}\rm
  Strictly speaking the category defined in Definition \ref{defin endo BS}
  is not the diagrammatic Soergel category introduced in \cite{EW}.
  To recover the category from \cite{EW} the relation \ref{rel muere a la izquierda} should
  be omitted. 
	
\end{remark}

Let us comment on the isotopy relation in Definition \ref{defin endo BS}. It follows from it that
the arcs of a Soergel graph may be assumed to be piecewise linear. It also follows from it together with \eqref{sgraphB} that the 
following relation holds

 \begin{equation}   \label{frob}
\raisebox{-.5\height}{\includegraphics[scale=0.7]{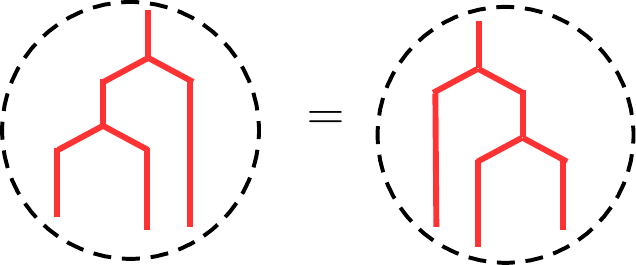}}
\end{equation}

\noindent
In other words the two trees on three downwards leaves are equal.
We also have equality for other trees.
Here is the case with four upwards leaves. Note the last
diagram which represents the way we shall often depict trees. 


\begin{equation}   \label{trees}
      \raisebox{-.5\height}{\includegraphics[scale=0.7]{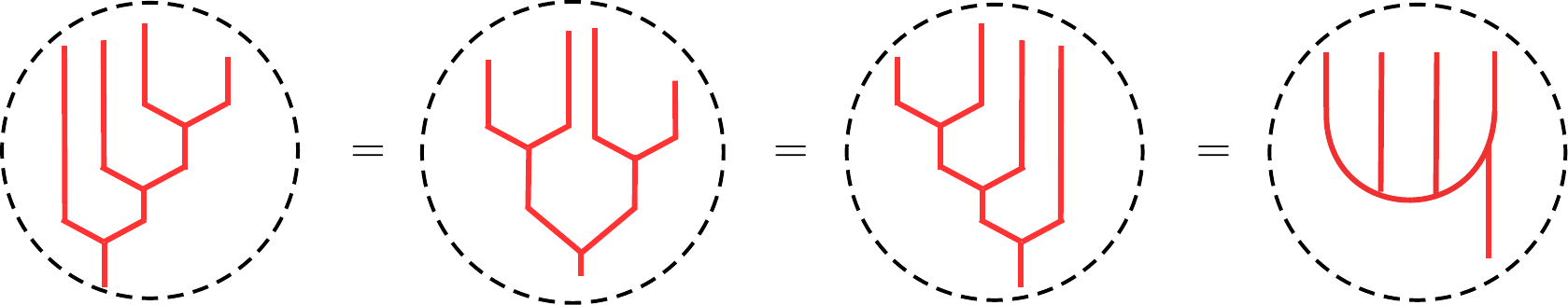}}
\end{equation}

\medskip

Let now $ n $ be a fixed positive integer and fix 
$ \underline{w}  := n_{ {\color{red} s}} \in \Exp $ as in \ref{redexp}. We then define
\begin{equation}
\tilde{A}_w :=   \mbox{End}_{\mathcal D} (\underline{w}).
\end{equation}
    {\color{black}{As mentioned above, $ w $ is a rigid element of $ W  $ and therefore we use the notation
 $ \tilde{A}_w  $ instead of $ \tilde{A}_{\underline{w}}$.}}         

\medskip
By construction, $ \tilde{A}_w $
      is an $ \F$-algebra with multiplication given by concatenation and the goal of this section is
to study the properties of this algebra.
First, for  $ i =1,\ldots, n-2 $
we define the following element of $ \tilde{A}_w $

\begin{equation}
     U_i := \quad     \raisebox{-.4\height}{\includegraphics[scale=0.9]{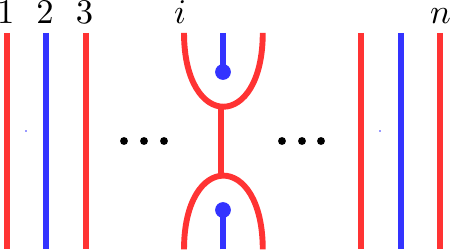}}
\end{equation}
\noindent
and similarly 
\begin{equation}
     U_0 := \quad     \raisebox{-.4\height}{\includegraphics[scale=0.9]{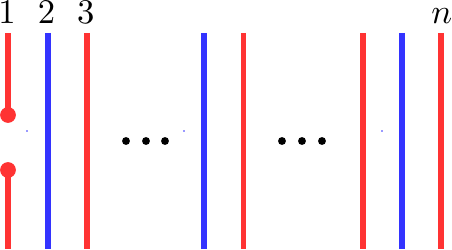}}
\end{equation}

\noindent
The following Theorem is fundamental for what follows. 
\begin{theorem}\label{universal}
 There is a homomorphism of $ \F$-algebras $ \varphi: \mathbb{NB}_{n-1} \rightarrow  \tilde{A}_w$ given 
    by $\mathbb{U}_i\mapsto U_i $ for $ i = 0, 1, \ldots, n-2$.
\end{theorem}
\begin{dem}
  We must check that $U_0, U_1, \ldots , U_{n-2} $ satisfy the relations
  given by the $ \mathbb{U}_i $'s in Definition \ref{definition blob}.
In order to show the quadratic relation \eqref{eq one} we argue as follows

\begin{equation}
  \, \, \, \, \, \, \, \, \, \, \, \, \, \, \,
\, \, \, \, \, \, \, \, \, \,
U_i^2 =
\raisebox{-.43\height}{\includegraphics[scale=0.6]{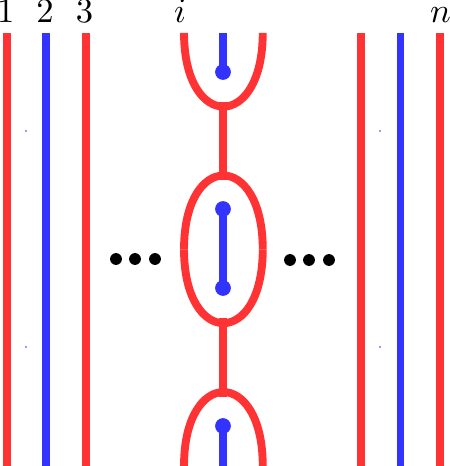}}
\, \, \,   =  \, \, \, 
\raisebox{-.43\height}{\includegraphics[scale=0.6]{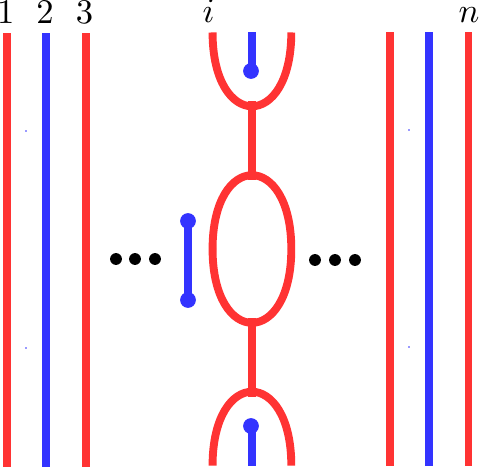}}
\, \, \, 
-2 \, \raisebox{-.43\height}{\includegraphics[scale=0.6]{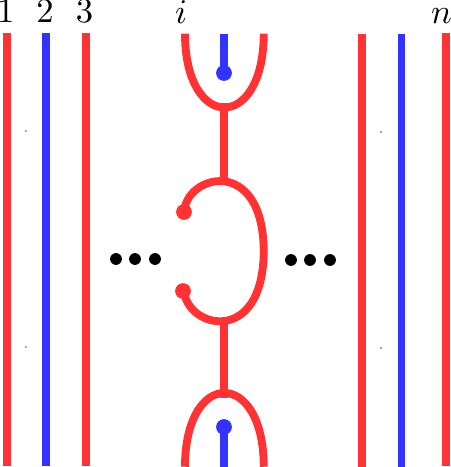}}
=  \,      -2U_i
 \end{equation}

\noindent 
where we used \eqref{sgraphA}, \eqref{sgraphC}, \eqref{sgraphD} and \eqref{sgraphE}.

We next show that \eqref{eq three} holds. 
If $|i-j|>2$ then \eqref{eq three} clearly holds, that is $ U_i U_j = U_j U_i $, 
but for $|i-j|=2$ it is not completely clear that it holds.
We shall only show it in the case $ n = 5 $, $ i= 1 $ and $ j =3 $: the 
general case is proved the same way. We have that

\begin{equation}\label{U1U3}
  U_3 U_1  = \, \, 
\raisebox{-.43\height}{\includegraphics[scale=0.6]{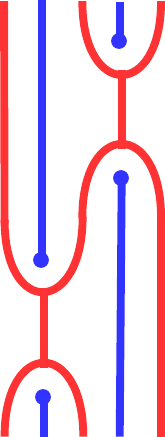}}
 \, \,  = \, \, 
\raisebox{-.43\height}{\includegraphics[scale=0.6]{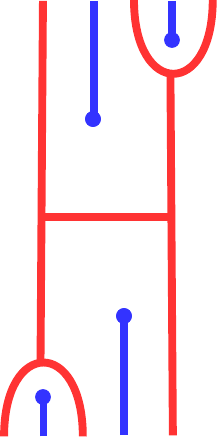}}
 \, \, = \, \, 
\raisebox{-.43\height}{\includegraphics[scale=0.6]{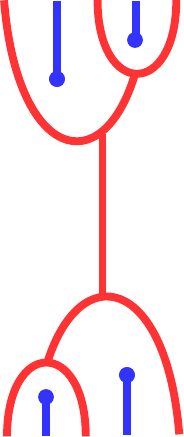}}
 \, \,  = \, \, 
\raisebox{-.43\height}{\includegraphics[scale=0.6]{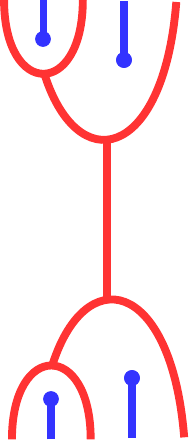}}  
\end{equation}

\noindent
where we used the `H'-relation \ref{sgraphB} for the third equality and \ref{frob} for the last equality.
But $ U_1 U_3 $ is obtained from $ U_3 U_1 $ by reflecting along a horizontal axis, and since
the last diagram of \eqref{U1U3} is symmetric along this axis, we conclude that $ U_1 U_3 = U_3 U_1 $ as
claimed.

The relation \eqref{eq two}, in the case $ n = 4, i = 1 $ and $ j = 2 $, is shown as follows.

\begin{equation}
U_1 U_2 U_1  = \, \, 
\raisebox{-.43\height}{\includegraphics[scale=0.6]{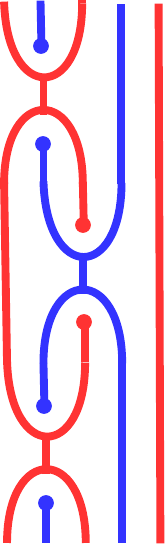}}  
 \, \,  =  \, \, 
\raisebox{-.43\height}{\includegraphics[scale=0.6]{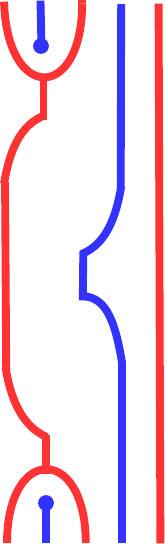}}  
  \, \, =  U_1
\end{equation}

The general case is treated the same way. We finally notice that \eqref{eq four} and \eqref{eq five}
  are a direct consequence of \eqref{rel muere a la izquierda}. The Theorem is proved.
\end{dem}

\medskip
For a general Coxeter system $ (W,S) $, Elias and Williamson 
found in \cite{EW} a recursive procedure for constructing
an $ \F$-basis for the morphism space $\botts{x}{y} $, for any
$ \underline{x}, \underline{y} \in \Exp$. It is a diagrammatic
version of Libedinsky's {\it double light leaves basis} for Soergel bimodules and the basis elements are
also called double light leaves 
in this case.
On the other hand we have fixed $ W $ as the infinite dihedral group, 
and in this particular case there is a non-recursive description of the double light leaves basis that we shall use.

\medskip

In order to describe it we first introduce some diagram conventions. 
First, in view of our tree conventions given in \ref{trees} we shall represent
the diagram from \ref{U1U3} as follows

\begin{equation}
U_1 U_3 =   \raisebox{-.43\height}{\includegraphics[scale=1]{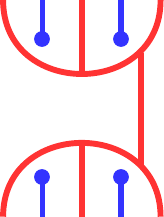}}  
\end{equation}

This can be generalized: for example using the last diagram in \ref{U1U3} we get that 

\medskip 
\begin{equation}
U_1 U_3 U_5 \, \, = \, \,
\raisebox{-.43\height}{\includegraphics[scale=0.7]{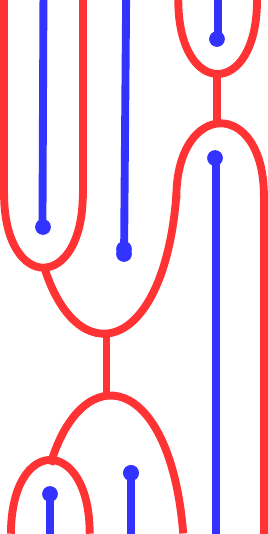}}  
\, \, \, \,= \, \, \,
\raisebox{-.43\height}{\includegraphics[scale=0.7]{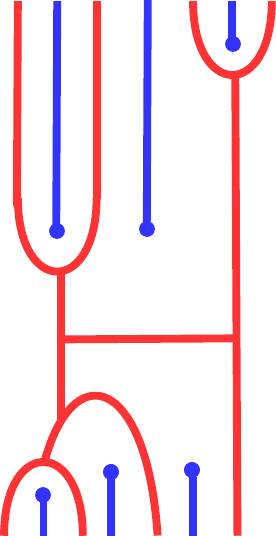}}  
\, \, =
\, \, \,
\raisebox{-.43\height}{\includegraphics[scale=0.7]{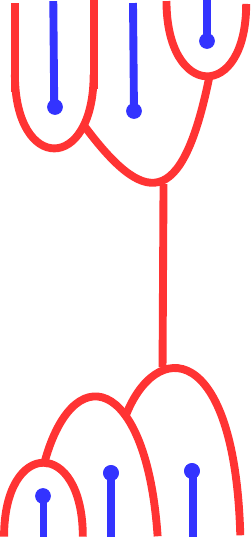}}  
=
\, \, \,
\raisebox{-.43\height}{\includegraphics[scale=0.7]{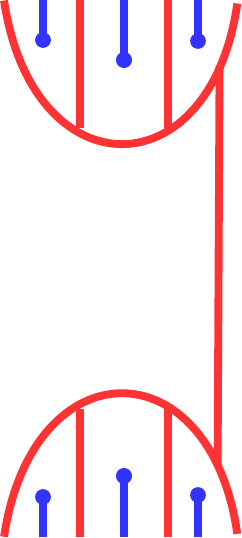}}  
\end{equation} 

\medskip \medskip \noindent
Even more generally, we have that

\vspace{-0.5cm}
\begin{equation}\label{evenmoregeneralA}
    U_i U_{i+2} \cdots U_{i+2k}
    = \, \, \,
\raisebox{-.43\height}{\includegraphics[scale=0.8]{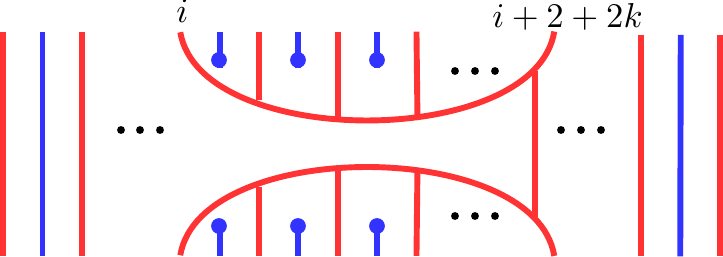}}      
\end{equation} 

\noindent
if $ i $ is odd and 

\begin{equation}\label{evenmoregeneralB}
    U_i U_{i+2} \cdots U_{i+2k}
    = \, \, \,
    \raisebox{-.43\height}{\includegraphics[scale=0.8]{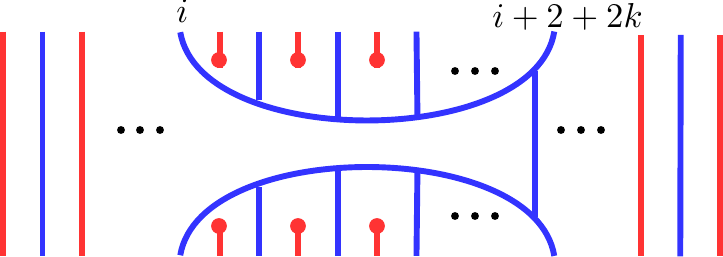}}      
\end{equation}

\noindent
if $ i $ is even. We now introduce a different kind of elements in $ \tilde{A}_w $, namely the 
\emph{JM-elements} $ L_i $ of $ \tilde{A}_w $, via
\begin{equation} 
  L_i :=
  \qquad
  \raisebox{-.43\height}{\includegraphics[scale=0.8]{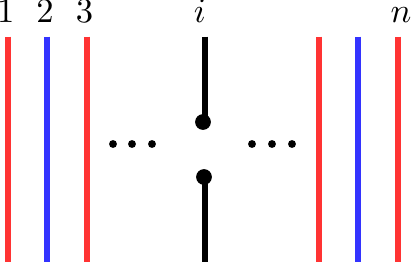}}      
\end{equation}

\noindent
where black means red if $ i $ is odd and blue if $ i $ is even.
Note that $ L_1 = U_0 $. 
(The name JM-element is motivated by the paper \cite{steen}  where it is shown that
$ L_i $ indeed is a JM-element in the sense of Mathas \cite{Mat-So}, for any Coxeter system).

\begin{lemma}  \label{lemma_writing_jucys_murphy_elements}
Let $1 < i < n$. Then we have the following formula in $\tilde{A}_w$
	\begin{equation}\label{equation_generators_not_one_not_n}
		L_i = U_{i-1}L_{i-1} +L_{i-1}U_{i-1} -2  U_{i-1} \sum_{j=1}^{i-2}  L_{j}.
	\end{equation}
        Consequently, for all $ 1 < i < n $ we have that $L_i$ 
        belongs to the subalgebra of $\tilde{A}_{w}$
        generated by the elements $L_1, U_1, \ldots ,U_{n-2}$.
\end{lemma}

\begin{dem}
  Let us show the formula \eqref{equation_generators_not_one_not_n} in the case
  $ i = n-1 $ and $ i $ odd. The general case of the formula, that is the case where $ i $ is any number strictly smaller than $ n $,
  is shown the same way. We have that 
  
\begin{equation}\label{JMexpansion}
  L_i =\raisebox{-.45\height}{\includegraphics[scale=0.6]{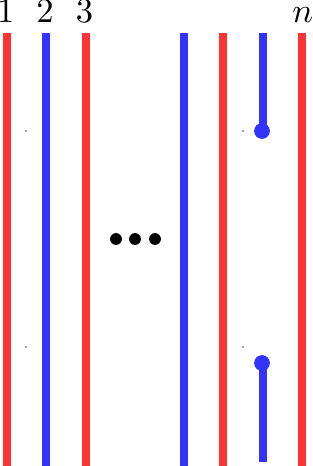}}      
  \, \,\,  =  \, \,\,
  \raisebox{-.45\height}{\includegraphics[scale=0.6]{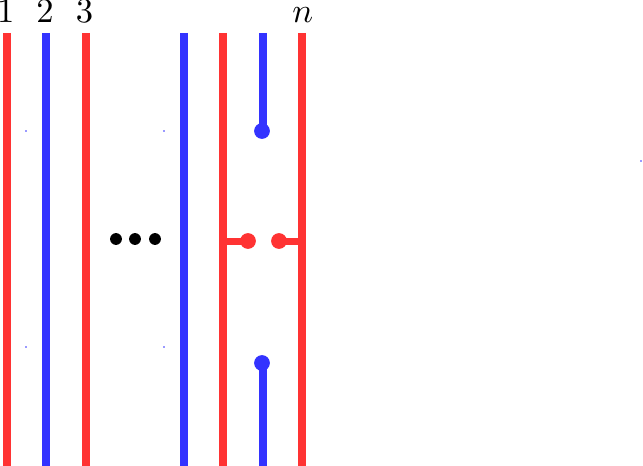}}
  \! \! \!\! \! \!\! \! \!\! \! \! \!\! \! \!\! \! \! \!\! \! \!\! \! \!  \!\! \! \!
  =    \, \,\, 
  \raisebox{-.45\height}{\includegraphics[scale=0.6]{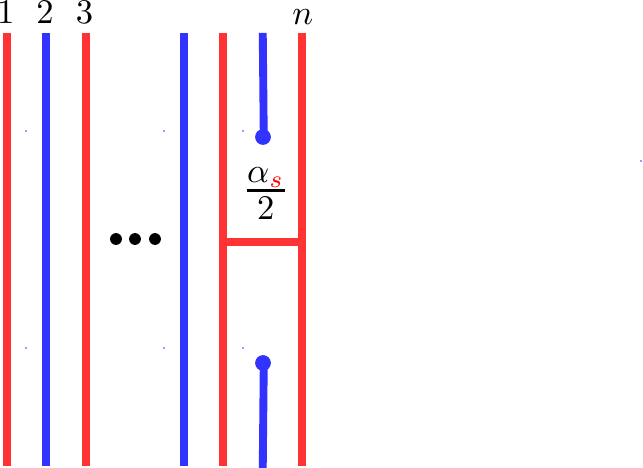}}
    \! \! \!\! \! \!\! \! \!\! \! \! \!\! \! \!\! \! \! \!\! \! \!\! \! \! \!\! \! \!
  +   \, \,\,    
  \raisebox{-.45\height}{\includegraphics[scale=0.6]{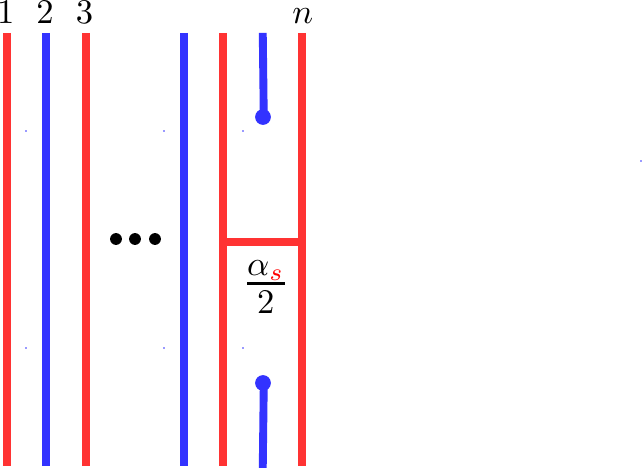}}
  \! \!\! \! \! \!\! \! \!\! \! \!  \! \!\! \! \!  \! \!\! \! \!                                     = 
\end{equation}

\medskip

\begin{equation}\label{JMexpansionA}
  \raisebox{-.43\height}{\includegraphics[scale=0.6]{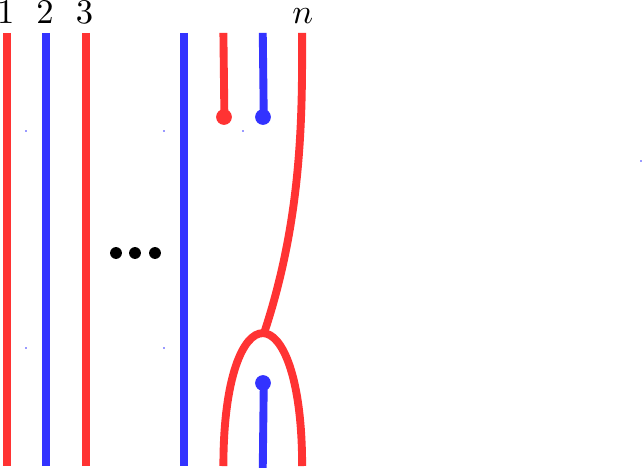}}
  \! \!\! \! \! \!\! \! \!\! \! \!\! \!\! \! \! \!\! \! \!\! \! \!   \!\! \! \!\! \! \!   
  + \, \, \, \, \, \, \, \, \, 
  \raisebox{-.43\height}{\includegraphics[scale=0.6]{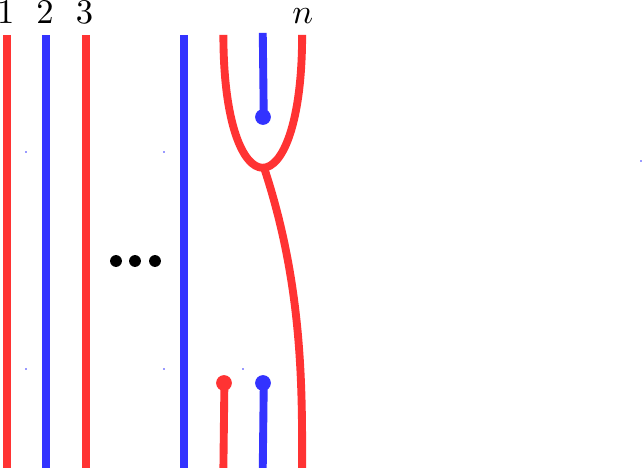}}
  \! \!\! \! \! \!\! \! \!\! \! \!\! \!\! \! \! \!\! \! \!\! \! \!\!\! \!
 +\, \, \, \, \, \, \, \, \, 
 \raisebox{-.43\height}{\includegraphics[scale=0.6]{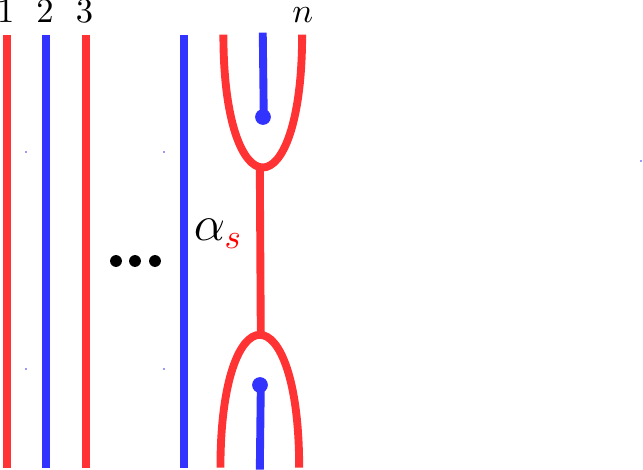}}      
\end{equation}

 \medskip\medskip
\noindent 
The first two diagrams of \ref{JMexpansionA} are 
$ U_{i-1} L_{i-1} $ and $ L_{i-1} U_{i-1}  $ and so we only have to check that the last
diagram of \ref{JMexpansionA} is equal to $-2 U_{i-1}\sum_{j=1}^{i-2}  L_{j} $.
But this follows via repeated applications
of the polynomial relation \ref{sgraphD}, moving $ {\alpha_\ese} ={-\alpha_\te} $ all the way to the left.
\end{dem}

\medskip
The $ L_i $'s are important since they allow us to generate variations of \eqref{evenmoregeneralA}
and \eqref{evenmoregeneralB} with no `connecting' arcs, as follows

\begin{equation}\label{generalnonhangingcage}
 \begin{array}{l}
  (U_1 U_3 U_5 \cdots U_{2k+1})  L_{2k+3}  (U_1 U_3 U_5 \cdots U_{2k+1})
\, \, \,  \, \, \,    = \, \, \, \, \, \, 
   \raisebox{-.43\height}{\includegraphics[scale=0.6]{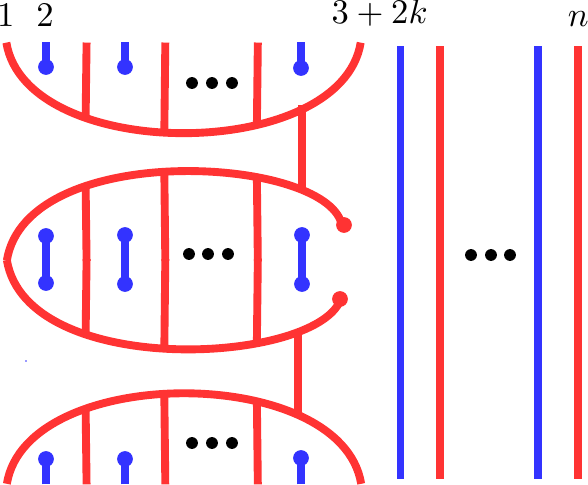}}      
   \, \, \, \,
   = 
  \\ \\ 
   \,  \, \,
\raisebox{-.43\height}{\includegraphics[scale=0.6]{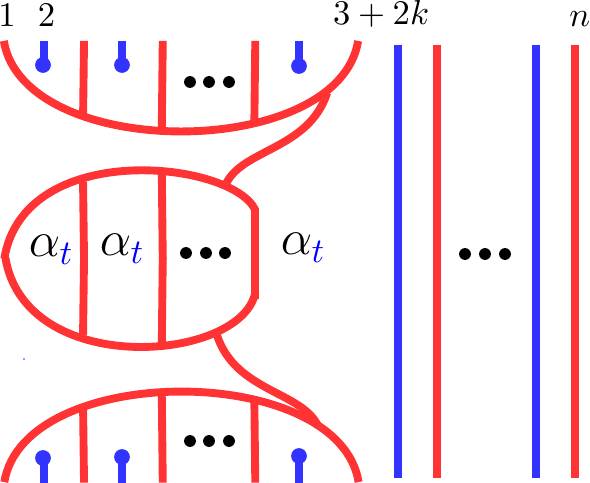}}      
 \, \, \,     \, \, \,    =   \, \, \,  \, \, \, 
\raisebox{-.43\height}{\includegraphics[scale=0.6]{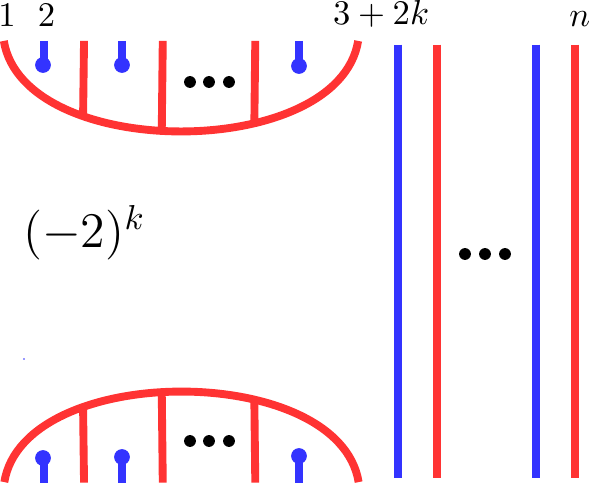}}      
 \end{array}
\end{equation} 

\medskip 
\noindent
where we for the last equality used the polynomial relation \ref{sgraphD} as well
as \ref{rel muere a la izquierda}. Thus any diagram of the form
\ref{hangingcageLEFTADJ} belongs to the subalgebra of $\tilde{A}_{w}$ generated by the $ L_i$'s and the $ U_i$'s.
Note on the other hand that in order for this argument to work, 
the diagram in question must be left-adjusted, that is
without any through arcs on the left 
as in \ref{hangingcageLEFTADJ}).

\begin{equation}\label{hangingcageLEFTADJ}
\raisebox{-.43\height}{\includegraphics[scale=0.8]{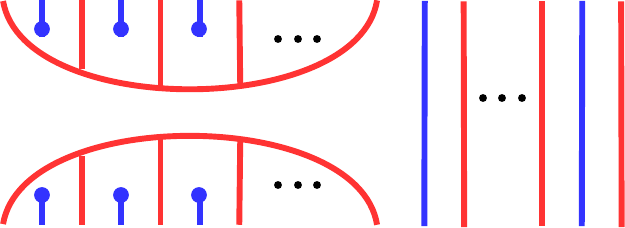}}      
\end{equation}

\noindent

\medskip
The diagrams corresponding to double light basis elements of $\tilde{A}_{w}$ are built up of top and bottom `half-diagrams', similarly to 
the Temperley-Lieb diagrams and the blob diagrams considered in the previous section. These half-diagrams are called light leaves. 

 We now introduce the following bottom half-diagrams, called \emph{full birdcages}
by Libedinsky in \cite{LibedinskyGentle}.

\begin{equation} 
  \raisebox{-.43\height}{\includegraphics[scale=0.9]{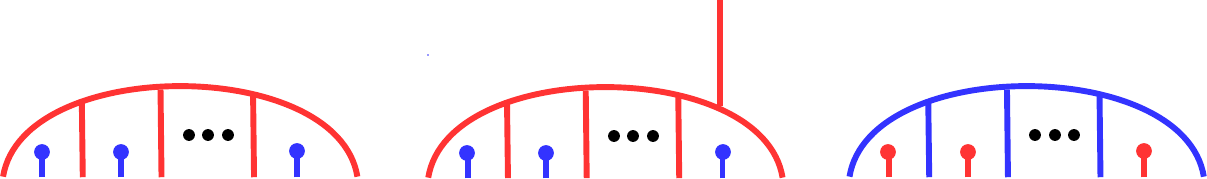}}      
\end{equation}

We say that the first and the last of these half-diagrams are \emph{non-hanging full birdcages},
whereas the middle one is \emph{hanging}. We also say that the 
first two full birdcages are \emph{red}, and the third one is \emph{blue}. 
We define the \emph{length} of a full birdcage to be the number of dots contained in it.
We view the half-diagrams

\begin{equation}
\!   \! \!   \! \!   \! \!   \! \!   \! \!   \! \!   \! \!   \! \!   \! \!   \! 
  \raisebox{-.43\height}{\includegraphics[scale=0.9]{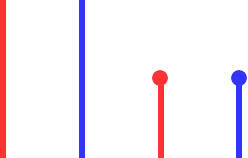}}      
\end{equation} 
as \emph{degenerate} full birdcages of lengths $0$. A full birdcage which is not degenerate is called
\emph{non-degenerate}. We shall also consider \emph{top full birdcages}, that are obtained from
bottom full birdcages, by a reflection through a horizontal axis. Here are two examples of lengths four and
three.

\begin{equation}
      \raisebox{-.43\height}{\includegraphics[scale=1]{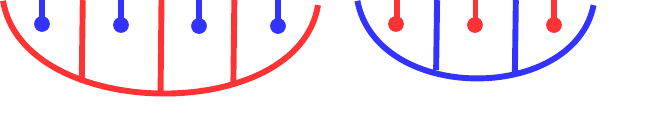}}      
\end{equation} 

Light leaves are built up of full birdcages in a suitable sense that we shall now explain.
We first consider the operation of \emph{replacing a degenerate non-hanging full birdcage by a non-hanging
non-degenerate full birdcage 
of the same color}. Here is an example

\begin{equation}\label{insertionStepFirst}
\! \! \! \! \! \! \! \! \! \! \! \! \! \!
 \raisebox{-.43\height}{\includegraphics[scale=1]{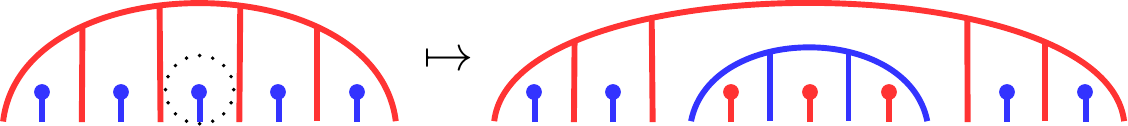}}      
\end{equation}

\noindent
The reason why we only consider the application of this operation to non-hanging birdcages is that 
applying it to a degenerate hanging birdcage only gives a new, larger full birdcage; in other words nothing new.
Here is an example

\begin{equation}\label{insertionStepNothing}
  \! \! \! \! \! \! \! \! \! \! \! \! \! \! \! \! \! \! \! \! \! \! \! \! \! \! \! \! \! \!
   \raisebox{-.43\height}{\includegraphics[scale=1]{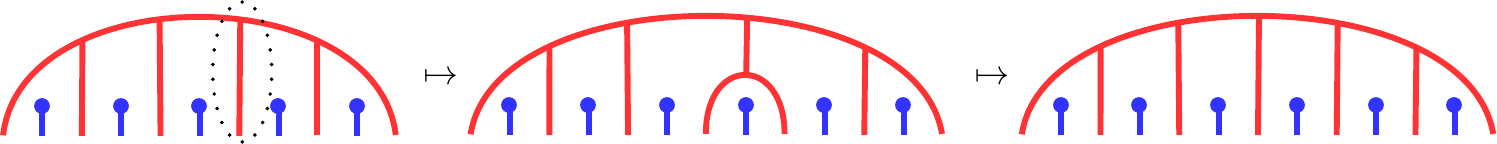}}      
\end{equation}

Following Libedinsky, we now define a \emph{birdcagecage} to be any diagram
that can be obtained from a  degenerate non-hanging birdcage by performing 
the above operation recursively a finite number of times on the degenerate birdcages that appear at each step.
Here is an example of a birdcagecage.


\begin{equation}\label{insertionStep}
   \raisebox{-.43\height}{\includegraphics[scale=1]{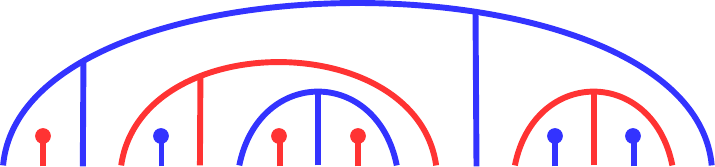}}      
\end{equation}

Now, according to  \cite{LibedinskyGentle}, any light leaf is built up of
birdcagecages as indicated below in \ref{birdcagecagezonesA}. 
Here in \ref{birdcagecagezonesA} 
the number of bottom boundary points is $ n $.
Zone A consists of a number of non-hanging birdcagecages whereas 
zone B consists of a number of hanging birdcagecages. On the other hand
zone C consists of at most one 
non-hanging birdcagecage.

\begin{equation}{\label{birdcagecagezonesA}}
 \! \!  \! \! \!  \! \! \!  \! \! \!
  \raisebox{-.43\height}{\includegraphics[scale=0.9]{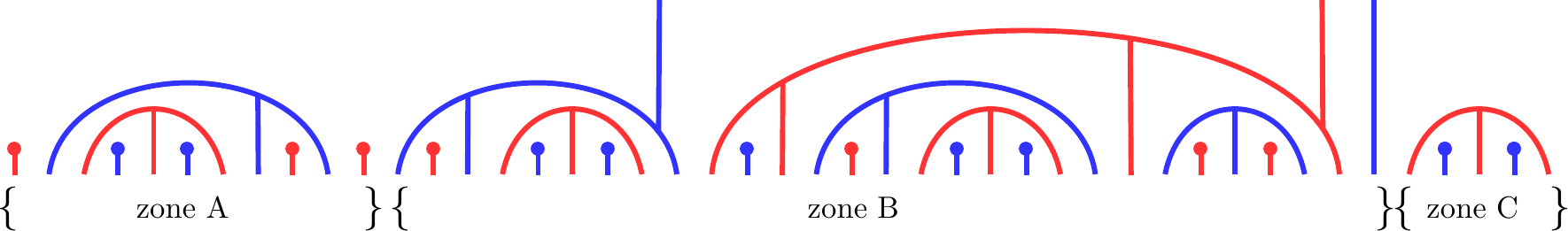}}      
\end{equation}

Note that each of the three zones may be empty, but they cannot all be empty since $ n > 0$.
In the case where zone B is empty, we define zone C to be the last birdcagecage. In other words, if zone B is empty then zone C is always nonempty, whereas zone A may be empty.

\medskip
The hanging birdcagcages of zone B define an element $ v \in W $.
It satisfies $ v \leq w $ where $ \leq $ denotes  
the Bruhat order on $W$. In the
above example we have $ v = \color{blue} t \color{red} s \color{blue} t$. 
The \emph{double leaves basis of $ \tilde{A}_w $} is now obtained by running over all $ v \le w  $
and over all pairs of light leaves that are 
associated with that $ v $. For each such pair $ (D_1, D_2) $ 
the second component $ D_2 $ is reflected through a horizontal axis, and finally the two
components are glued together. The resulting diagram is a double leaf.  
{\color{black} In Figure \ref{fig:an example of a double left} we give an example.}

\begin{figure}[h]
  \raisebox{-.43\height}{\includegraphics[scale=0.85]{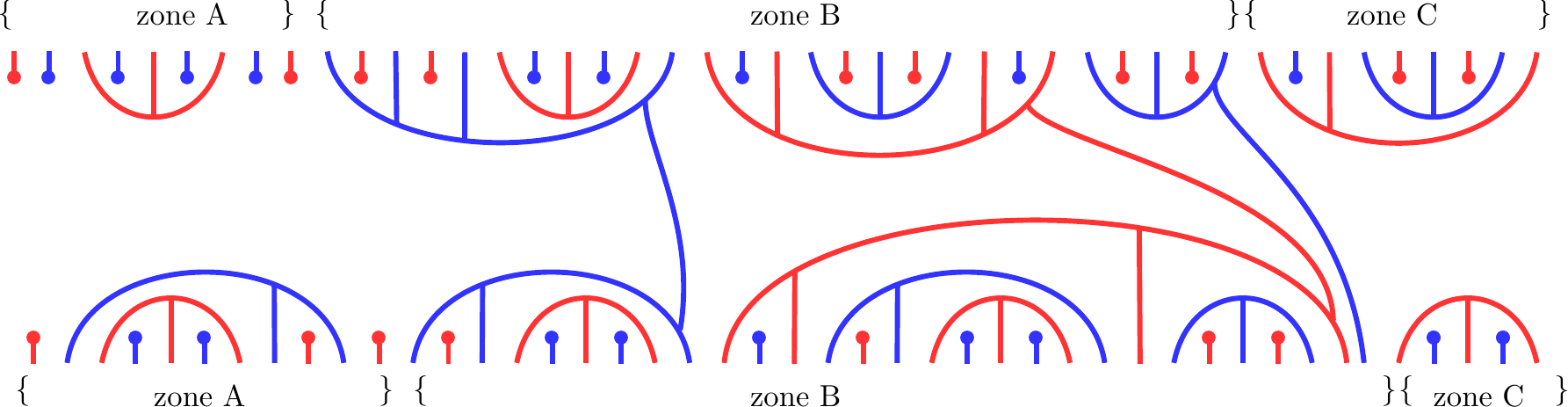}}      
\caption{An example of a double leaf.}
\label{fig:an example of a double left}
\end{figure}

%

\noindent
Note that although the total number of top and bottom boundary points of each double leaf is the same,
the number of boundary points in each of the three zones need not coincide, although the parities do coincide.
In the above example, there are for instance nine top boundary points in zone C but only five bottom boundary
points in zone C.
Note also that the number of top and bottom birdcagecages in zone B always is the same,
three in the above example. 
This is of course also the case in zone C but not necessarily in zone A, although the parities must coincide. In the above example,  we have five top birdcagecages in zone A but only three bottom birdcagecages in zone A. Moreover, there are nine top boundary points in zone A but eleven bottom boundary points in zone A. 

\medskip
For future reference we formulate the Theorem already alluded to several times.
\begin{theorem}
The double leaves form an $ \F $-basis for $ \tilde{A}_w$. 
\end{theorem}  
\begin{dem}
  This is mentioned in \cite{LibedinskyGentle}. It is a consequence of the recursive construction of
  the light leaves. 
\end{dem}

\begin{definition} 
Let $ A_w $ be the subspace of $ \tilde{A}_w$ spanned by the double leaves with empty zone C .
\end{definition}

With these notions and definitions at hand, we can now formulate and prove the following Theorem.

\begin{theorem}{\label{mainTheoremSection3}}
Let $w\in W$ with $w=n_{\ese} $. Then, we have
\begin{description}
	\item[a)] As an algebra, $\tilde{A}_w$ is generated by the elements $U_1,\ldots , U_{n-2}$ and $L_1,\ldots,L_n$.
	\item[b)]  $ A_w $ is a subalgebra of $ \tilde{A}_w$. It is generated by $U_1,\ldots , U_{n-2}$ and $L_1 = U_0 $.
	\item[c)] The dimensions of $ A_w $ and $ \tilde{A}_w $ are given by the 
  formulas
\begin{equation}
  \dim_{\mathbb{F}}( A_w) = \binom{2n}{n}   \qquad
  \mbox{ and } \qquad  \dim_{\mathbb{F}}(\tilde{ A}_w) = 2\binom{2n}{n}   .	
\end{equation}	
\end{description}
\end{theorem}
\begin{dem}
  We first prove $a)$ of the Theorem. We define $\tilde{A}^{\prime}_w$ as the subalgebra of $\tilde{A}_w$
  generated by the $ U_i$'s and the $ L_i$'s. Thus, in order to show $a) $ we must prove that
  $\tilde{A}^{\prime}_w = \tilde{A}_w $. We shall do so by proving that $\tilde{A}^{\prime}_w$ contains
  all the double leaves basis elements for $ \tilde{A}_w$.

  \medskip
  We first observe that the diagrams in \ref{evenmoregeneralA} and \ref{evenmoregeneralB} both
  belong to $\tilde{A}^{\prime}_w$. In fact, multiplying them together we get that any diagram
  of the form   


\begin{equation}\label{symmetric}
 \raisebox{-.5\height}{\includegraphics[scale=0.85]{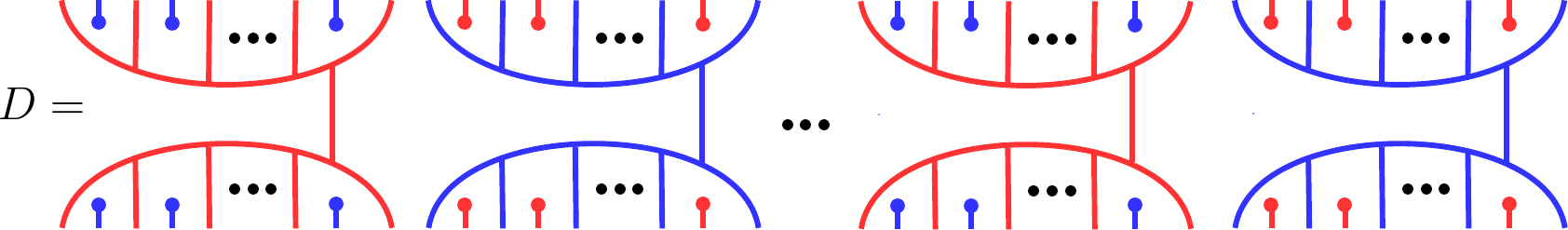}}      
\end{equation}
\noindent
belongs to $\tilde{A}^{\prime}_w$. Here the length of each full birdcage on the bottom (which may
be zero) 
is equal to the length of the
corresponding full birdcage on top of it, that is the diagram in
\ref{symmetric} is symmetric with respect to a horizontal axis.
Note that the diagram $ D $ in \ref{symmetric} is a preidempotent; to be precise we have that 
\begin{equation}
D^2 = (-2)^{l_1 + \ldots + l_r } D,
\end{equation}
where $ l_1, l_2, \ldots, l_r $ are the lengths of the bottom full birdcages that appear in $ D $.
Now we can repeat the calculations from \ref{generalnonhangingcage} and \ref{hangingcageLEFTADJ}
in order to remove the connecting arc between the first bottom full birdcage of
$ D $ and its top mirror image: 

\vspace{-0.5cm}

\begin{equation}\label{twokindsof}
  D L_{k_1} D =  
  \raisebox{-.5\height}{\includegraphics[scale=0.5]{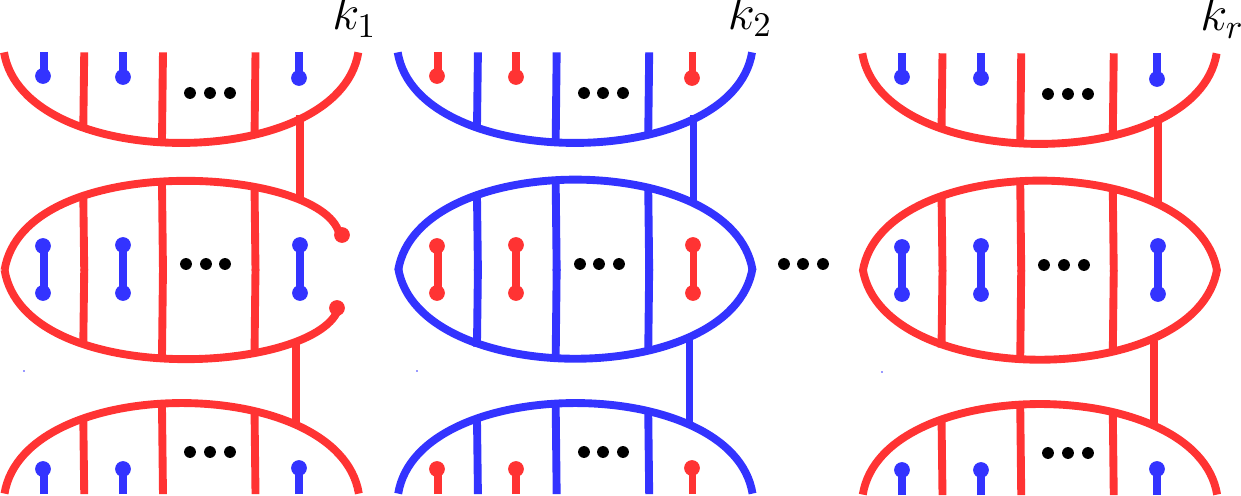}}      
  = 
\raisebox{-.5\height}{\includegraphics[scale=0.5]{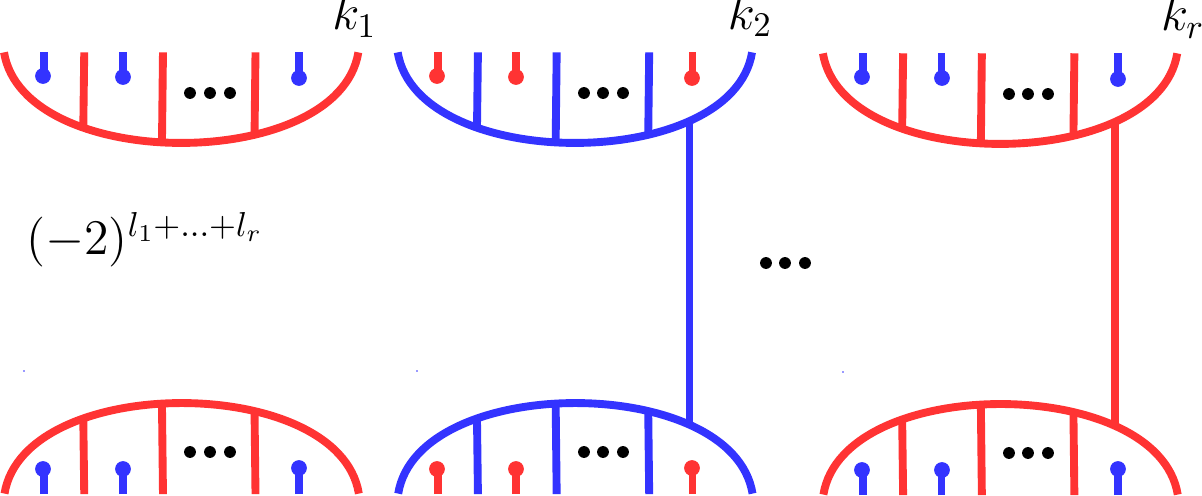}}      
\end{equation}

\noindent
In other words, we get that $ D_1 := (-2)^{-(l_1 + \ldots + l_r) }   D L_{k_1} D    $
is equal to $ D $, but with the first connecting arc removed, and that $D_1$ belongs to $ \tilde{A}^{\prime}_w$.

\medskip
From $ D_1 $ we can now remove the next connecting arc as follows

\begin{equation}\label{twokindsofX}
  D_1 L_{k_2} D =  
  \raisebox{-.5\height}{\includegraphics[scale=0.5]{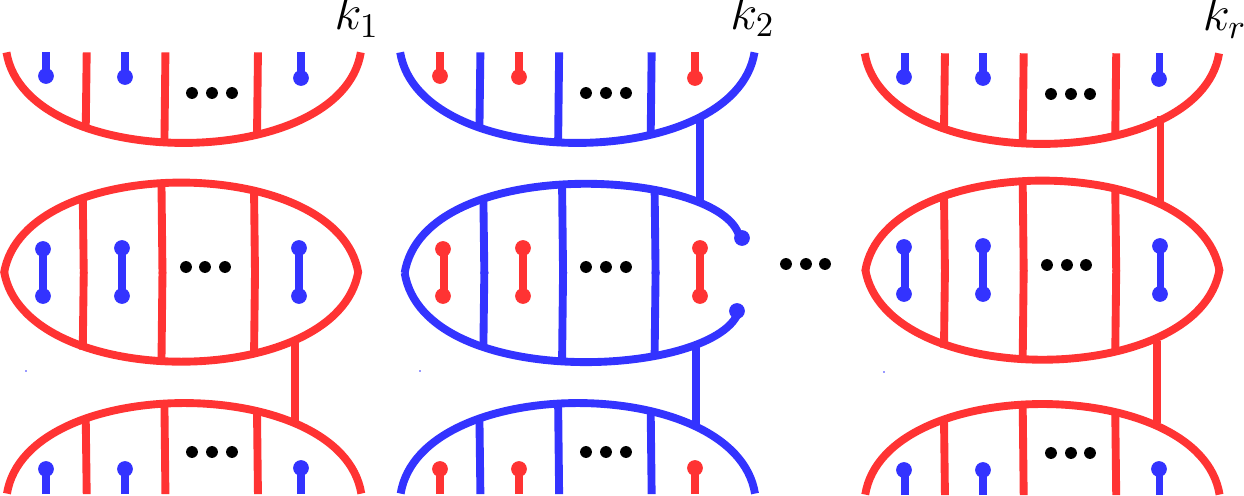}}      
= 
\raisebox{-.5\height}{\includegraphics[scale=0.5]{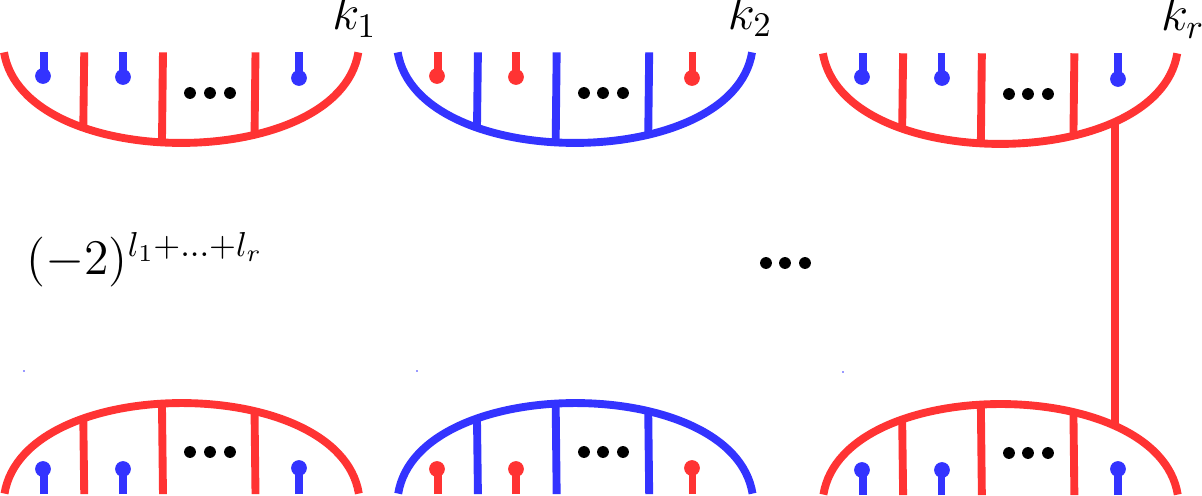}}      
\end{equation}

\medskip
\noindent
Continuing this way we find that any diagram of the form 

\begin{equation}\label{symmetricXX}
  \raisebox{-.5\height}{\includegraphics[scale=0.8]{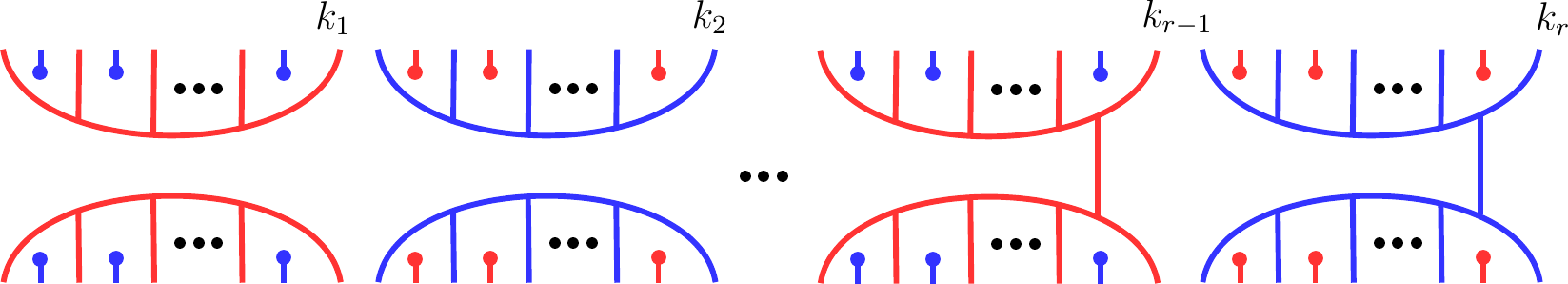}}      
\end{equation}

\noindent
belongs to $ \tilde{A}^{\prime}_w$. 

\medskip
The diagrams in \ref{symmetricXX} consist of a number of non-hanging full birdcages followed by
a number of hanging full birdcages. We shall now prove that the rightmost hanging full birdcage
of \ref{symmetricXX} may be
transformed into a non-hanging full birdcage and still give rise to an element of $ \tilde{A}^{\prime}_w$.
Let $i<n$ be a positive integer
{\color{black}{of}} the same parity {\color{black}{as}} $n$. We consider
the diagram $F_i:=U_i U_{i+3} \cdots U_{n-2}$:
\begin{equation}\label{last1}
 F_i= 
 \raisebox{-.45\height}{\includegraphics[scale=0.8]{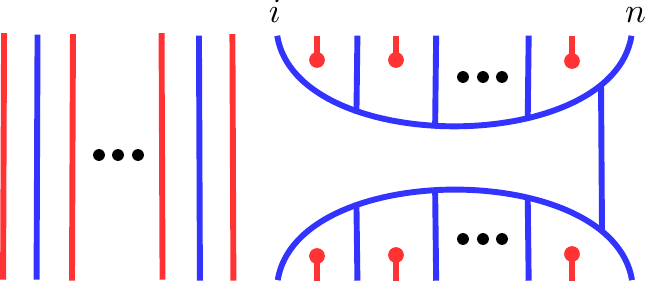}}      
\end{equation}
\medskip \noindent
We notice that only the rightmost top and bottom full birdcages
{\color{black}{of $ F_i$}} 
are non-degenerate, {\color{black}{of length $l:=(n-i)/2$}}.

Then we have that $ F_i  L_n  F_i \in \tilde{A}^{\prime}_w$. On the other hand,
we also have that 
\begin{equation}\label{lastdiagram}
    F_i  L_n  F_i = 
   \raisebox{-.45\height}{\includegraphics[scale=0.6]{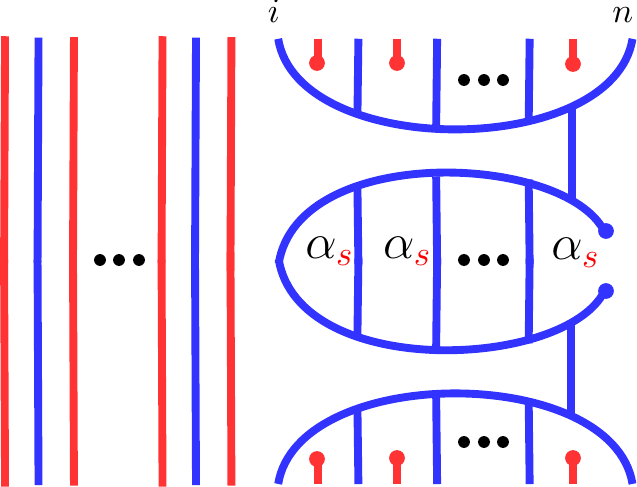}}      
  \, \,\, =   \, \,\,
   \raisebox{-.45\height}{\includegraphics[scale=0.6]{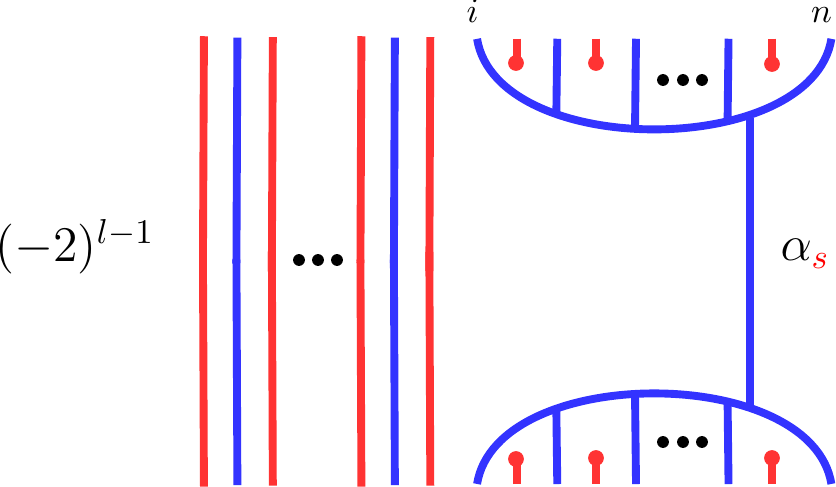}}      
   =
\end{equation}    
\begin{equation}    
      \raisebox{-.45\height}{\includegraphics[scale=0.6]{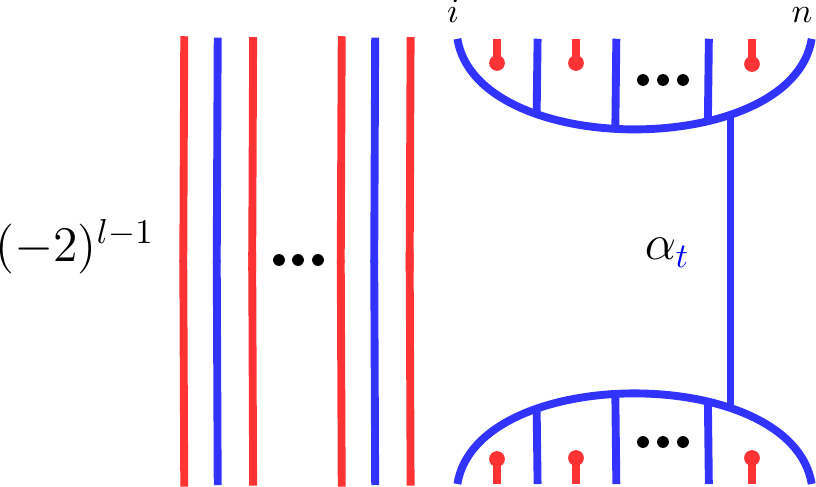}}      
 \, \,\,  +   \, \,\, 
   \raisebox{-.45\height}{\includegraphics[scale=0.6]{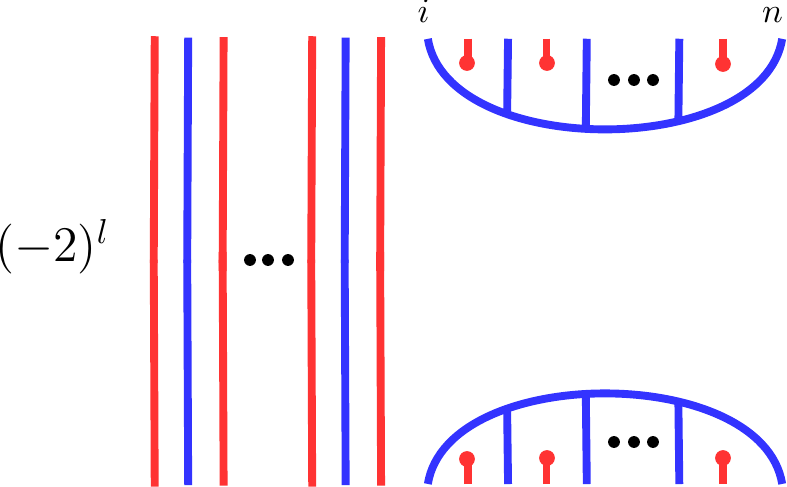}}      
  \end{equation} 
\noindent
We consider the first diagram $ X$ of the last sum.
Moving $  \alpha_\te $
all the way to the left we get that
\begin{equation}
	X= -2F_i\sum_{j=1}^{i-1} L_i
\end{equation}
Therefore, $X$ belongs to $ \tilde{A}^{\prime}_w$. But from this we conclude that also 
the second diagram of
the sum belongs to $ \tilde{A}^{\prime}_w$.
Finally, multiplying this diagram with diagrams from \ref{symmetricXX} we conclude that 
any diagram of the form 

\vspace{-0.5cm}

\begin{equation}\label{symmetricXXX}
 \raisebox{-.5\height}{\includegraphics[scale=0.75]{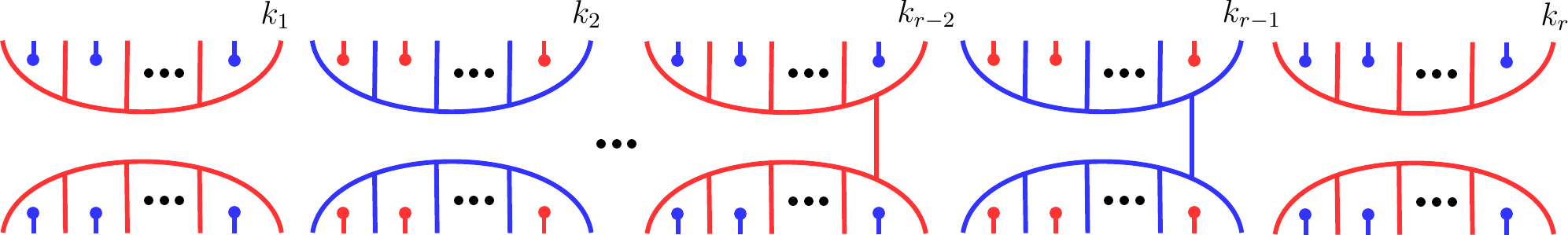}}      
\end{equation}

\noindent
belongs to $ \tilde{A}^{\prime}_w$, proving the above claim. In other words, we have shown that any double leaves basis
element of $ \tilde{A}_w$, that is 
built up of full birdcages and is symmetric 
with respect to a horizontal axis, belongs to $ \tilde{A}^{\prime}_w$.

\medskip
We next show that omitting the symmetry condition in the diagrams \ref{symmetricXXX} still 
gives rise to an element of $ \tilde{A}^{\prime}_w$.
Our first step for this is to produce a way of `moving points' from a full birdcage to its
neighboring full birdcage. 
We do this by multiplying by `overlapping' $ U_i$'s. Consider the following example

\begin{equation}\label{birdcagemovepoints}
  \raisebox{-.5\height}{\includegraphics[scale=1]{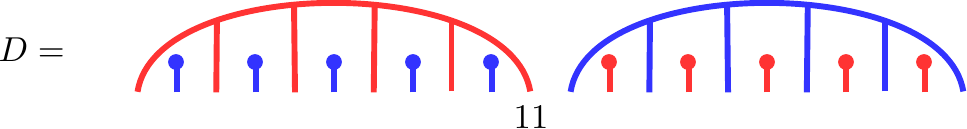}}      
\end{equation}

\noindent 
consisting of two full birdcages, both of length 5. In this case the overlapping $ U_i$'s are
$ U_{10} $ and $ U_{11} $. Multiplying $ D $ below with $ U_{10} $ produces a diagram with
two full birdcages as well, but this time of lengths 4 and 6, whereas multiplying $ D $ below by $ U_{11} $ 
produces a diagram with
two full birdcages, of lengths 6 and 4: 


\begin{equation}\label{birdcagemovepointsX}
   \raisebox{-.5\height}{\includegraphics[scale=0.7]{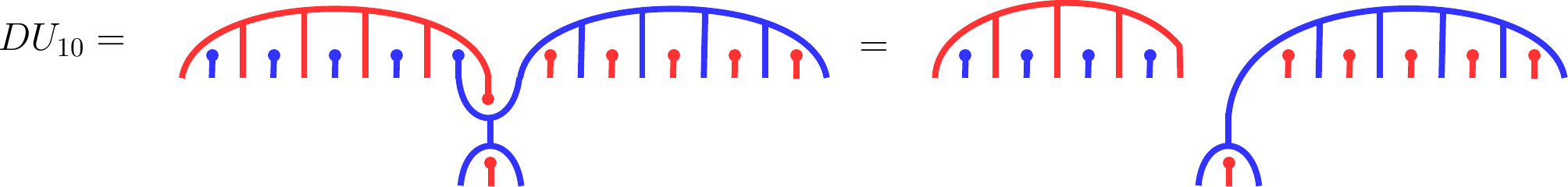}}      
\end{equation}

\begin{equation}\label{birdcagemovepointsXX}
     \raisebox{-.5\height}{\includegraphics[scale=0.7]{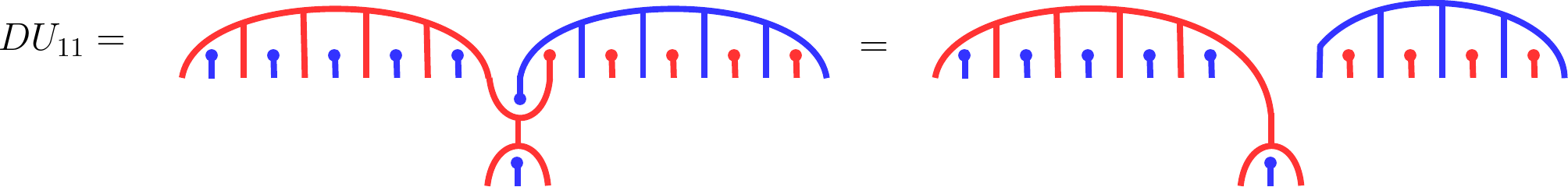}}      
\end{equation}


\medskip
This gives us a method for moving points from one full birdcage to a neighboring full birdcage that works
in general, for hanging as well as for non-hanging full birdcages, and so we get that any diagram of the
form

\begin{equation}\label{nonsymmetricX}
  \raisebox{-.5\height}{\includegraphics[scale=0.8]{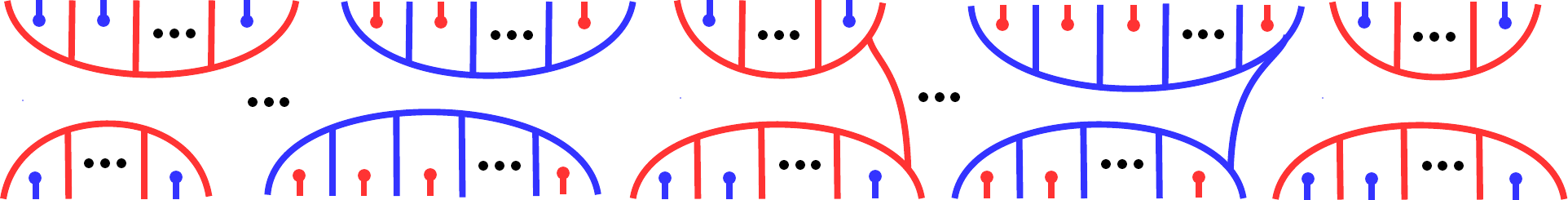}}      
\end{equation}

\noindent
belongs to $ \tilde{A}^{\prime}_w$. These diagram are not horizontally symmetric anymore
but still the total number of top full birdcages is equal to the total number of
bottom full birdcages. Actually, by the description of the light leaves basis, this is expected in zones B and 
C, but not in zone A. However, multiplying a full birdcage in zone A
with an JM-element $ L_i $ of the opposite color it 
breaks up in three smaller full birdcages, the middle one being degenerate. For example, for

\begin{equation}\label{birdcagebreak}
   \raisebox{-.5\height}{\includegraphics[scale=0.8]{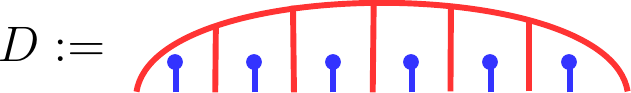}}      
\end{equation}
we have that \\

\begin{equation}\label{birdcagebreak}
     \raisebox{-.5\height}{\includegraphics[scale=0.8]{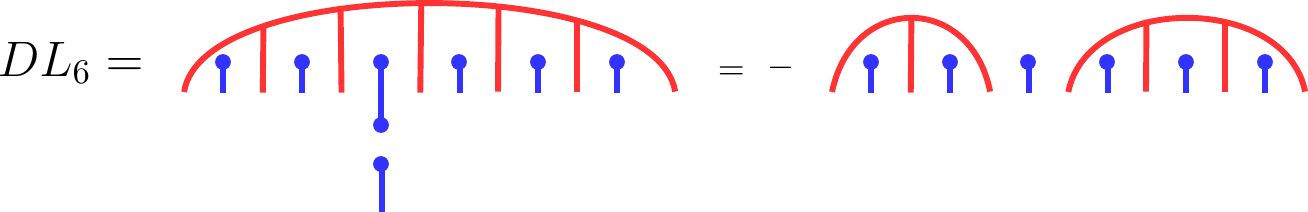}}      
\end{equation}

Combining this with the procedure
of moving points from a full birdcage to a neighboring full birdcage, we conclude
that in the diagram \ref{nonsymmetricX} we may
assume that the number of top full birdcages in zone A is different
from the number of bottom full birdcages and still the diagram 
belongs to $ \tilde{A}^{\prime}_w$.

\medskip
Thus, to finish the proof of $a)$ we now only have to show that the full birdcages in the diagram \ref{nonsymmetricX}
may be replaced
by birdcagecages.
It is here enough to consider a single bottom birdcage. 

\medskip
The replacing of a degenerate non-hanging birdcage by a non-degenerate full birdcage can be viewed as  
the  
insertion of a non-hanging birdcage in a full birdcage of the opposite color.
But this can be achieved via multiplication with appropriate diagrams 
of the 
form \ref{evenmoregeneralA} and \ref{evenmoregeneralB}.
Consider for example the birdcagecage $ D $ in \ref{insertionStepFirst}. It can be obtained as follows

\begin{equation}\label{insertX}
  \raisebox{-.5\height}{\includegraphics[scale=0.8]{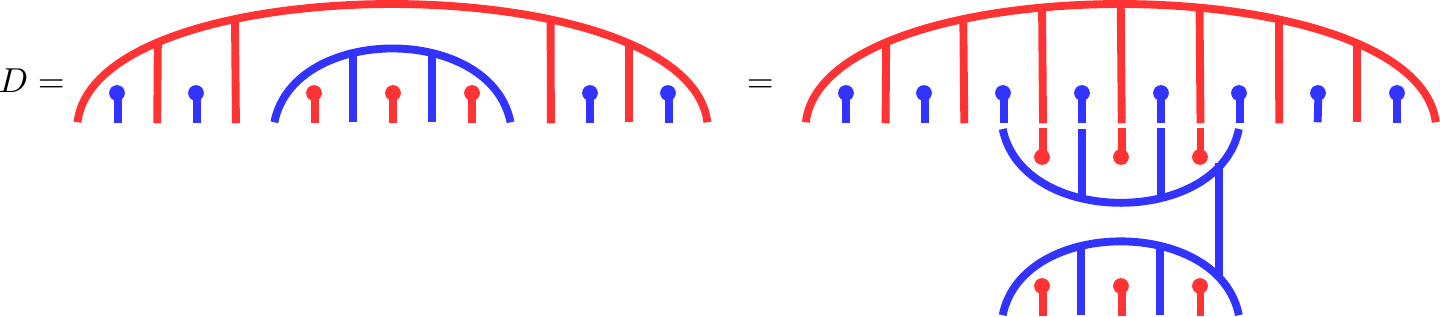}}      
\end{equation}

\noindent
Repeating this process we can obtain any birdcagecage. 
This finishes the proof of $a)$.

\medskip
We next show $c) $. For this we first note that 
there is a bijection between double leaves with empty zone C and double leaves with
nonempty zone C, given by removing the connecting line between the 
last bottom and top birdcagecage. Hence we have that  
\begin{equation}
\dim_{\mathbb{F}}( \tilde{A}_w) = 2 \dim_{\mathbb{F}}( A_w).
\end{equation}
On the other hand, from the vector space isomorphism given in Corollary 8.3 of 
\cite{Esp-Pl} it follows that 
$  \dim (\tilde{A}_w) = \dim \widetilde{\mathbb{NB}_n} $ and so $ c) $ follows from Theorem
\ref{blob diagram realization} and
Corollary \ref{cor dim}. 
(Note that in \cite{Esp-Pl} the authors use the notation $ {A}_w $ for $ \tilde{A}_w $).

\medskip

We finally show $b) $.  Let $ A_w^{\prime}$ be the subalgebra of $ \tilde{A}_w $ generated by 
$ U_1, \ldots, U_{n-2} $ and $U_0 = L_1 $. In view of Lemma \ref{lemma_writing_jucys_murphy_elements} we first observe that 
$ A_w^{\prime}$ is the same 
as the subalgebra of $ \tilde{A}_w$ generated by $ U_1, \ldots, U_{n-2} $ and $L_1, \ldots, L_{n-1}  $. 
On the other hand,  going through the proof of $ a) $ we see that the last JM-element $ L_n $ is only needed for the steps \ref{last1} and 
\ref{lastdiagram} where a hanging birdcage at the right
end of the diagram is transformed into a non-hanging one, and so we have that 
 $ A_w \subseteq A_w^{\prime}   $. But from Theorem \ref{universal} we have that $ \dim (A_w^{\prime} ) \le  \dim {\mathbb{NB}_{n-1}} =  \dim (A_w )$ 
where we used $ c) $ for the last equality. Hence the inclusion $ A_w \subseteq A_w^{\prime}   $ is an equality and $ b) $ is proved.
\end{dem}

\begin{corollary}  \label{corollary short presentation}
Let $w\in W$ with $w=n_{\ese} $. Then, we have
\begin{description}
	\item[a)] The map $\varphi$ defined in Theorem \ref{universal} induces an algebra 
  isomorphism $ \varphi:  \mathbb{NB}_{n-1} \rightarrow A_w$. 
    \item[b)] Setting $ J_n:= L_1+L_2+\ldots +L_n $ we have that the
  extension of $ \varphi $ to $ \widetilde{\mathbb{NB}}_{n-1} $ given by $ \tilde{\varphi}(\mathbb{J}_{n-1}) = J_n  $ 
induces an algebra isomorphism $  \tilde{\varphi}: \widetilde{\mathbb{NB}}_{n-1} \rightarrow \tilde{A}_w$.
\end{description}
\end{corollary}

\begin{dem}
Part $ a) $ was already proved in the previous Theorem so let us concentrate on part $ b) $. 
Here we have already checked all the relations that do not involve $J_n$
and so we only have to check that $J_n^2=0$ and that $J_n$ is central in $ \tilde{A}_w$.
Now by \cite[Lemma 3.4]{Esp-Pl} we know that $L_1^2=0$ and that 
	\begin{equation}\label{JorgeDavid}
		L_i^2=-2L_i \sum_{j=1}^{i-1} L_j,
	\end{equation}
	for all $2\leq i \leq n$. Thus we obtain
			\begin{align} \label{JorgeDavidA}
		J_n^2	&=(L_1+L_2+\ldots +L_n)^2 
		        = \sum_{i=2}^n L_i^2 +2\sum_{i=1}^{n-1}\sum_{j=i+1}^nL_iL_j   \\
		        &= -2\sum_{i=2}^n\sum_{j=1}^{i-1}L_i  L_j+2\sum_{i=1}^{n-1}\sum_{j=i+1}^nL_iL_j 
		        =0, 
		\end{align}
as claimed. Now let us show that $J_n$ is central in $ \tilde{A}_w $. It is enough to show that $[U_j,J_n]=0$,
                        for all $1\leq j\leq n-2$, where $ [\cdot, \cdot ] $ denotes the 
                        usual commutator bracket. We notice that $[U_j,L_i]=0$ if $i\neq j,j+1,j+2$.
                        Then we are done if we are able to show that 
	\begin{equation} \label{check relation uno}
	  [U_i,L_i+L_{i+1}+L_{i+2}]=0.
         \end{equation}

\noindent
But we have that 
\vspace{0.5cm}

\begin{equation*}\label{JMexp}  
\begin{array}{lll}  U_i \cdot (L_i+L_{i+1} + L_{i+2})  =
 \raisebox{-.5\height}{\includegraphics[scale=0.4]{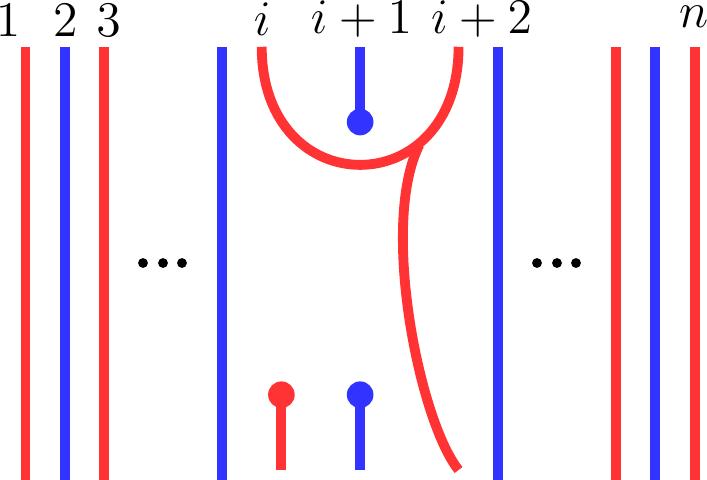}}      
 &  \, \, \,  +  \,\, \,
 \raisebox{-.5\height}{\includegraphics[scale=0.4]{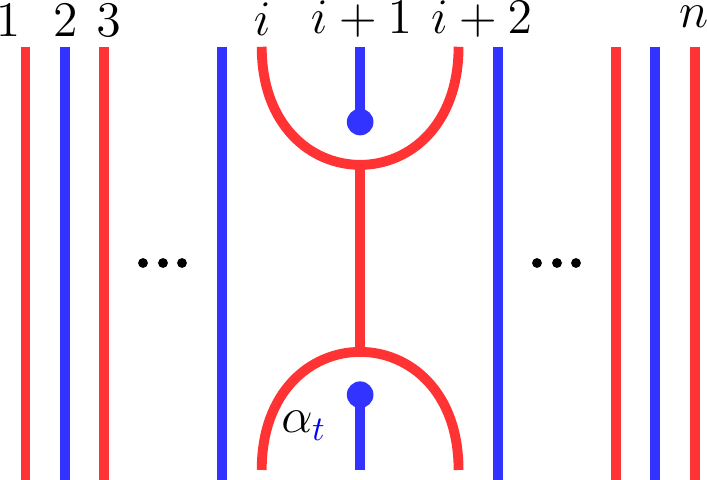}}      
 & \, \, \,  +  \,\, \, 
 \raisebox{-.5\height}{\includegraphics[scale=0.4]{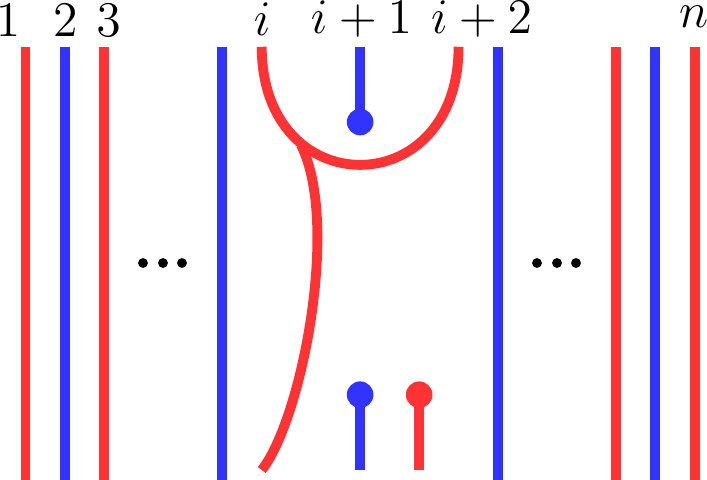}}      
\end{array}
\end{equation*}

\vspace{0.5cm}

\noindent
In the second diagram we first rewrite $ \alpha_\te= -\frac{\alpha_\ese}{2} - \frac{\alpha_\ese}{2} $
and next use the polynomial relation \ref{sgraphD}, to take the first $ -\frac{\alpha_\ese}{2} $ out of the
birdcage to the left and the second $ -\frac{\alpha_\ese}{2} $ out of the birdcage to the right. 
This will give rise to a cancellation of the first and the third terms in the expression for
$  U_i \cdot (L_i+L_{i+1} + L_{i+2})  $ and so we have that

\begin{equation*}\label{JMexp}  
\begin{array}{lll} U_i \cdot  (L_i+L_{i+1} + L_{i+2})  =
   \raisebox{-.5\height}{\includegraphics[scale=0.4]{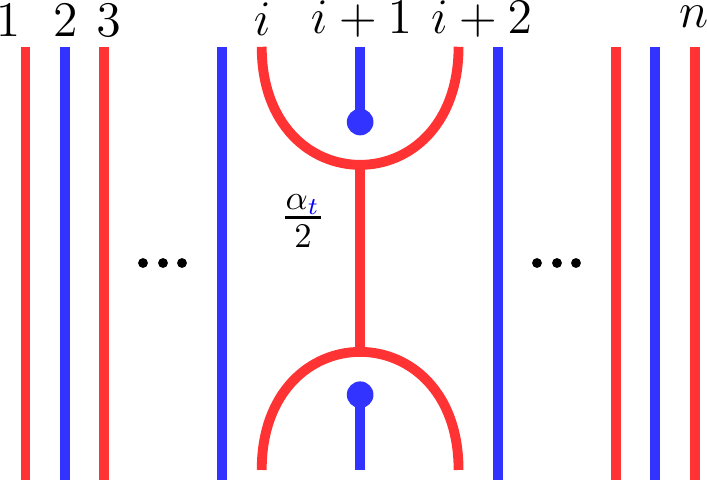}}      
  &   \, \, \,  +  \,\, \, 
 \raisebox{-.5\height}{\includegraphics[scale=0.4]{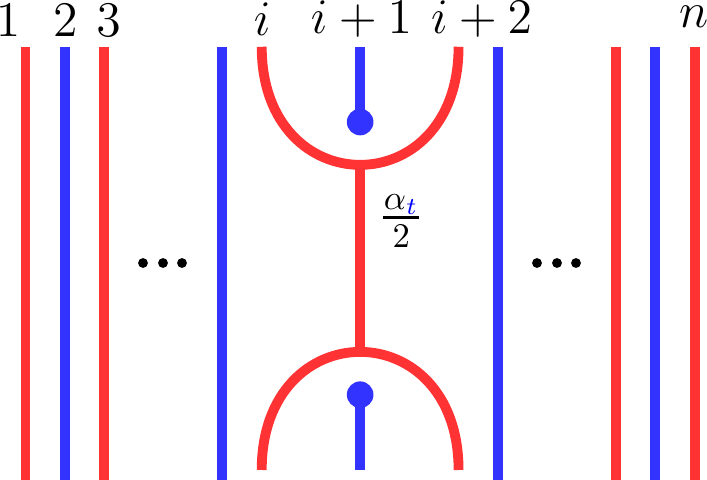}}      
   &
 \,\, \,    =  \, \,   -\, \, \,
  \raisebox{-.5\height}{\includegraphics[scale=0.4]{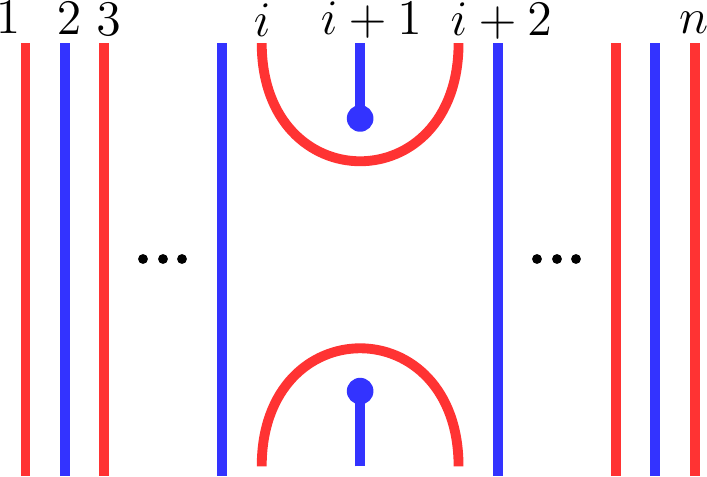}}      
\end{array}
\end{equation*}
This last diagram is symmetric with respect to a horizontal reflection and so 
\begin{equation}
	 U_i  \cdot (L_i+L_{i+1} + L_{i+2}) = (L_i+L_{i+1} + L_{i+2}) \cdot U_i   
\end{equation}
  as claimed. The Corollary is proved.
\end{dem}

\begin{remark}\label{Jucys-Murphy-proof} \rm
  Combining the isomorphism $  \mathbb{NB}_{n-1} \cong A_w$
  with Lemma \ref{lemma_writing_jucys_murphy_elements}, we obtain a proof of
Lemma \ref{JMB}.
  \end{remark} 
  
 \begin{remark} \rm
   All the results in this section consider the case $w=n_{\ese}$. Of course, {\color{black}{they}}
   remain valid if we replace $n_{\ese}$ by $n_{\te }$.
  \end{remark}

\section{The blob algebra}\label{sect-gen-blob-alg}
In this section we briefly recall the homogeneous presentation and the graded cellular basis
for the blob algebra $ \mathbb{B}_n $, given in \cite{PlazaRyom}.  We also introduce certain subalgebras of $ \mathbb{B}_n $ that are obtained by idempotent truncation. These algebras will be the main objects of study in the next two sections.

\subsection{KLR-type presentation for $\mathbb{B}_n$}
Recall from Definition \ref{definition original blob} that $ \mathbb{B}_n $ depends on the
parameter $ q \in \F$.
In general, the quantum characteristic of $ q \in F $ is defined as the minimal integer $ n \ge  0 $ such
that
\begin{equation}
1 + q + q^2 + \ldots + q^{n-1} = 0 
\end{equation}  
with the convention that it is $ \infty $ if no such $ n $ exists. We suppose from now on that $e< \infty $ is the
quantum characteristic of $q^2$.  Set $\II:=\mathbb{Z}/e\mathbb{Z}$. The symmetric group $ \Si_n $ acts on the left on $ \II^n$ via permutation of 
the coordinates $ \II^n $, that is   $s_k \cdot \bi   :=(i_1, \ldots,  i_{k+1}, i_k, \ldots, i_n)$.

\medskip
An element $ \kappa = (\kappa_1, \kappa_2 ) \in  I_e^2 $
is called a \emph{bi-charge}. We fix such a $ \kappa $ and assume that it is \emph{adjacency-free}, that is 
$\kappa_2-\kappa_1 \not \equiv 0, \pm 1  \, \,\mbox{mod }\,  e   \mbox{ for all } i\neq j $.

\medskip

We now define a diagrammatic algebra $\B^{diag}(\kappa) $ 
by introducing some extra relations in the 
Khovanov-Lauda-Rouquier algebra, see \cite{KhovanovLauda}.

\begin{definition} 
  A Khovanov-Lauda-Rouquier (KLR)
  diagram $ D $ for $ \B^{diag}(\kappa)  $ is a finite and decorated graph embedded in the 
  strip $\mathbb{R}\times [0,1]$. There are $n $ arcs in $ D $ that 
  may intersect transversally, but triple intersections are not allowed. The intersections
  of two arcs are called crosses. Each arc
  is decorated with an element of $ \II $ and its segments may be decorated with a finite number of dots. 
Each arc intersects the top boundary $\mathbb{R}\times \{ 0 \} $ and
the bottom boundary $\mathbb{R}\times \{ 1 \} $ in exactly one point.
{\color{black}{For the details concerning this definition, we refer the reader to \cite{KhovanovLauda}.}}
\end{definition} 

\noindent
Here is an example of a KLR-diagram for $ {\mathbb B}_8^{diag}(\kappa)   $  with $ e= 4 $.

\begin{equation}\label{KLRex}
  \! \! \! \! \! \! \! \! \! \! \! \! \! \! \! \! \! \! \! 
  D=\raisebox{-.5\height}{\includegraphics[scale=0.9]{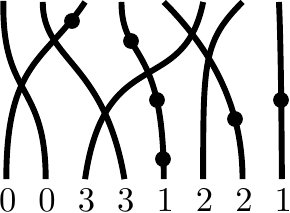}}
\end{equation}

\noindent
A KLR-diagram $D$ for $\B^{diag}(\kappa) $ gives rise to two residue sequences 
$ top(D), bot(D)  \in  \II^n $ 
obtained by reading the residues of the top and  and bottom boundary points from left to right.
In the above example \ref{KLRex} we have that $bot(D)=(0,0,3,3,1,2,2,1)$ and $top(D)=(0,3,0, 1, 2,3,2,1)$.

\medskip

Let us now define the algebra $\B^{diag}(\kappa) $.  As an $ \F$-vector
space it consists of the $\F$-linear combinations of KLR-diagrams for
$ \B^{diag}(\kappa)$ modulo planar isotopy and modulo the following relations:

\begin{equation}\label{mal-inicio}
  \! \! \! \! \! \! \! \! \! \! \! \! 
\raisebox{-.6\height}{\includegraphics[scale=0.9]{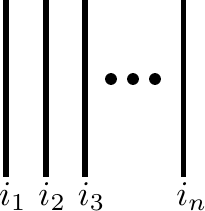}}
= 0 \,\, \,\,\, \,  \mbox{ if } i_1 \notin \{ \kappa_1, \kappa_2 \} 
\end{equation}

\begin{equation}\label{otro-mal-inicio}
  \! \! \! \! \! \! \! \! \! \! \! \! 
  \raisebox{-.6\height}{\includegraphics[scale=0.9]{malinicio.pdf}}
= 0 \,\, \,\,\, \,  \mbox{ if }  i_2 = i_1 +1
\end{equation}

\begin{equation}\label{dot-al-inicio}
  \! \! \! \! \! \! \! \! \! \! \! \!  \! \! \! \! \! \! \! \! \! \! \! \!  \! \! \! \! \! \! \! \! \! \! \! \!  \! \! \! \! \! \! \! \! \! \! \! \! 
    \raisebox{-.6\height}{\includegraphics[scale=0.9]{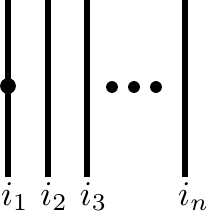}}
  = 0 \,\, \,\,\, \,
\end{equation}

\begin{equation}\label{punto-arriba}
  \! \! \! \! \! \! \! \! \! \! \! \! \! \! \! \! \!  \! \! \! \! \! \! \! \! \! \! \,\, \,\,\, \,
  \! \! \! \! \! \! \! \! \! \! \! \! \! \! \! \! \!  \! \! \! \! \! \! \! \! \! \! \,\, \,\,\, \,
  \raisebox{-.6\height}{\includegraphics[scale=0.9]{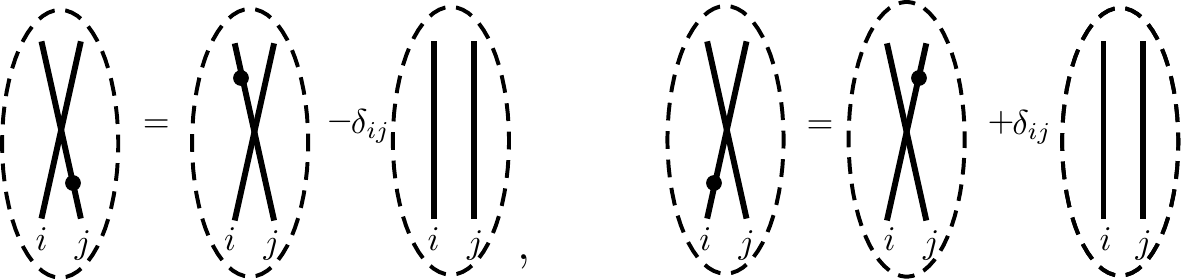}}
\end{equation}

where $ \delta_{ij} $ is Kronecker delta. Moreover we have the usual braid relation

\begin{equation}\label{KLRBraid}
  \! \! \! \! \! \! \! \! \! \! \! \! \! \! \! \! \!  \! \! \! \! \! \! \! \! \! \! \,\, \,\,\, \,
      \raisebox{-.6\height}{\includegraphics[scale=0.9]{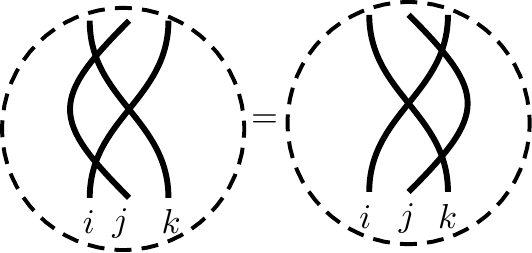}}
\end{equation}
for all $ i,j,k $ except when $ i =k = j \pm 1 $. In those cases we have that

\begin{equation}\label{KLRBraid}
  \! \! \! \! \! \! \! \! \! \! \! \! \! \! \! \! \!  \! \! \! \! \! \! \! \! \! \! \,\, \,\,\, \,
    \! \! \! \! \! \! \! \! \! \! \! \! \! \! \! \! \!  \! \! \!
          \raisebox{-.6\height}{\includegraphics[scale=0.9]{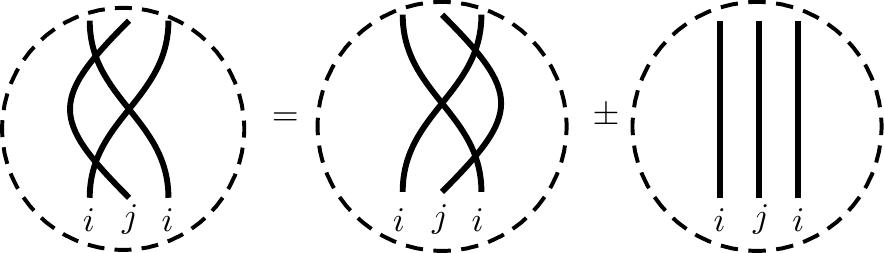}}
\end{equation}

\noindent
Finally we have the following quadratic relations

\begin{equation}\label{KLRquadratic}
  \! \! \! \! \! \! \! \! \! \! \! \! \! \! \! \! \!  \! \! \! \! \! \! \! \! \! \! \,\, \,\,\, \,
    \raisebox{-.6\height}{\includegraphics[scale=0.9]{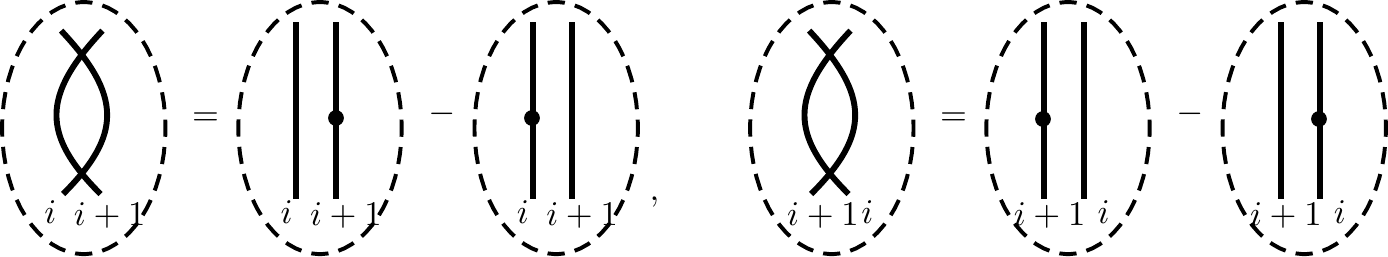}}
\end{equation}

\begin{equation}\label{KLRquadraticA}
  \! \! \! \! \! \! \! \! \! \! \! \! \! \! \! \! \!  \! \! \! \! \! \! \! \! \! \! \,\, \,\,\, \,
  \! \! \! \! \! \! \! \! \! \! \! \! \! \! \! \! \!  \! \! \! \! \! \! \! \! \! \! \,\, \,\,\, \,
        \raisebox{-.6\height}{\includegraphics[scale=0.9]{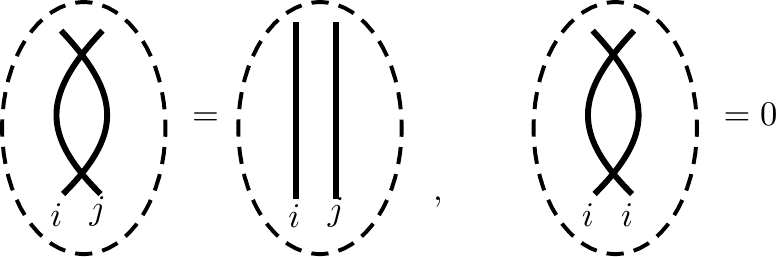}}
\end{equation}

\noindent
where $ i \neq j, j\pm 1 $. 
The identity element of  $ \B^{diag}(\kappa) $ is as follows


\begin{equation}\label{identity}
  \! \! \! \! \! \! \! \! \! \! \! \!   \! \! \! \! \! \! \! \! \! \! \! \! 
  \raisebox{-.5\height}{\includegraphics[scale=0.9]{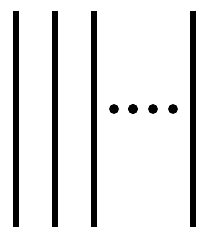}} \, \, := \, \, 
  \sum\limits_{\boldsymbol{i} \in I^n_e}
  \raisebox{-.57\height}{\includegraphics[scale=0.9]{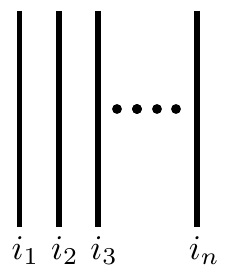}} 
\end{equation}

\noindent
with $ \bi = (i_1, \ldots, i_n) \in \II^n $.
For two diagrams $D$ and $D'$ for $\B^{diag}(\kappa)$, the multiplication $DD'$ 
is defined via vertical concatenation with $D$ on top of $D'$ if $bot(D)=top(D')$.
If $bot(D)\neq top(D')$ the product is defined to be zero.
This product is extended to all $ \B^{diag}(\kappa)$ by linearity. 

\medskip
Let $ \psi_i $, $ y_i $ and $ e({ \bi}) $ be the following elements of $ \B^{diag}(\kappa)$ (where the
upper indices refer to the positions rather than residues)



\begin{equation}\label{KLRpsi}
  \psi_i:=  \sum\limits_{\boldsymbol{i} \in I^n_e}\raisebox{-.5\height}{\includegraphics[scale=0.9]{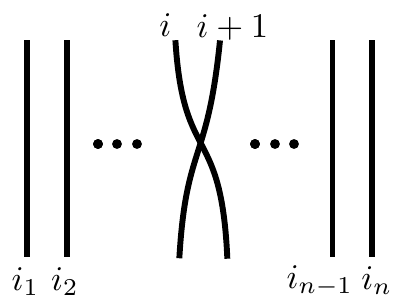}}
  \, \, \,     \, \, \,   
  y_i:=  \sum\limits_{\boldsymbol{i} \in I^n_e}\raisebox{-.5\height}{\includegraphics[scale=0.9]{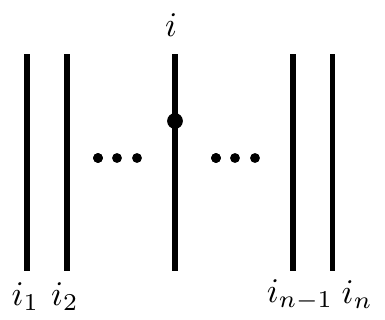}}
  \, \, \, \, \,    
  e(\bi):=  \raisebox{-.55\height}{\includegraphics[scale=0.9]{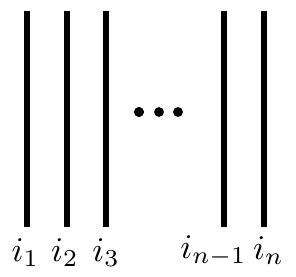}}
  \end{equation}

\noindent
with ${ \bi}  = (i_1, i_2, \ldots, i_n ) $.
Then $ \B^{diag}(\kappa)$ can also be described as the
algebra generated by the generators $ \psi_i,  y_i $ and $ e({ \bi}) $ subject to relations
\ref{mal-inicio} to \ref{KLRquadraticA} but now formulated in terms of $ \psi_i $ and $ y_i$
and $ e({ \bi}) $. In particular, by the multiplication rule for $ \B^{diag}(\kappa)$ we have that the  $ e({ \bi}) $'s
are orthogonal idempotents
\begin{equation}\label{orto}
e(\bi) e(\bj) = \delta_{ \bi, \bj } e(\bi).
\end{equation}  

{\color{black}
  \begin{remark} \rm
    Let $ R_n(\kappa) $ be the algebra with the same generators and relations as $\B^{diag}(\kappa) $,
  except relation \ref{otro-mal-inicio}, which is omitted. Then $ R_n(\kappa) $ is the cyclotomic
  KLR-algebra, see for example \cite{brundan-klesc}. Thus, $\B^{diag}(\kappa) $ is
  the quotient of $ R_n(\kappa) $ by the ideal generated by \ref{otro-mal-inicio}. 
\end{remark}}

The following Theorem is proved in \cite{PlazaRyom}. It is fundamental for the results of this section.
\begin{theorem}
  Suppose that $ 1 < m < e-1 $ and that $ \kappa = (0, m ) $. Then $ \kappa $ is an
 adjacent-free bi-charge and $  \B^{diag}(\kappa)$ is isomorphic to the blob algebra $ \B(m) $.
\end{theorem}

In view of the Theorem we simply write $ \B =   \B^{diag}(\kappa)$ in the following.
We shall from now on fix $ \kappa = (0, m ) $.

{\color{black}\label{comparisonofdefi}
  \begin{remark} \rm
    In \cite{PlazaRyom}, relation \ref{otro-mal-inicio} is formulated using the condition
    $ i_2 = i_1 -1$. On the other hand, as pointed out in Remark 1.4 of \cite{HMP}, this sign change is irrelevant.
    Indeed, let $ \B^{\prime}(m)$ be the algebra defined by the relations of \cite{PlazaRyom}. Then there is 
    an isomorphism $  \B(e-m) \cong  \B^{\prime}(m)$, induced by
    \begin{equation}
\psi_i \mapsto -\psi_i, y_i \mapsto -y_i, e(\bi) \mapsto e(-\bi).
    \end{equation}  
\end{remark}}

\medskip
We next recall the graded cellular basis for $\B$, constructed also in \cite{PlazaRyom}. 
For this we need some combinatorial notions. A \emph{one-column bipartition} of  $n$ is an ordered
pair $\blambda=(1^{\lambda_1} , 1^{\lambda_2})$ with $\lambda_1,\lambda_2 \in \mathbb{Z}_{\geq 0}$
and $\lambda_1+\lambda_2=n$. We denote by $\OnePar$ the set of all one-column bipartitions of $n$.   Given $\blambda , \bmu \in \OnePar$ we write $\blambda  \lhd \bmu$ if $|\lambda_1-\lambda_2 |>|\mu_1 -\mu_2|$. This defines a partial order on $\OnePar$. We define the \emph{Young diagram} of $\blambda$ by 
\begin{equation}
	[\blambda] = \{ (r,j) \, |\, 1\leq j\leq 2 \mbox{ and } 1\leq r\leq \lambda_j  \}.
\end{equation}
For $ j=1 $ or $ j=2 $ we refer to the elements of the form $ (r,j) $ as the \emph{$j$'th column} of $[\blambda]$
and in a similar way we define the \emph{$ r$'th} row of $[\blambda]$.
We represent graphically the elements of  $[\blambda]$ as boxes in the plane.
For instance, the Young diagram of $\blambda = (1^5,1^6) $ is depicted in \ref{examples Young}.
A \emph{tableau of shape $\blambda$} is a bijection $\T:[\blambda] \rightarrow \{1,2,\ldots , n \} $. We represent $\T$ graphically via a labelling of the boxes of $[\blambda]$ according to the bijection $\T$, that is the box $(r,j)$ is labelled with  $\T (r,j)$.

We say that $ i $ is in the $j$'th column (resp. $r$'th row) of $ \T $ if
$ \T^{-1}(i) $ is in the $j$'th column (resp. $r$'th row) of $[\blambda] $. 
We denote by $\tab (\blambda) $  the set of all tableaux of shape $\blambda$.
{\color{black}{We write $\shape(\T) = \blambda $ if $ \T \in \tab (\blambda)$.}}
A tableau is called \emph{standard} if its entries are increasing along each column. Two examples of
standard tableaux of shape $\blambda=(1^5,1^6) $ are given below

\begin{equation}\label{examples Young}
  [\blambda ] =  \left(
\raisebox{-.5\height}{\includegraphics[scale=0.9]{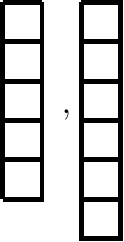}}
  \, \right) \qquad \qquad
  \s = \left(  \,
  \raisebox{-.5\height}{\includegraphics[scale=0.9]{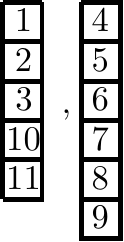}}
  \, \right) \qquad \qquad 
  \T = \left(
  \raisebox{-.5\height}{\includegraphics[scale=0.9]{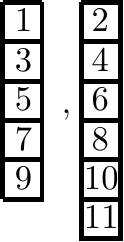}}
  \, \right)
\end{equation}

\vspace{0.3cm} 

We denote by $\std (\blambda) $ the set of all standard tableaux of shape $[\blambda ]$.
We define $\T^{\blambda}\in \std (\blambda)$ as the standard tableau in which the numbers $\{ 1,2,\ldots , n  \}$ are filled in increasingly along the rows of $[\blambda]$. For instance, if $\blambda=(1^5,1^6) $  then $\T^\blambda$ corresponds to the standard tableau $\T$ in \eqref{examples Young}. The symmetric group $\mathfrak{S}_n$ acts faithfully on the right on $\tab (\blambda)$ by permuting the entries inside a given tableau. Given $\T \in \std (\blambda )$ we define $d(\T) \in \mathfrak{S}_n$ by the condition 
$\T = \T^{\blambda} d(\T) $.
Let $s_i$ be the simple transposition $(i,i+1)\in \mathfrak{S}_n$ and let $S_n=\{ s_1,\ldots , s_{n-1}  \}$.
Then it is well known that the pair $(\mathfrak{S}_n, S_n)$ is a Coxeter system.  For $w\in \mathfrak{S}_n$ and a
reduced expression $\underline{w}=s_{i_1} \cdots s_{i_k}$ of $w$ we define 
\begin{equation}\label{psiw}
 	\psi_{\underline{w}}:= \psi_{i_1} \cdots \psi_{i_k} \in \B .
\end{equation}
In particular, for $\T \in \std (\blambda )$ we define $\psi_{d(\T)}:= \psi_{\underline{d(\T)}}$, where $\underline{d(\T)}$ is any reduced expression for $d(\T)$. It can be shown that  $\psi_{d(\T)}$ is
independent of the choice of reduced expression.


\medskip
It follows from the the relations for $ \B $ that there is a $ \Z$-grading on $ \B$ given by 
\begin{equation}
  {\rm deg}(e(\bi) ) = 0,
\, \, \,  \, \, \, \, \, \, \, \, \, \, \, \, \, \, \, 
   {\rm deg}(y_i) = 2,
\, \, \,  \, \, \, \, \, \, \, \, \, \, \, \, \, \, \,    
   {\rm deg}(\psi_je(\bi)  ) = \left\{
\begin{array}{rl}
  -2, & \mbox{ if } i_j = i_{j+1}; \\
  1,  &  \mbox{ if } i_j = i_{j+1} \pm 1; \\
  0,  &  \mbox{ otherwise. }
\end{array}   \right.
\end{equation}

It also follows from the relations for $\B$ that the reflection along a horizontal axis defines an
anti-automomorphism $ \ast $ of $\B$. It fixes the generators $ \psi_i $, $ y_i $ and $ e(\bi)$.

\medskip

For a box $A=(r,j)$ we define its \emph{residue} by $\res (A) := \kappa_j -(r-1) \in I_e$, that
is
\begin{equation}
\res (A) := \left\{ \begin{array}{rl} -(r-1) ,&  \mbox{ if } j = 1; \\  m-(r-1), & \mbox{ if } j = 2. \end{array} \right.
\end{equation}  
Given a  tableau $\T$ we define its \emph{residue sequence} by $\boldsymbol{i}^\T :=(i_1,\ldots , i_n)\in I_e^n $,
where $i_k=\res (\T^{-1}(k))$. For notational convenience we
define $\boldsymbol{i}^{\blambda}:= \boldsymbol{i	}^{\T^{\blambda}}$ for $\blambda \in \OnePar$.

\medskip
We are now in position to define the elements of the graded cellular basis for $\B$. Let $\blambda \in \OnePar$ and $\s , \T \in \std (\blambda)$. We define
\begin{equation}
	m_{\s \T}^\blambda := \psi_{d(\s)}^\ast e(\boldsymbol{i}^\blambda  ) \psi_{d(\T)}
\end{equation}

 \medskip

The following is the main result of \cite{PlazaRyom}. 

\begin{theorem}\label{Cellular}
  The set $\Basis := \{ m_{\s \T}^{\blambda} \mid \s, \T \in \std(\blambda), \blambda \in \OnePar \}$ is a graded
  cellular basis for $\B$, {\color{black}{in the sense
  of \cite{hu-mathas}}}, with respect to the order $\lhd$ and the degree function given by  $\deg (\T )= \deg{ ( m_{\T \T^\blambda  } )} $.    
  \end{theorem}

We now explain an algorithm for producing a reduced expression for the elements $d(\T)$.
This algorithm has already been used in the previous papers \cite{PlazaRyom}, \cite{HMP}, \cite{Esp-Pl} and \cite{LiPl}.


\medskip

We first need to reinterpret standard tableaux as paths on the Pascal triangle.
This is a generalization of the correspondence,
explained in the paragraph prior to Figure \ref{fig:two walks}, 
between usual two-column standard tableaux
and walks in a coordinate system. 
Let $\T \in \std (\blambda )$.  Then we define $p_{\T}:\{0,1,\ldots ,n\}\rightarrow \mathbb{Z} $
as the function given recursively by $p_{\T } (0)=0$ and $p_\T (k)= p_\T (k-1) +1$ (resp. $p_\T (k)= p_\T (k-1) -1$) if $k$ is located in the second (resp. first) column of $\T$. Moreover, we define $P_\T: [0,n] \rightarrow {\mathbb R}^2 $
as the piecewise linear path such that $ P_\T(k) = (p_\T (k),k) $ for $k =0,1,\ldots, n $ and 
such that $ P_\T\! \mid_{ [k,k+1]} $ is a line segment for all $ k=0,1,\ldots, n-1$.

\medskip
We depict $ P_\T$ graphically inside the standard two-dimensional coordinate system, but reflected through the $x$-axis.
For instance, if $ \s $ and $\T$ are the standard tableaux in \eqref{examples Young} then $ P_\s$ and
$P_{\T} $ are depicted in \eqref{example path beginning}, with $ P_\s$ in red and $ P_\T $ in black.
In general, we denote by $P_{\blambda}$ the path obtained from the tableau $\T^{\blambda}$.
Thus in \ref{example path beginning} we have that $ P_\T = P_{\blambda}$ for $ \blambda = ( 1^5, 1^6) $.

\begin{equation}  \label{example path beginning}
 \raisebox{-.6\height}{\includegraphics[scale=0.9]{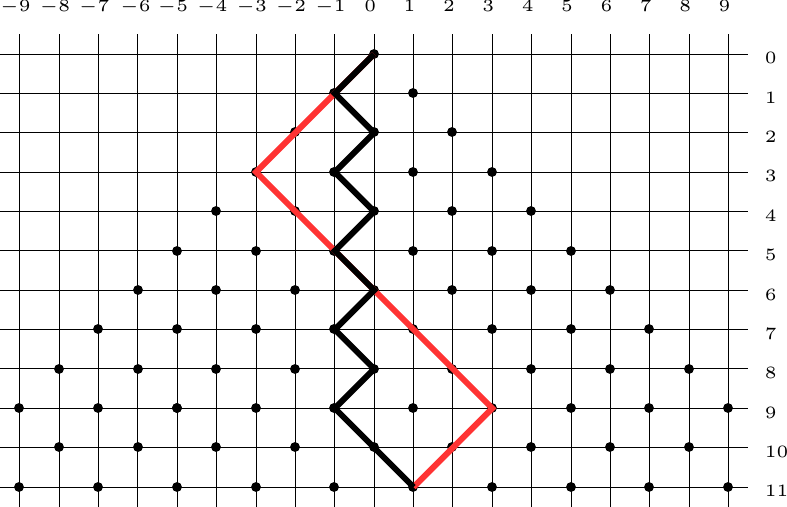}}
\end{equation}

Note that in general the integral values of $ P_\T $ belong to the set $ \{ (p, k) \mid k \in  {\mathbb Z}_{\ge 0} , \, \,
p=  -k, -k+2, \ldots, k-2, k   \} $. This set has a Pascal triangle structure which is 
why we say that standard tableaux correspond to paths on the Pascal triangle. 

\medskip

It is clear that the map $ \T \mapsto P_\T $  defines a bijection
between $\std (\blambda)$ and the set of all such piecewise linear paths
with final vertex $(\lambda_2-\lambda_1, n )$.
For this reason, we sometimes identify $\blambda$ with the
{\color{black}{point
$(\lambda_2-\lambda_1,n)$.}}

\medskip
Suppose now that both $ \T $ and $  \T s_k$ are standard tableaux for some $ \blambda \in \OnePar $ and
$ s_k \in S$. Then $ k $ and $ k+1 $ are in different columns of $ \T$ and so we conclude that the functions 
$ p_\T $ and $ p_{\T s_k } $ are equal except that $ p_\T(k) = p_{\T s_k}(k) \pm 2$, and hence also the
paths $ P_\T $ and $ P_{\T s_k } $ are equal except in the interval $[k-1, k+1] $ where they are related 
in the following two possible ways

\begin{equation}  \label{hooks}
  \raisebox{-.6\height}{\includegraphics[scale=0.9]{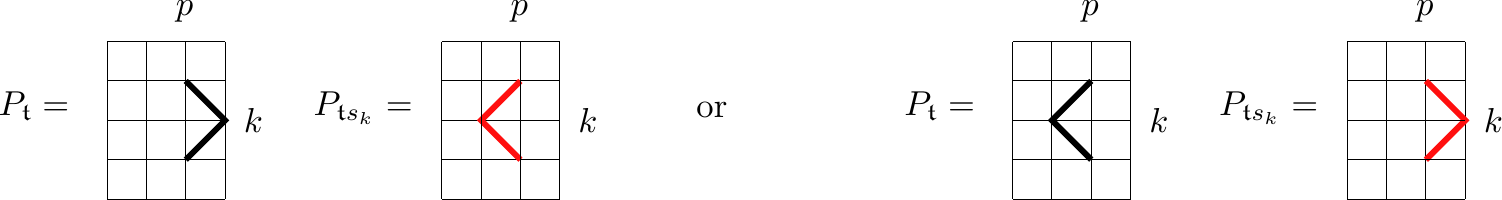}}
\end{equation}

\noindent
Conversely, if $ \s$ and $ \T $ are standard tableaux in $ \std(\blambda) $
such that $ P_\s$ and $ P_\T$ are equal except in the interval $[k-1, k+1] $ where they are related as in
\eqref{hooks}, then we have that $ \s = \T s_k$.  Let us now consider the following algorithm.

\begin{algorithm}\rm  \label{algorithm paths}
  Let $\blambda \in \OnePar $ and $ \T \in \std (\blambda )$.
  Then we define a sequence  $ seq:=(s_{i_1} ,s_{i_2} ,\ldots, s_{i_N} )  $ of elements of $ S_n $ as follows.
	\begin{description}
		\item[Step 1.] Set $P_0 := P_{\blambda }$.
                  If $ P_0 \neq P_\T$ then choose $ i_1 $ any such that $ \T^\blambda s_i \in \std(\blambda)  $ and such that
                  the area bounded by
                  $ P_1 :=P_{\T^\blambda s_i} $ and $ P_\T $ is strictly smaller than the area bounded by $ P_0 $ and $ P_\T$. 
                 \item[Step 2.] If $ P_1 = P_\T $ then the algorithm stops
                   with $ seq:=(s_{i_1})$. Otherwise choose any $ {i_2} $ such
                   that $ \T^\blambda s_{i_1} s_{i_2} \in \std(\blambda) $ and such that the area bounded by
                   $ P_2 := P_{\T^\blambda s_{i_1} s_{i_2}} $ and $ P_\T $ is strictly
                   smaller than the area bounded by $ P_1 $ and $ P_\T$.
                 \item[Step 3.] If $ P_2 = P_\T $ then the algorithm stops
                   with $ seq:=(s_{i_1}, s_{i_2})$. Otherwise choose any $ {i_3} $
                   such that $ \T^\blambda s_{i_1} s_{i_2} s_{i_3} \in \std(\blambda) $ and such that the area bounded by
                  $ P_3 := P_{\T^\blambda s_{i_1} s_{i_2} s_{i_3} } $ and $ P_\T $ is strictly smaller than the area
                  bounded by $ P_2 $ and $ P_\T$.
                \item[Step 4.] Repeat until $ P_N = P_\T$. The resulting sequence
                  $ seq = (s_{i_1}, s_{i_2}, \ldots, s_{i_N} ) $ gives
                  rise to a reduced expression for $ d(\T) $ via $  d(\T)= s_{i_1} s_{i_2} \cdots s_{i_N}$.
        \end{description}
\end{algorithm}


\medskip
\noindent
Note that it follows from \ref{hooks} that the $ i_k$'s in $ {\rm \bf Step \, \, 2}  $ and 
$ {\rm \bf Step \, \, 3}  $ do exist and so the Algorithm \ref{algorithm paths} makes sense.
For example in the case of the tableau $ \s $ from \ref{examples Young} we get, using
\eqref{example path beginning}, that for example 
\begin{equation}
d(\s)  = s_2 s_4s_3  s_7 s_9 s_8 s_{10} s_9
\end{equation}
is a reduced expression for $d(\s) $.
{\color{black}{For completeness, we now present a proof of the correctness of the Algorithm.}}
\begin{theorem}\label{Algorithm}
Algorithm \ref{algorithm paths} computes a reduced expression for $d(\s ) $.
\end{theorem}  
\begin{dem}
This is a statement about the symmetric group $ \Si_n $ viewed as a Coxeter group.
Let $ \T_k := \T^{\blambda} s_{i_1} s_{i_1} \cdots s_{i_k} $ be the tableau constructed after
$ k$ steps of the algorithm. Then we have that $d(\T_k) = s_{i_1} s_{i_1} \cdots s_{i_k} $ 
and we must show that $ l(s_{i_1} s_{i_1} \cdots s_{i_k}) = k $ where $l(\cdot) $ is the length function for $ \Si_n$.
We therefore identify $ d(\T_k) $ with a permutation of $ \{1,2,\ldots,n\}  $ via the row reading for $ \T_k $. 
To be precise, using the usual one line notation for permutations, we write

\begin{equation}\label{bijection}
      \raisebox{-.6\height}{\includegraphics[scale=0.9]{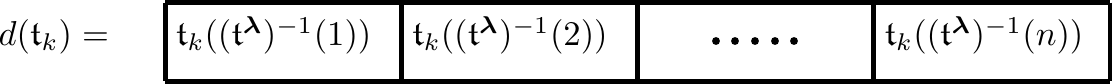}}
\end{equation}

\noindent
We call this the \emph{one line representation} for $ d(\T_k)$.
If for example $ \T_k = \s $ from \ref{examples Young} then we have the following one line representation
for $ d(\T_k) $

\begin{equation}\label{bijection}
    \raisebox{-.6\height}{\includegraphics[scale=0.9]{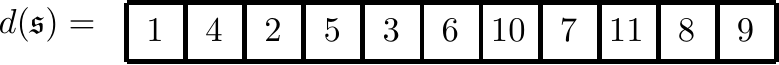}}
\end{equation}

\noindent

\noindent
whereas for $ \T_k = \T^{\blambda} $ from \ref{examples Young} we have the identity one line representation, that is

\begin{equation}\label{bijectionId}
  \raisebox{-.6\height}{\includegraphics[scale=0.9]{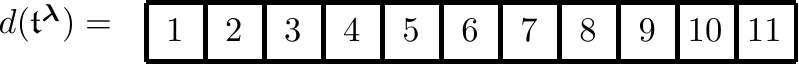}}
\end{equation}

\noindent
In general, by the Coxeter theory for $ \Si_n $, we have that $ l( d(\T_k )) $
is the number of \emph{inversions} of the one line representation of $ d(\T_k ) $ that is
\begin{equation}
 l( d(\T_k )) = inv( d(\T_k )) := | \{ (i, j ) : i< j \mbox{ and } \T_k (  ( \T^{\blambda})^{-1}(i)  >
\T_k (  ( \T^{\blambda})^{-1}(j)   \} |
\end{equation}
To prove the Theorem we must now show that $ inv( d(\T_k )) = k $. 
We proceed by induction on $ k $. For $ k= 0 $ we have that
$inv( d(\T_k )) =  inv( d(\T^{\blambda} )) = 0 $, see
\ref{bijectionId}, and so the induction basis is ok.
We next assume that $ inv( d(\T_{k-1} ))  = k-1 $ and must show that $ inv( d(\T_{k} ))  = k $.
At step $ k $ of Algorithm \ref{algorithm paths},
we have that $ \T_{k-1},  \T_{k}  \in \std(\blambda) $  
and $ \T_{k-1} s_{i_k} = \T_{k} $ and hence
$ \T_{k-1} $ and $ \T_{k} $ are in one of the two situations
described in \ref{hooks}. Let $ p $ be as in \ref{hooks}. Then, since $ \T_{k} $ is closer to $ \T $ than $ \T_{k-1} $, 
we have that $ \T_{k-1} $ and $ \T_{k} $ are in the first situation of \ref{hooks}
if $ p \le -1 $ and in the second situation of \ref{hooks} if $ p \ge 0 $. In other words,
the first situation of \ref{hooks} only takes places in the left half of the Pascal triangle \eqref{example path beginning}
and the second situation of \ref{hooks} only takes places in the right half of the Pascal
triangle \eqref{example path beginning},
with the vertical axis $ p=0$ is included.

\medskip

These two situations translate into the following two possible relative positions for $ k $ and $ k+1 $ in
$ \T_{k-1}$.

\begin{equation}\label{examplesORDER}
 \left(  \,
\raisebox{-.5\height}{\includegraphics[scale=1]{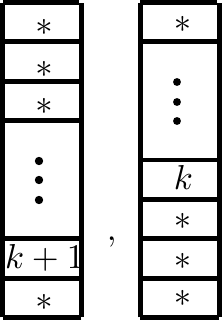}}
\, \right) \, \, \, \, \, \, \, \, \, \, \, \, \, \, \, \, \, \, \, \, \, \, \, \, 
\left(  \,
\raisebox{-.5\height}{\includegraphics[scale=1]{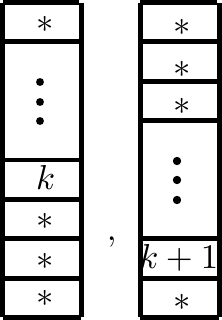}}
\, \right)
\end{equation}

Here, in both tableaux $k$ and $ k+1 $ are in different columns, but 
in the first tableau, corresponding to $ p < 0 $, we have that $ k+1 $ is in a strictly lower
{\color{black}{row}}
than $ k $, whereas in the second tableau, corresponding to $ p \ge 0 $, we have that $ k+1 $ is in a lower
or equal {\color{black}{row}} than $ k$.

\medskip
On the other hand, in each of the two cases of \ref{examplesORDER} we have that $ k $ appears before
$ k+1 $ in the one line representation for $ \T_{k-1}$ and so $ inv( d(\T_{k} )) = inv( d(\T_{k-1} )) +1$.
This proves the Theorem.
\end{dem}

{\color{black}{\begin{remark} \rm
We remark that the reduced expression for $d(\s)$ obtained via
Algorithm \ref{algorithm paths} is by no means unique. In general,
we have many choices for the $i_k$'s and the reduced expression obtained
depends on the choices we make. On the other hand, it is known that
$d(\s)$ is fully commutative.
In other words, any two reduced expressions 
for $d(\s)$ are related via the commuting braid relations.
\end{remark}}}



\subsection{Idempotent truncations of $\mathbb{B}_n$}
From now on we shall study a certain subalgebra of $\B $ that arises from idempotent truncation of $ \B$.
This subalgebra has already appeared in the literature, for example in  \cite{Esp-Pl}, \cite{LiPl}.

\begin{definition}\rm
	Suppose that $\blambda  \in \OnePar$. Then the subalgebra $ \trunc $ of $ \B $ is defined as 
\begin{equation}
  \trunc := e( \bi^{\blambda } ) \B e( \bi^{\blambda}  ).
\end{equation}
\end{definition}


\noindent
Let us mention the following Lemma without  proof.
\begin{lemma} \label{lemma super reduction}
  Let $\blambda = (1^{\lambda_1},1^{\lambda_2} ) \in \OnePar$.
Set $ \bmu := ( 1^{\lambda_2},1^{\lambda_1} ) \in \OnePar$ and 
$\bnu= (1^{\lambda_1-M}, 1^{\lambda_2-M}) \in \mbox{Par}_{2,n-2M}^1$
where $ M=\min \{ \lambda_1 , \lambda_2  \}$.
  There is an isomorphism $\trunc \cong \mathbb{B}_{n-2M}(\bnu )$ of $\F $-algebras.
\end{lemma}




We shall from now on fix  
$ \blambda $ of the form
\begin{equation}\label{fixlambda} 
\blambda = (1^{n}, 1^0).
\end{equation}

{\color{black}{
 \begin{remark} \rm
When defining $ \trunc $ we could have taken more general $ \blambda$, but in view of the Lemma it is enough to
consider $ \blambda $ either of the form $(1^{n}, 1^0) $ or $\bmu:=(1^{0}, 1^n) $. Moreover,
using the notation and isomorphism of Remark \ref{comparisonofdefi} we have that
\begin{equation}\label{analogues}
 e( \bi^{\bmu } ) \B e( \bi^{\bmu}  ) \cong e( \bi^{\blambda } ) \B^{\prime}(e-m) e( \bi^{\blambda})  .
\end{equation}
On the other hand, the methods and results for $ \trunc $ that we shall develop during the rest of
the paper will have almost identical analogues for
the right hand side of \ref{analogues}, as the reader will notice during the reading, 
with the only difference that one-column bipartitions
and tableaux are replaced by one-row bipartitions
and tableaux. Thus, there is no loss of generality in assuming that $ \blambda $ is of the form
given in \ref{fixlambda}.
 \end{remark} }}

One of the advantages of the choice of $ \blambda $ in \ref{fixlambda} is that the residue sequence
$ \bi^\blambda $ is particularly simple since it decreases in
steps by one.
Let us state it for future reference
\begin{equation}\label{lambdaresidue} 
\bi^\blambda = (0,-1,-2,-3, \ldots, -n+1) \in \II^n.
\end{equation}

In the main theorems of this section we shall find generators 
for $\trunc $, verifying the same relations as the generators 
${\mathbb{NB}}_n$ or $\widetilde{\mathbb{NB}}_n$. 
The following series of definitions and recollections of known results from the literature
are aimed at introducing these generators.

\medskip

It follows from general principles that $ \trunc $ is a graded cellular algebra with identity element 
$ e( \bi^{\blambda } )$.  Let us describe the corresponding cellular basis.
Set first $ \std( \OnePar) := \bigcup_{ \bmu \in \OnePar}  \std(\bmu)  $ and define for $ \bi \in \II^{n} $: 
\begin{equation} 
 \std(\bi):=\{\bT\in\std(\OnePar) \mid \bi^{\T} = \bi \}. 
\end{equation}
Furthermore, for $  \bmu \in \OnePar $ define 
\begin{equation}\label{commutationRule}
  \std_{\blambda}(\bmu  ) := \std(\bi^{\blambda}) \cap \std (\bmu). 
\end{equation}
Then we have the following Lemma.

\begin{lemma}\label{isacellularbasis}

\begin{description}
	\item[a)]  For $ \s, \T \in \std(\bmu) $ we have that 
  \begin{equation}
    e(  \bi^\bmu) \psi_{ d(\T)} = \psi_{ d(\T)}e( \bi^\T)
     \, \, \, \, \,  \mbox{      and     }  \, \, \, \, \, \,
 \psi_{ d(\s)}^{\ast} e(\bi^\bmu) = e( \bi^\s) \psi_{ d(\s)}^{\ast}.
  \end{equation}  
\item[b)] The set $ \Basis(\blambda):=\{ m_{\s \T}^{\bmu} \mid \s, \T \in
 \color{black}{ \std(\bi^\blambda  ), \bmu = \shape(\s) = \shape(\T) }  \}$
  is a graded cellular basis for $ \trunc$.
\end{description}
\end{lemma}
\begin{proof}  
  From the multiplication rule in $ \B$ we have that $ \psi_k e( \bi ) =  e(s_k  \bi ) \psi_k $ for any
  $k=1,\ldots, n-1   $ and $ \bi \in \II^n$. Hence if $ d( \T) = s_{i_1} s_{i_2} \cdots s_{i_N} $ is a reduced expression 
  we get that 
  \begin{equation}  
   e(\bi^{\bmu})  \psi_{ d(\T)} =   \psi_{ d(\T) }e ( s_{i_N} \cdots s_{i_2} s_{i_1} \bi^{\bmu} )
    = \psi_{ d(\T) } e ( \bi^{\T }),
  \end{equation}
  proving the first formula of $a)$. The second formula of $a)$ is proved the same way. On the other hand, by using $ a) $ and \eqref{orto} we obtain
\begin{equation}
  e( \bi^\blambda)  m_{\s \T}^{\bmu} e( \bi^\blambda) =
  e( \bi^\blambda)  \psi_{ d(\s)}^{\ast}  e( \bi^\bmu) \psi_{ d(\T)}  e( \bi^\blambda)  =
  e( \bi^\blambda) e( \bi^{ \s})  \psi_{ d(\s)}^{\ast}   \psi_{ d(\T)}  e( \bi^{ \T})   e( \bi^\blambda)
  = \delta_{ \bi^\s, \bi^{\blambda}} \delta_{ \bi^\T, \bi^{\blambda}} m_{\s \T}^{\bmu}
\end{equation}  
and so $ b) $ follows. 
%
%
\end{proof}

We now introduce  
an $ \tilde{A}_{1}$ alcove geometry on $ {\mathbb R}^2$.
For each $ j  \in \mathbb Z $ we introduce a 
\emph{wall} $ M_j $ in $ {\mathbb  R}^2 $ via
\begin{equation}  
M_j := \{ ((j-1)e+m, a) \mid a \in {\mathbb  R} \} \subset {\mathbb  R}^2.
\end{equation}   

The connected components of $\mathbb{R}^2 \setminus \bigcup_j M_j  $ are called \emph{alcoves} 
and the alcove containing $(0,0)$ is denoted by $\mathcal{A}^0$ and is called the \emph{fundamental alcove}.
Recall that we have fixed $W$ as the infinite
dihedral group with generators $\ese $ and $\te $.
We view $ W$ as the reflection group associated with this alcove geometry, where
$\ese$ and $\te $ are the reflections through the walls $ M_0 $ and $ M_1$, respectively.
This defines a right action of $W$ on $\mathbb{R}^2 $ and on the set of alcoves.
For $w\in W$, we write $\mathcal{A}^w :=  \mathcal{A}^0 \cdot w$.




\medskip
Let $P: [0,n] \rightarrow {\mathbb R}^2 $ be a path on the Pascal
triangle and suppose that $ P(k) \in M_j  $ for some integers $ k $ and $j$.
Let $ r_j $ be the reflection through the wall $ M_j $. We then define a new path $ P^{(k,j)} $ by
applying $ r_j $ to the part of $ P $ that comes after $ P(k) $, that is 
\begin{equation}
 P^{(k,j)}(t) := \left\{  \begin{array}{ll} P(t), & \mbox{ if } 0 \le t \le k; \\
     P(t) r_j, & \mbox{ if } k \le t \le n . \end{array} \right.
\end{equation}
For two paths on the Pascal triangle we write $ P \stackrel{(k,j)}{\sim} Q$ if $  Q = P^{(k,j)} $ 
and denote by $ \sim$ the equivalence relation on the paths on the
Pascal triangle induced by the $ \stackrel{(k,j)}{\sim} $'s. 
Then we have the following Lemma which is a straightforward consequence of the 
definitions. 

\begin{lemma} \label{lemma equivalence classes}
  Suppose that $ \s, \T \in \std( \OnePar) $. Then 
  $\bi^\s = \bi^\T$ if and only if $P_{\s} \sim P_{\T}$. 
\end{lemma}

We can now provide an alcove geometrical description of $\std (\bi^{\blambda})$.
It is  a direct consequence of Lemma \ref{lemma equivalence classes}.


\begin{lemma}  \label{lemma description singular}
Let $[P_\blambda]$ be the equivalence class of $ P_{\blambda } $ under the equivalence relation $ \sim$. Then, 
$ \std(\bi^\blambda) = [ P_{\blambda}] $.
\end{lemma}


\begin{figure}[h]
\centering
  \raisebox{-.6\height}{\includegraphics[scale=0.9]{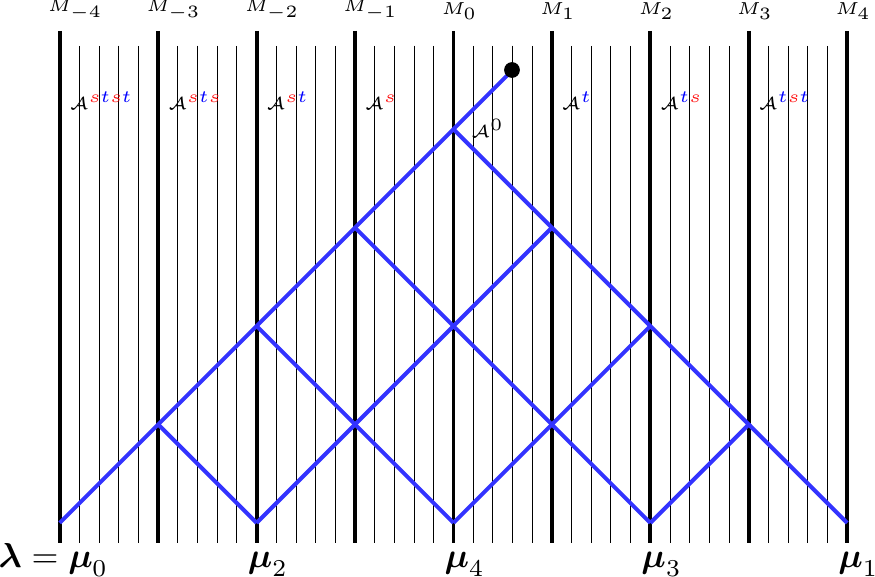}}
\caption{The Pascal triangle for $ m=2, e = 5 $ and $ n = 23$.}
\label{fig:pascalAA}
\end{figure}

In Figure \ref{fig:pascalAA} we indicate for $ m=2, e = 5 $ and $ n = 23$ the
paths corresponding to elements in $ \std(\bi^\blambda)$, according to Lemma \ref{lemma description singular}. The path $ P_{\blambda} $ is
the one to the extreme left.
The endpoints of the paths are enumerated according to the order relation $\lhd$ on
$\OnePar$, 
with $ \bmu_0 = \blambda$,  
$ \bmu_1 $ the rightmost path, and so on.

\vspace{0.2cm}
To illustrate the connection between
paths and tableaux, we present in Figure \ref{fig:standardpaths} the six elements of $ \std_{\blambda}(\bmu_4) $
for Figure \ref{fig:pascalAA} as tableaux. 
We have here colored the entries of each tableau according to
the \emph{path intervals} to which they belong.
The \emph{zero'th path interval} corresponds to the path segment from the origin $ (0,0) $ to the first wall $ M_0$
and its entries have been colored red. The first \emph{full path interval} 
corresponds to the path segment from $M_0$ to the next wall
which may be either $ M_{-1}$ or $ M_{1}$ depending on the tableau and the corresponding elements have
been colored blue, and so on. We shall give the precise definition of full path intervals shortly.

In Figure \ref{fig:standardpaths}
we have also given the \emph{residue tableau} $ {\rm res }\, \bmu_4 $
for $ \bmu_4$. By definition, it is obtained from $ [\bmu_4] $ by decorating each node $ A $
with its residue $ {\rm res}(A)$.
Using it, one checks that for each $ \T\in \std_{\blambda}(\bmu_4) $
the corresponding residue sequence is $ \bi^{\blambda} $, as it should be: 
\begin{equation}
\bi^{\blambda}= \bi^{\T}=(0,4,3,2,1,0,4,3,2,1,0,4,3,2,1,0,4,3,2,1,0,4,3,2,1)
\end{equation}


\begin{figure}[h]
 $\std_{\blambda}(\bmu_4) =  \left(
\raisebox{-.5\height}{\includegraphics[scale=1]{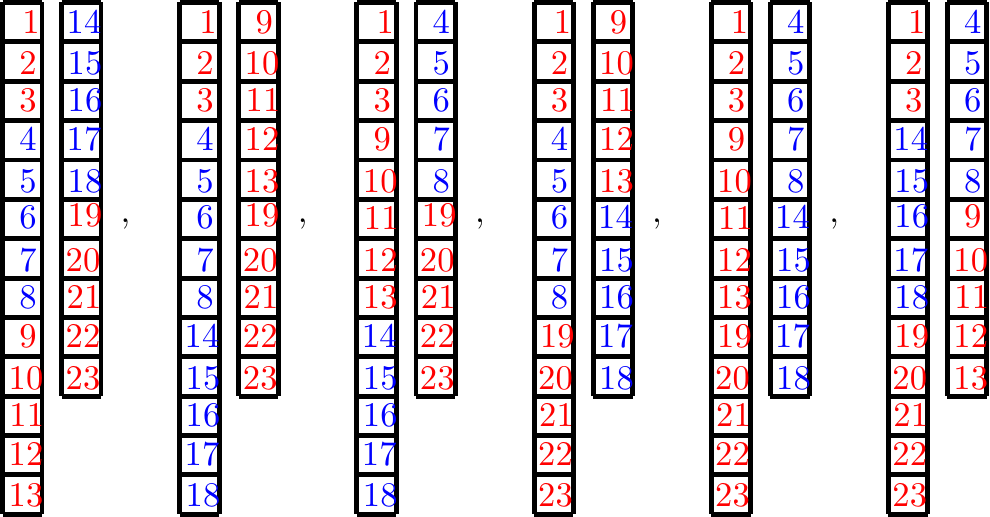}}
  \right) \! \!, \, \,
        [{\rm res} \, \bmu_4] =
        \left(
\raisebox{-.5\height}{\includegraphics[scale=1]{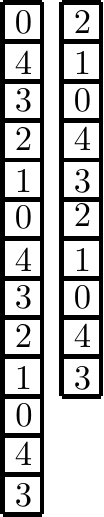}}  \right)$
\caption{The six elements of $ \std_{\blambda}(\bmu_4) $
for Figure \ref{fig:pascalAA} as tableaux.}
\label{fig:standardpaths}
\end{figure}


\begin{figure}[h]
\centering
  \raisebox{-.6\height}{\includegraphics[scale=0.9]{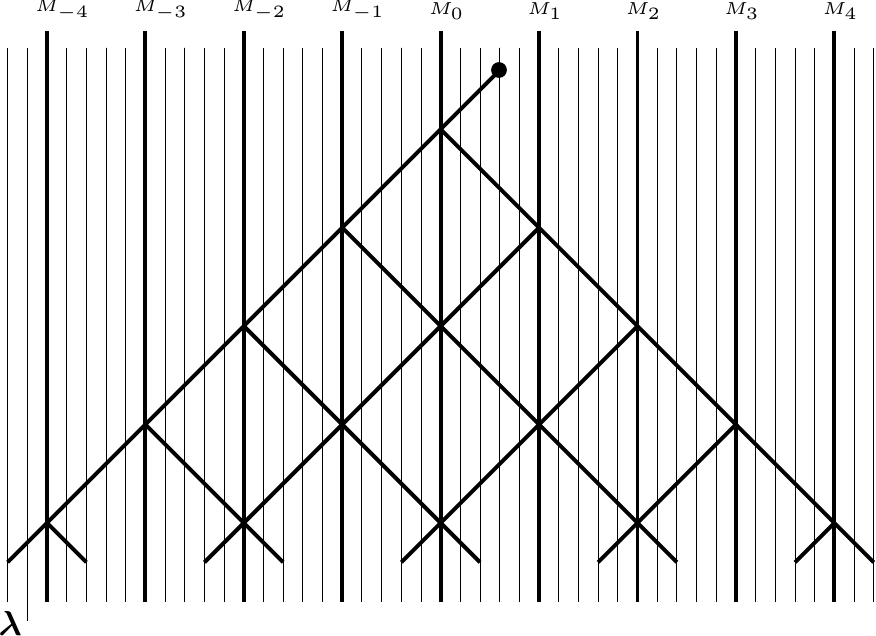}}
  \caption{The regular situation.}
\label{pascalAAA}
\end{figure}


The structure of $  \std(\bi^\blambda) $ depends on whether $ \blambda $ is \emph{singular or regular}:

\begin{definition}\label{definition epsilon}
Let the integers $K_{n,m}=K $ and $0\leq R_{n,m}=R<e$ be defined via integer division 
$ n -(e-m) = Ke+R$. Then we say that $\blambda $ is singular if $R=0$ and otherwise we say $\blambda $ that is regular.
  Graphically, $\blambda$ is singular if it is located on a wall, otherwise it is regular.
\end{definition}

The paths in Figure \ref{fig:pascalAA} represent a singular situation whereas the paths in
Figure \ref{pascalAAA} represent
a regular situation. In both cases, regular or singular, the cardinality $ \mid \! \! \std_{\blambda}(\bmu) \! \! \mid $
is given by binomial coefficients {\color{black}{and so we have the following Lemma}}.
{\color{black}{
\begin{lemma}  \label{lemma description singular}
  \begin{description}
  \item[a)]  Let $[P_\blambda]$ be the equivalence class of $ P_{\blambda } $ under
    the equivalence relation $ \sim$. Then, 
  $ \std(\bi^\blambda) = [ P_{\blambda}] $.
\item[b)]  Suppose that $ \blambda $ is singular. Then $ \sum_{\bmu \in [P_{\blambda}](n)} \mid \! \! \std_{\blambda}(\bmu) \! \! \mid^2 = \binom{2K}{K}$.	
        \item[c)]  Suppose that $ \blambda $ is regular. Then
          $ \sum_{\bmu \in [P_{\blambda}](n)} \mid \! \! \std_{\blambda}(\bmu) \! \! \mid^2 = 2\binom{2K}{K}$.	
 \end{description}          
\end{lemma}}}

\medskip

We now define the integer valued function 
\begin{equation}\label{functionblocks}
f_{n,m}(j) = f(j):= -m +je  \mbox{ for }  j \in \Z_{+}.
\end{equation}  
Then for $ \T \in  \std_{\blambda}(\bmu) $ we have that 
$ k=f(1), f(2), \ldots,f(K) $ are the values of $ k $ such that $ P_{\T}( k) $ belongs to
a wall $ M_j $ and we then define for $ i =1,2, \ldots, K$
the $i$'th \emph{full path interval for $ \blambda $} as the set 
\begin{equation}\label{DEFfullBlock}
 B_i :=[ f(i)+1, f(i)+2, \ldots, f(i)+e].
\end{equation}
For example, in the situations of Figure \ref{fig:pascalAA} and \ref{pascalAAA} we have the following full
path interval
\begin{equation}\label{fullblocks}
B_1=[4,5,6,7,8], \, B_2=[9,10,11,12,13], \, B_3=[14,15,16,17,18], \, B_4=[19,20,21,22,23].
\end{equation}

For $ 1 \le i < K $ we next define $ U_i \in \Si_n $ as the order preserving permutation that
interchanges the path intervals $ B_i $ and $ B_{i+1} $ that is
\begin{equation}\label{Si}
U_i := (f(i)+1, f(i+1)+1    ) \, (f(i)+2, f(i+1)+2 ) \, \cdots \, ( f(i)+e, f(i+1)+e).
\end{equation}
For example, in the situation \ref{fullblocks} we have 
\begin{equation}
U_1 = (4,9)(5,10) (6,11) (7,12) (8,13)
\end{equation}
written as a product of non-simple transpositions.
We need a reduced expression for \ref{Si} and therefore
for $ i \le j $ of the same parity we introduce the following element of $ \Si_n$ 
\begin{equation}\label{sameparity}
  s_{[i,j]} := s_i s_{i+2} \cdots s_{j-2}s_j.
\end{equation}
Then we have
\begin{equation}
U_i = s_{[a,a]}  s_{[a-1,a+1]}   \cdots
   s_{[a-e+1,a+e-1]} \cdots s_{[a-1,a+1]} s_{[a,a]} 
\end{equation}  
where $ a = f(i+1)$ which upon expanding out the $ s_{[i,j]}$'s becomes a reduced expression for $ U_i $.
We can now recall the following important definition
from \cite{LiPl}. 
\begin{definition}\rm 
  For $ 1 \le i < K $ we define the \emph{diamond} of $\blambda$ at position $f(i)$
  by
\begin{equation}  
  U^{\blambda}_i:= \psi_{U_i}e(\bi^{\blambda})  = \psi_{[a,a]}  \psi_{[a-1,a+1]}   \cdots
   \psi_{[a-e+1,a+e-1]} \cdots \psi_{[a-1,a+1]} \psi_{[a,a]} e(\bi^{\blambda})
\end{equation}
where $ a = f(i+1)$ and $ \psi_{[i,j]} := \psi_i \psi_{i+2} \cdots \psi_{j-2} \psi_j$.
\end{definition}

The name `diamond' comes from the diagrammatic realization of $ \trunc$. Here is for example
the $n=13,m= 2, e=5 $ and $  i=1  $ case 

\begin{equation}
  U^{\blambda}_1  = 
\raisebox{-.5\height}{\includegraphics[scale=0.6]{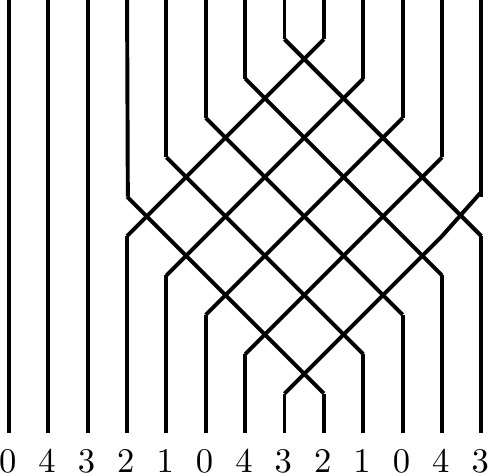}}
\end{equation}


In this section we consider the case where $\blambda $ is singular.
Our aim is to show that $\trunc$ and $\mathbb{NB}_{K} $ are isomorphic $\F$-algebras. 
The first step towards this goal is
to prove that the following subset of $ \trunc$  
\begin{equation}\label{genG}
  {G}(\blambda)
  :=  \{ U_j^\blambda  \, |\, 1\leq j < K  \}	 \cup \{ y_ie(\bi^{\blambda}) \, |\, 1\leq i \leq n \}
\end{equation}
is a generating set for $ \trunc$. 
To be precise, letting $ \truncPrime $ be the subalgebra of $ \trunc$ generated by
${G}(\blambda) $ we shall show that each element
$ m_{\s \T}^{\bmu}$ of the cellular basis $ \Basis(\blambda) $ for $ \trunc $, given in
Lemma \ref{isacellularbasis}, belongs to $ \truncPrime $. The proof of this will take up the next few
pages. 

\medskip
We shall rely on a systematic way of applying
Algorithm \ref{algorithm paths} to get 
reduced expressions for the elements 
 $ d(\T)$, $\T \in \std (\bi^\blambda)$. Let us now explain it.

\medskip
Let $\blambda_{\max} \in \OnePar$ be the maximal element in the $W$-orbit of $\blambda$ with respect to the order $\lhd$. Clearly, $\blambda_{\max}$ is located on one of the two walls of the fundamental alcove.
{\color{black}{Recall that $ P_{\blambda_{\max}} $ is the path associated with the tableau $ \T^{\blambda_{\max}}$; it 
    zigzags along the vertical central axis of the Pascal triangle as long as possible, and finally goes linearly off to
    $ \blambda_{\max}$.}}
The set of paths $ P_{\T} $ for $ \T \in \std(\bi^{\blambda})$ together with $ P_{\blambda_{\max}} $,
{\color{black}{which does \emph{not} belong to $\std(\bi^{\blambda}) $}}, 
determine three kind of bounded regions 
that we denote by $ h_i, u_i $ and $ u_i^{\prime}$: 


\begin{equation}\label{regions}
  \raisebox{-.6\height}{\includegraphics[scale=0.7]{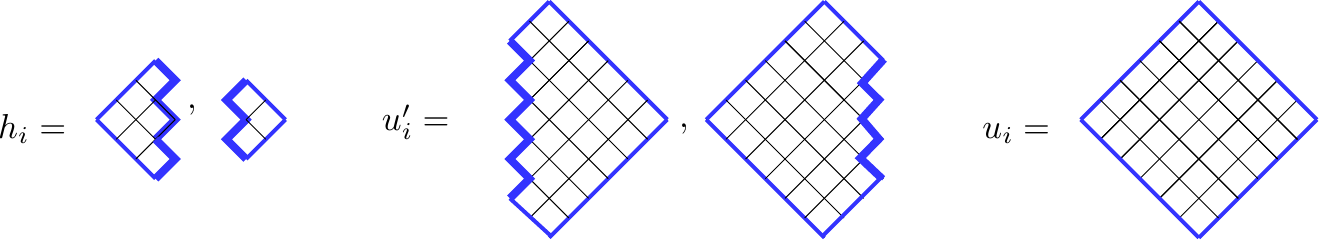}}
\end{equation}

\medskip
\noindent
See also Figure 
\ref{fig:cutdiagrams}.
{\color{black}{In \ref{regions} as well as Figure \ref{fig:cutdiagrams} we have indicated
    $ P_{\blambda_{\max}} $ with bold blue.}} 
 
In general the $ h_i$'s are completely embedded in $ {\mathcal A }^0$,
whereas the `diamond' regions $ u_i$'s have empty intersection with $ {\mathcal A }^0$. The  
`cut diamond' regions $ u_i^{\prime}$'s have 
non-empty intersection with $ {\mathcal A }^0$ but also with one of the alcoves $ {\mathcal A }^{\ese}$ or
$ {\mathcal A }^{\te}$. Note that the union of $ h_i $ and $   u_i^{\prime} $ forms a diamond shape. 
We enumerate the regions from top to bottom as in \ref{fig:cutdiagrams},
with the $ h_i$'s starting with $ i=0 $ and the $  u_i^{\prime} $ and $ u_i$'s with $ i=1$. 
Note that there are repetitions of the $ u_i$'s.


\begin{figure}[h]
\centering
\raisebox{-.6\height}{\includegraphics[scale=0.7]{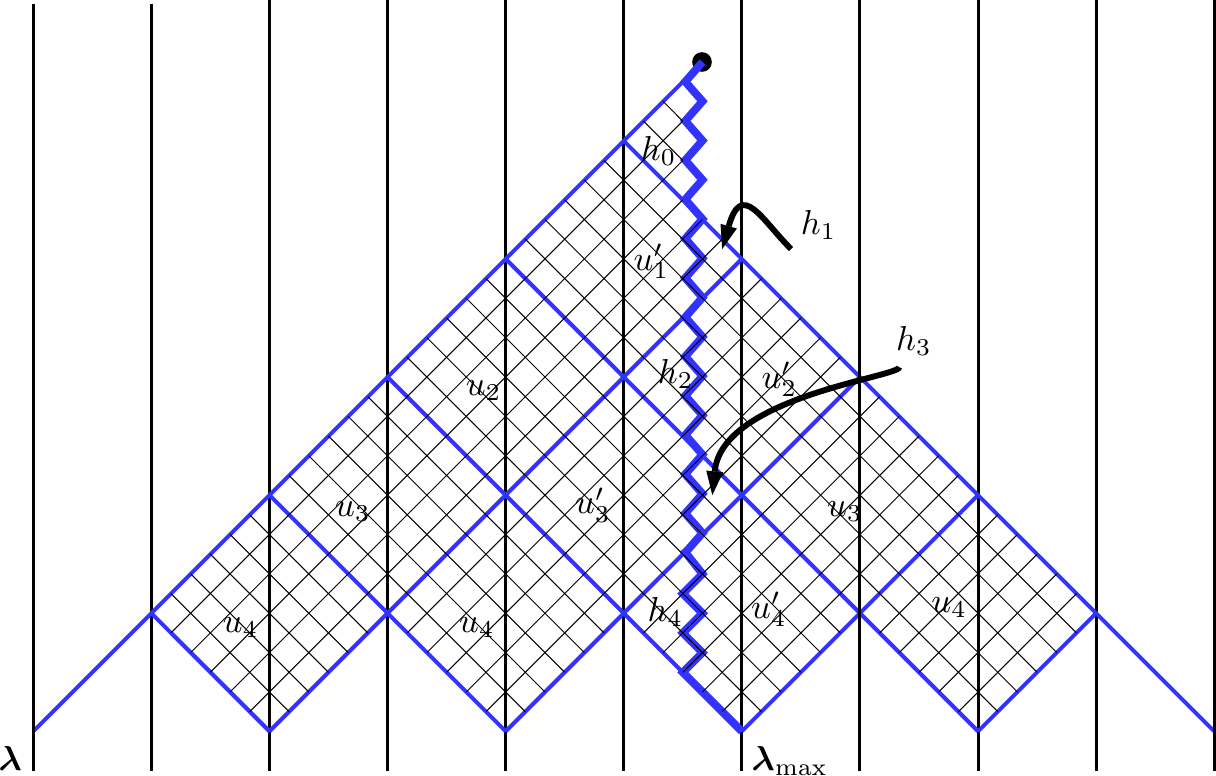} }
  \caption{The different regions $ h_i, u_i $ and $ u_i^{\prime}$.  }
\label{fig:cutdiagrams}
\end{figure}


For each of the three kinds of regions $  h_i, u_i, u_i^{\prime} $  we now introduce 
an element $ H_i, U_i ,  U_i^{\prime} \in  \Si_n$ in the following way. For $ R = h_i, u_i, u_i^{\prime} $  we
let $ \partial(R) $ be the boundary of $ R $ with respect to the usual metric topology.
Then for any $ R = h_i, u_i, u_i^{\prime} $ we have that $ \partial(R) $
is a union of line segments and we define the outer boundary, $ \partial_{out}(R) $,
as the union of the two line segments that are the furthest away from {\color{black}{$ P_{\blambda_{\max}}$.}}
Moreover we define
the \emph{inner boundary} as 
$ \partial_{in}(R) =  \overline{\partial(R) \setminus \partial_{out}(R) }$,
where the overline means closure 
{\color{black}{with respect to the metric topology.}}


Suppose now that $ R = h_i $ (resp. $ R= u_i $ and $ R= u_i^{\prime} $).
We then choose any tableau $ \bb \in \std(\OnePar) $ such that
$ \partial_{in}(R) \subseteq P_{\bb}$. Let $ P_{\bb}^{\prime}$ be the path obtained from 
$ P_{\bb} $ by replacing $  \partial_{in}(R) $ by $ \partial_{out}(R)$. Then we define $ H_i \in S_n $ (resp. 
$ U_i \in \Si_n $ or $ U_i^{\prime}\in \Si_n $) by the equation
\begin{equation}
  P^{\prime}_{\bb} = P_{\bb H_i} \, \, (\mbox{resp. } P^{\prime}_{\bb} = P_{\bb U_i}  \mbox{ and }
  P^{\prime}_{\bb} =
\color{black}{P_{\bb U_i^{\prime}}}).
\end{equation}
In other words, $ H_i $ (resp. $ U_i $ and $ U_i^{\prime} $) is simply the element of $ \Si_n$ that is used to fill
in the region $ h_i$ (resp. $ u_i $ and $ u_i^{\prime} $) in the sense of Algorithm \ref{algorithm paths},
where each $ s_i $ appearing in
$ H_i $ (resp. $ U_i $ and $ U_i^{\prime} $) corresponds to the filling in of one of
the little squares of $ h_i $ (resp. $ u_i $ and $ u_i^{\prime} $).
For example, in the situation of Figure \ref{fig:cutdiagrams} we have that
\begin{equation}
  H_0 = s_2 s_4 s_6 s_3 s_5 s_4, \,\, \, H_1 =s_9 s_{11} s_{10}, \,\, \, U_1^{\prime} =s_{[8,12]}  s_{[7,13]}  s_{[6,14]}
  s_{[5,15]}       s_{[6,14]} s_{[7,13]} s_{[8,12]} s_{[9,11]} s_{[10,10]}
\end{equation}  
where we used the notation from \ref{sameparity} for the formula for $ U_1^{\prime} $.
Note that the $ U_i $'s coincide
with the $ U_i$'s defined in \ref{Si}. It is also possible to give formulas 
for the $ H_i$'s and the $ U_i^{\prime}$'s, in the spirit of \eqref{Si}, but we do not need them.




\medskip
For any $ \T \in \std_{\blambda}(\bmu) $ we now introduce a reduced expression
for $ d(\T) $ by applying Algorithm \ref{algorithm paths} in a way compatible with the regions.
To be precise, starting with {\color{black}{$ P_{\blambda_{\max}}$}} we first choose those regions $ h_i $ that
give rise to a path closer to $ P_\T $ than {\color{black}{$ P_{\blambda_{\max}}$}}, by replacing the inner boundaries
with the outer boundaries. Having adjusted   
{\color{black}{$ P_{\blambda_{\max}}$}}
for those $ h_i$'s we next choose those regions $ u_i^{\prime} $ that the same way
give rise to a path even closer to $ P_\T $ and finally we repeat the process with the
regions $ u_i $. It may be necessary to repeat the last step more than once.
The product of the corresponding symmetric group elements is
now a reduced expression for $ d(\T) $: this is \emph{our favorite reduced expression} for
$ d(\T) $ that we shall henceforth use.

\medskip
In Figure \ref{fig:cutdiagramsA} we consider two examples with $ e=6 $ and $ m=2$.
{\color{black}{These examples shall be applied repeatedly throughout this section.}}

\begin{figure}[h]
\centering
$  P_\s= \, \, \, \, \, \, 
\raisebox{-.5\height}{\includegraphics[scale=0.5]{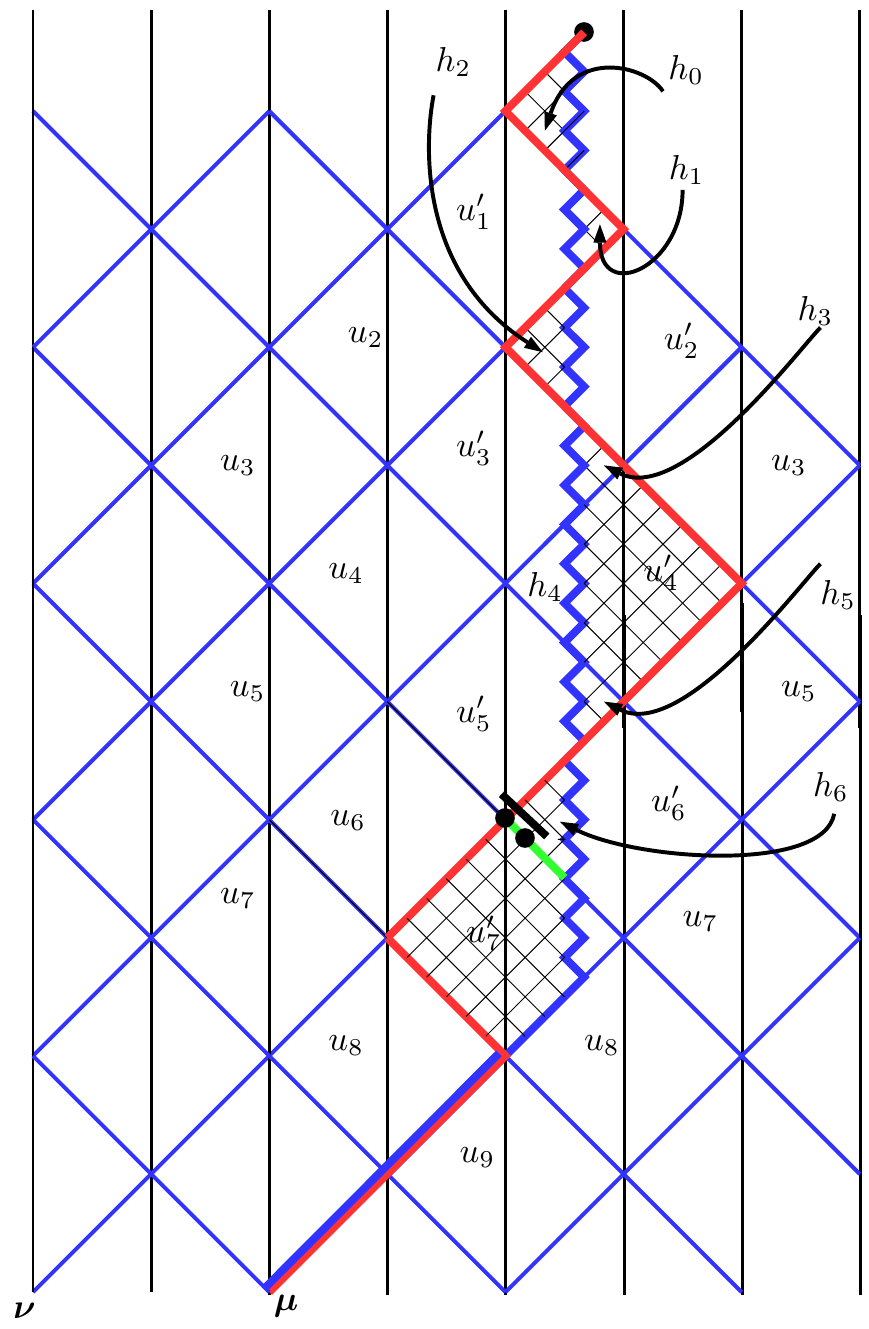}}
  \, \, \, \, \, \, \, \, \, \, \, \, \, \, \, \, \, \, \, \, \, \, \, \, \, \, \, 
  \, \, \, \, \, \, \, \, \,
  P_\T= \, \, \, \, \, \, \, \, \, 
\raisebox{-.5\height}{\includegraphics[scale=0.5]{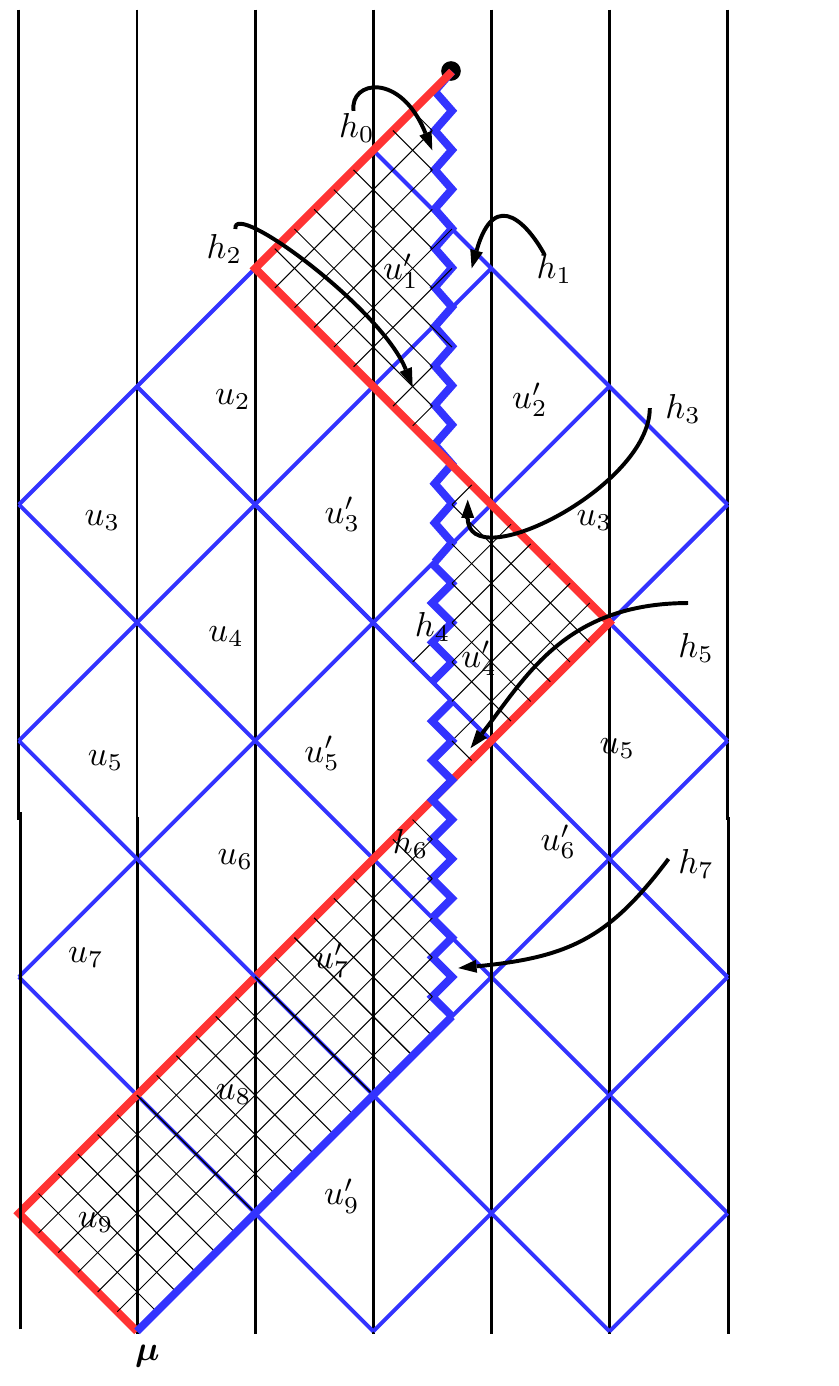}} $
\caption{Examples of the application of Algorithm \ref{algorithm paths} with $ e=6 $ and $m=2$.}
\label{fig:cutdiagramsA}
\end{figure}

\begin{figure}[h]
 \centering
${\color{black}{e(\bi^{\bmu})}}  \psi_{d(\s)} = $
 \raisebox{-.5\height}{\includegraphics[scale=0.7]{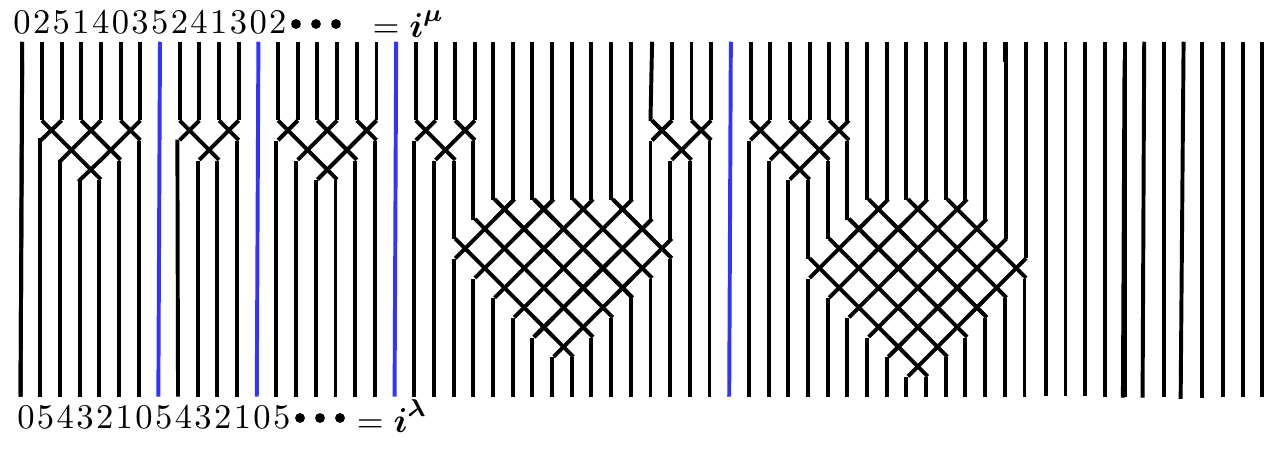}}
\caption{The KLR-diagram corresponding to the path $ P_\s $ of Figure \ref{fig:cutdiagramsA}.}
\label{fig:TKLRXXX}
\end{figure}


\begin{figure}[h!]
 \centering
${\color{black}{e(\bi^{\bmu})}} \psi_{d(\T)}=  \raisebox{-.5\height}{\includegraphics[scale=0.7]{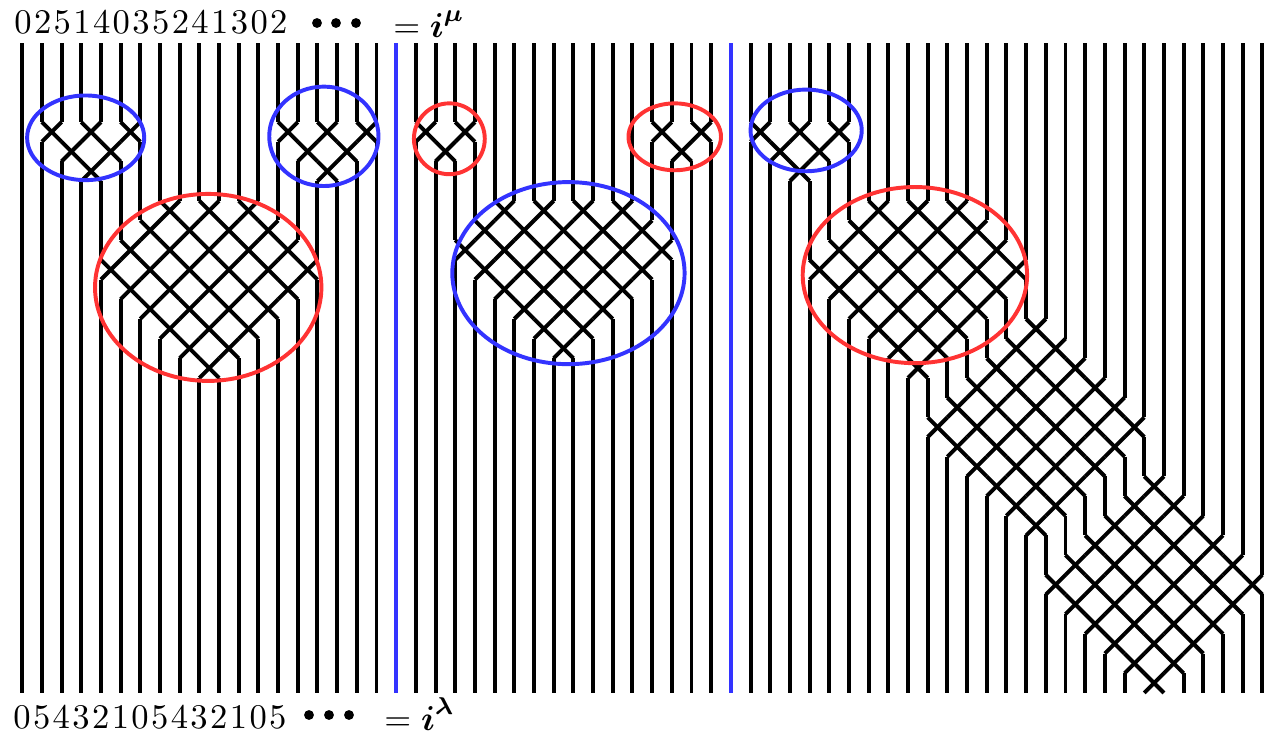}}$
\caption{The KLR-diagram corresponding to the path $ P_{\T} $ of Figure \ref{fig:cutdiagramsA}.}
\label{fig:KLRTXXX}
\end{figure}



\medskip
We let $ \psi_{H_i} $ (resp. $ \psi_{U_i} $ and $ \psi_{U_i^{\prime}} $) be the element of $\B $ obtained by
    replacing each $ s_i \in \Si_n$ in $ H_i $ (resp.  $ U_i $ and $ U_i^{\prime}$) with the corresponding $ \psi_i$.
We then get an expression for $ \psi_{ d(\T)} $ by replacing each occurring 
$ H_i $ (resp.  $ U_i $ and $ U_i^{\prime}$) in the above expansion for $ d(\T) $ by     
$ \psi_{H_i} $ (resp. $ \psi_{U_i} $ and $ \psi_{U_i^{\prime}} $).
Note that $ \psi_{U_i}e(\bi^\blambda) = U_{i}^{\blambda} \in G(\blambda) $ from \ref{genG}.
For example, in the cases given in Figure 
\ref{fig:cutdiagramsA} we have
\begin{equation}
\psi_{d(\s)}  = \psi_{H_0} \psi_{H_1} \psi_{H_2} \psi_{H_3} \psi_{H_5} \psi_{H_6}  \psi_{U_4^{\prime}} \psi_{U_7^{\prime}}
\qquad \mbox{ and } \qquad 
\psi_{d(\T)} = \psi_{H_0} \psi_{H_2} \psi_{H_3} \psi_{H_5} \psi_{H_6}  \psi_{U_1^{\prime}}
  \psi_{U_4^{\prime}} \psi_{U_7^{\prime}}  \psi_{ U_8} \psi_{U_{9}}.
\end{equation}

Let us give some comments related to the combinatorial
structure of the KLR-diagrams in Figure \ref{fig:TKLRXXX} and Figure \ref{fig:KLRTXXX}; these comments
hold in general. 
Note first that only the lower residue sequence of the KLR-diagrams in Figure \ref{fig:TKLRXXX}
and \ref{fig:KLRTXXX} is $ \bi^{\blambda} $ and
so $ {\color{black}{e(\bi^{\bmu})}} \psi_{d(\s)} $
and $e(\bi^{\bmu})\psi_{d(\T)} $ actually do not belong to $ \trunc$, only to $ \B$.
Secondly, note that the KLR-diagrams for the $\psi_{H_i}$'s are located in the `top lines' of
the KLR-diagrams in 
\ref{fig:TKLRXXX} and \ref{fig:KLRTXXX}, whereas the KLR-diagrams for the $\psi_{U_i^{\prime}}$'s and the
$\psi_{U_i}$'s are situated in `the middle and the bottom lines' of the diagrams in \ref{fig:TKLRXXX}
and \ref{fig:KLRTXXX}, respectively.
For each $ i $ only one of the diagrams $\psi_{H_i}$ or $\psi_{U_i^{\prime}}$ appears. 
The appearing $\psi_{H_i}$'s and $\psi_{U_i^{\prime}}$'s are ordered from the left to the right, with
$\psi_{H_0}$, that always appears, to the extreme left and so on. On the other hand, in  general
the $\psi_{U_i}$'s do not appear ordered.

Next, we observe that the shapes of $\psi_{H_i}$'s and the $\psi_{U_i^{\prime}}$'s depend on their parity. In other words, if $i$ and $j$ have the same parity then  $\psi_{H_i}$ and   $\psi_{H_j}$ (resp. $\psi_{U_i^\prime}$ and $\psi_{U_j^\prime}$) have the same shape.  In Figure \ref{fig:KLRTXXX} we have encircled with blue the \emph{even} diagrams $\psi_{H_i}$ and
$\psi_{U_i^{\prime} }$ and with red the \emph{odd} diagrams $\psi_{H_i}$ and
$\psi_{U_i^{\prime} }$. 

\medskip
Our next observation is that the diagrams $\psi_{U_i^{\prime}}$ always lie between two
diagrams $\psi_{H_{i-1}}$ and $\psi_{H_{i+1}}$, except possibly for the rightmost $\psi_{U_i^{\prime}}$. 
The rightmost $\psi_{U_i^{\prime}}$ is always preceded by $\psi_{H_{i-1}}$ but it may be followed by
$\psi_{U_{i+1}}$, as in Figure \ref{fig:KLRTXXX}, or by a number of
through lines, {\color{black}{as in Figure \ref{fig:TKLRXXX}}}.

\medskip
In general, we have that the $ \psi_{ H_i}$'s are `distant' apart and so pairwise commuting. This is not the case for the $\psi_{U_i^{\prime}}$'s. However, we still have that $\psi_{U_i^{\prime}}\psi_{U_j^{\prime}}= \psi_{U_j^{\prime}}\psi_{U_i^{\prime}}$ if $|i-j|>1$. By the previous paragraph we know that each occurrence of $\psi_{U_i^{\prime}}$ is surrounded by $\psi_{H_{i-1}}$ and $\psi_{H_{i+1}}$. We conclude that if $\psi_{U_i^{\prime}}$ and $\psi_{U_j^{\prime}}$ occur in the diagram of some $\psi_{d(\T)}$ then $|i-j|>1$, and therefore, they do commute.  The relations between the $\psi_{U_i}$'s are known from \cite{LiPl}, we shall return to them shortly. Between the different groups there is no commutativity in general, that
is $\psi_{U_i^{\prime}}$ does not commute with $\psi_{H_{i-1}}$ and $\psi_{H_{i+1}}$ and so on.


\medskip
Finally, we observe that the all of the diagrams $\psi_{H_i}, \psi_{U_i^{\prime}}$ and
$  \psi_{U_i}$ are organized tightly. There are for example only two
through lines in Figure \ref{fig:KLRTXXX}. 
In both Figure \ref{fig:TKLRXXX} and Figure \ref{fig:KLRTXXX} we have colored blue the through lines
that correspond to the places where $ P_\s $ and $ P_\T $ change from the left to right half of the
Pascal triangle, or reversely. In general these lines lie between two $ \psi_{H_i} $'s.
Thus the contours' of the diagrams \ref{fig:TKLRXXX} and \ref{fig:KLRTXXX} are a mirror of the shapes of the
paths of Figure \ref{fig:cutdiagramsA}, with the modification that the through blue lines indicate a
change from left to right of reversely.

\medskip

For $\T \in \std (\bi^{\blambda})$ we define $\theta(\T) $ as the element of $\Si_n$
obtained from the favorite reduced expression for $d(\T)$ by erasing all the $ U_i$-factors 
and similarly we define $u(\T) \in \Si_n$
by erasing both the $H_i$ and the $ U^{\prime}_i$-factors.
Then clearly 
\begin{equation}
d(\T) =   \theta(\T) u(\T).
\end{equation}
We now have the following Lemma.
\begin{lemma} \label{lemma reduction to central}
Suppose that $\s,\T \in \std_\blambda (\bmu)$ and let
$ P_{\s_1} $ and $ P_{\T_1}$ be the paths obtained from  $P_\s$ and $ P_\T$ by replacing outer boundary with
inner boundary for all the $\color{black}{ u_i}$-regions. Then we have that 
$\theta(\s)= d(\s_1) $ and $ \theta(\T) = d(\T_1) $. Moreover 
\begin{equation}\label{reducion}
m_{\s \T}^{\bmu} = \psi_{u(\s)}^{\ast} m_{\s_1 \T_1}^{\bmu} \psi_{u(\T)}.
\end{equation}  
\end{lemma}
\begin{proof}
The result is a direct consequence of the definitions. 
\end{proof}

Our goal is to prove that $ m_{\s \T}^{\bmu}$ belongs to $ \truncPrime $. On the other hand, 
$  \psi_{u(\s)} $ and $ \psi_{u(\T)} $ in \ref{reducion} are products of $ U_i^{\blambda}$'s 
and so it follows from Lemma \ref{lemma reduction to central} that to achieve this goal it is enough to consider the case where
$ \s= \s_1 $ and $ \T = \T_1$. Let us give the corresponding formal definition.

\begin{definition}\rm
  Let $\T \in \std (\bi^{\blambda})$. We say that $\T$ is \emph{central} if $u(\T)$ is the empty word.
  Equivalently, $\T$ is \emph{central} if  ${d(\T)}= \theta(\T)$. 
\end{definition}
Geometrically, $ \T $ is central if the path $ P_\T$ stays close to the central vertical axis of the Pascal triangle. 
{\color{black}{In other words, $P_\T $ does not cross the walls $ M_{-1} $ and $ M_{2} $, except possible once
in the final stage.}}
For example, in Figure \ref{fig:cutdiagramsA} we have that $ \s$ is central but $ \T $ is not.
In view of  Lemma \ref{lemma reduction to central} we will from now on only consider central tableaux.




\medskip

Suppose therefore that $\T \in \std_\blambda(\bmu_k)$ is central 
where $ \bmu_k$ is as described in Figure \ref{fig:pascalAA}. 
Then one checks that the total number of $\psi_{H_i} $'s and $ \psi_{U_i^{\prime}}$'s
appearing in $ \psi_{ d(\T)} $ is $k$.
We now define a 
$(2\times k)$-matrix $c(\T)=(c_{ij})$ of symbols that completely 
determines $\psi_{d(\T)}$. It is given by the following rules.
 \begin{enumerate}
 \item  If $H_i$ appears in appears in $d(\T)$ then $c_{1,i+1}:={H_i}$
   and $c_{2,i+1} := \emptyset$.
 \item If $U_{i}^{\prime}$ appears in $d(\T)$ then $c_{2,i+1}={U_{i}}'$ and $c_{1,i+1}:= \emptyset $.
 \end{enumerate}

We view the matrix $ c(\T) $ as a codification for $ \psi_{d(\T )}$,
where the first row of $ c(\T) $ corresponds to the top line of $ \psi_{d(\T )}$ and
the second row of $ c(\T) $ to the second line of $ \psi_{d(\T )}$. The comments that were made on the
structure of the digrams in Figure \ref{fig:TKLRXXX} and Figure \ref{fig:KLRTXXX}
carry over to the matrices $c(\T)$. 
In particular, 
exactly one of $ H_i $ or $ U_i^{\prime} $ appears in $c(\T)$ for each $ i $.
Moreover, $ H_0 $ always appears and each $ U_i^{\prime} $, except possibly $  U_{k-1}^{\prime} $, 
is surrounded by $ H_{i-1}  $ and $ H_{i+1} $.

For example if $ \psi_{d(\s)} $ is as in Figure \ref{fig:TKLRXXX}, then   
%
\begin{equation}\label{codeA}
c(\s) = \raisebox{-.5\height}{\includegraphics[scale=1]{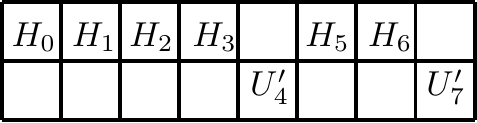}}  
\end{equation}
Note that we leave the entries containing $ \emptyset $ empty. 
Similarly, let $ \T $ be as in Figure \ref{fig:cutdiagramsA} but with 
the regions $ U_8 $ and $ U_9 $ eliminated. Then $ \T $ is central and
$ \psi_{d(\T)}$ 
is obtained by 
deleting $ \psi_{U_8}$ and $ \psi_{U_9}$ from the diagram in Figure \ref{fig:KLRTXXX} and we have

\begin{equation}\label{TKLRtheta}
  \! \!\! \! \! \! \! \!   \! \!   \! \!   \! \!   \! \!    \! \!   \! \!   \! \!    \! \!   \! \!   \! \!
  \! \!   \! \!   \! \!   \! \!   \! \!   \! \!   \! \!   \! \!    \! \!   \! \!   \! \!    \! \!   \! \!   \! \!
  \! \!   \! \!   \! \!   \! \!   \! \!   \! \!   \! \!   \! \!  
  \psi_{d(\T)} =
  \raisebox{-.5\height}{\includegraphics[scale=0.7]{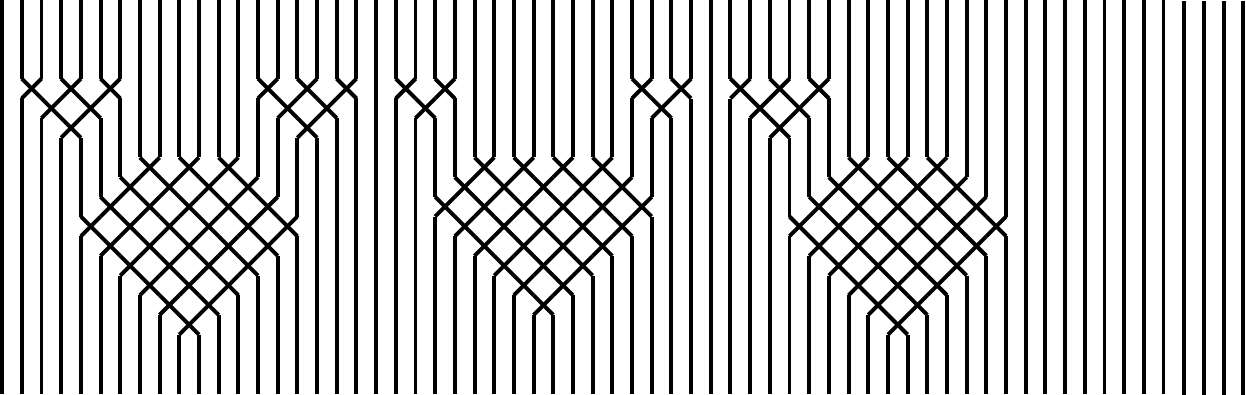}}
\end{equation}
with corresponding matrix
\begin{equation}\label{codeB}
  c(\T) =  \raisebox{-.5\height}{\includegraphics[scale=1]{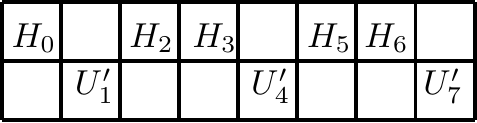}}  
\end{equation}
We are interested in the elements $ m_{\s \T}^{\bmu}$. In the cases of $ \s $ and $ \T$ given in Figure
 \ref{fig:TKLRXXX} and in \ref{TKLRtheta} it is given in Figure \ref{fig:KLRmst}.

\begin{figure}[h!]
 \centering
$  m_{\s \T}^{\bmu}=
   \raisebox{-.5\height}{\includegraphics[scale=0.7]{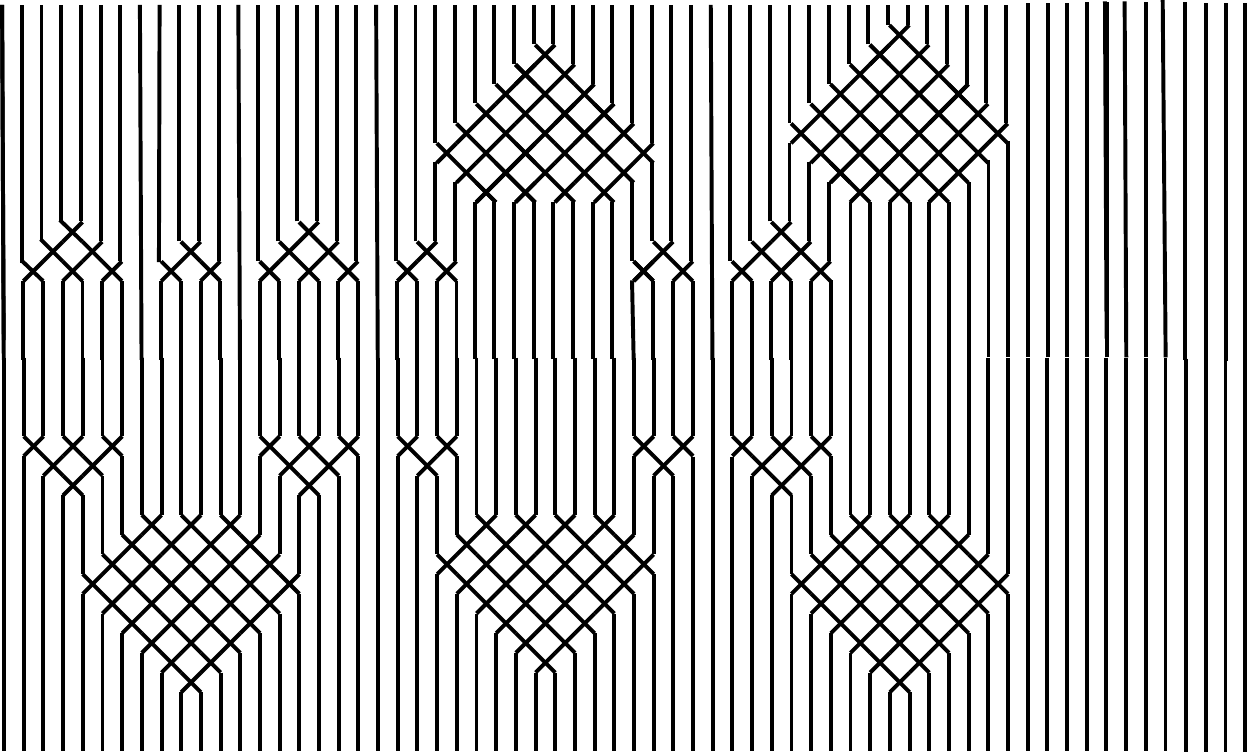}} \, \, \, 
   c(\s, \T) = 
      \raisebox{-.5\height}{\includegraphics[scale=0.7]{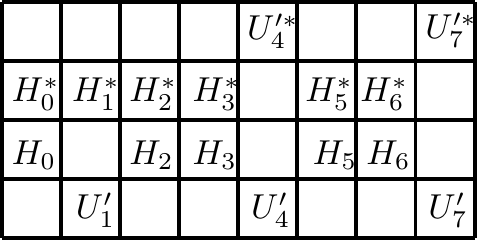}}$
\caption{The element $m_{\s \T} $ for $\s $ and $ \T$ given in Figure
 \ref{fig:TKLRXXX} and in \ref{TKLRtheta} and its codification.}
\label{fig:KLRmst}
\end{figure}


In general, for $\T \in \std_\blambda(\bmu_k)$ central we 
define $c^\ast(\T)$ as the $(2\times k)$-matrix $(d_{ij})$ where
$ d_{1j}= c_{2j}^\ast $ and $ d_{2j}= c_{1j}^\ast $. Here we set $ \emptyset^\ast := \emptyset$.
Moreover, for $\s, \T \in \std_\blambda(\bmu_k)$ both central we define
$ c(\s, \T)$ as the $(4\times k)$-matrix that has $ c^\ast(\s) $ on top of $ c(\T)$.
Then $c(\s, \T)$ is our codification of $ m_{\s \T}^{\bmu}$. 
In Figure \ref{fig:KLRmst} we have given $c(\s, \T)$ next to $ m_{\s \T}^{\bmu}$.

\medskip
Our task is now to show that any diagram as in Figure \ref{fig:KLRmst} can be written in 
terms of the elements from $ G(\blambda)$. This requires calculations using the defining relations  
for $ \B$. Let us first recall a couple of results from the literature.  



\begin{lemma}\label{previousLemma}
The idempotent $ e( \bi ) \in \B $ is nonzero only if $ \bi = \bi^{\T} $ for some $ \T \in \std(\OnePar)$.
\end{lemma}
\begin{proof}
  This follows from Lemma 4.1(c) of \cite{hu-mathas}, where it was proved for cyclotomic Hecke algebras in general,
  combined with the fact that $ \B$ is a graded quotient of the cyclotomic Hecke
  algebra of type $G(2,1,n)$, see \cite{PlazaRyom}. 
\end{proof}

\begin{lemma}\label{blockcancel}
  Let $ B_i $ be a full path interval for $ \blambda $ as introduced in \ref{DEFfullBlock} and suppose that
  $ k, l \in B_i $. Then we have that
  \begin{equation}
y_k e(\bi^\blambda) = y_l  e(\bi^\blambda).
   \end{equation}   
\end{lemma}
\begin{proof}
This follows from relation \ref{KLRquadratic} and Lemma \ref{previousLemma}.
\end{proof}

    \begin{lemma} \label{lemma muere y a la izquerda}
{\color{black}{Suppose that $ \bnu \in  \OnePar$ and that}}
  $\T \in \std(\OnePar ) $. Suppose moreover that
  $ P_\T \! \mid_{[0,k]}  = P_{\bnu} \! \! \mid_{[0,k]} $ for 
some integer $ k \ge 0$ and that $ P_\T([0,k-1]) \subseteq  {\mathcal A }^0 \setminus (M^0 \cup M^1) $.
Then for all $1\leq  r \leq  k$ we have in $\B $ that 
    \begin{equation}
	y_re(\bi^{\T}) = 0.
\end{equation}
\end{lemma}
    \begin{proof}
      {\color{black}{Recall that $ P_{\bnu} $
          zigzags along the vertical central axis of the Pascal triangle
      and finally goes linearly off to
    $ \bnu$. If $ r $ belongs to the zigzag part of $ P_{\bnu} $}}, the result 
follows from the Lemmas 14 and 15 of \cite{Lobos-Ryom-Hansen}, see also Theorem 6.4 of \cite{Esp-Pl}.
{\color{black}{Otherwise, if $ r $ belongs to the linear part of $ P_{\bnu} $, we argue as in
    the previous Lemma and get that $ y_re(\bi^{\T}) = y_{r-1}e(\bi^{\T})$. Continuing like this, we finally end up in
the zigzag part of $ P_{\bnu} $.}}
\end{proof}

Henceforth, we color the intersections of our KLR-diagrams 
according to the difference of the relevant residues. More precisely, we shall use
the following color scheme

\begin{equation}
  \raisebox{-.5\height}{\includegraphics[scale=1]{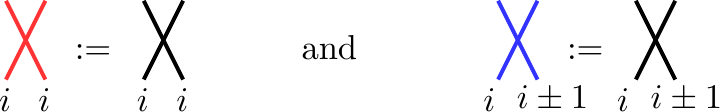}} 
\end{equation}

\noindent
whereas for all other crossing we keep the usual black color.
In this notation we now have the following Lemma which is a direct consequence of the relations
\eqref{punto-arriba} and \eqref{KLRquadratic}.

\begin{lemma} \label{doublecross}
  We have the following relations in $ \B$
  \begin{equation}
 \raisebox{-.5\height}{\includegraphics[scale=1]{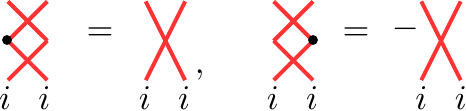}} 
  \end{equation}
  \end{lemma}

We can now finally prove the Theorem that was announced in the beginning of this section.

\begin{theorem}\label{bigtheorem}
  The set $ G(\blambda)$ introduced in \ref{genG} generates $ \trunc$.
\end{theorem}

\begin{proof}
  Using the coloring scheme introduced above, the diagram Figure \ref{fig:KLRmst}
  looks as follows

\begin{equation}\label{KLRmstXA}
\raisebox{-.5\height}{\includegraphics[scale=0.7]{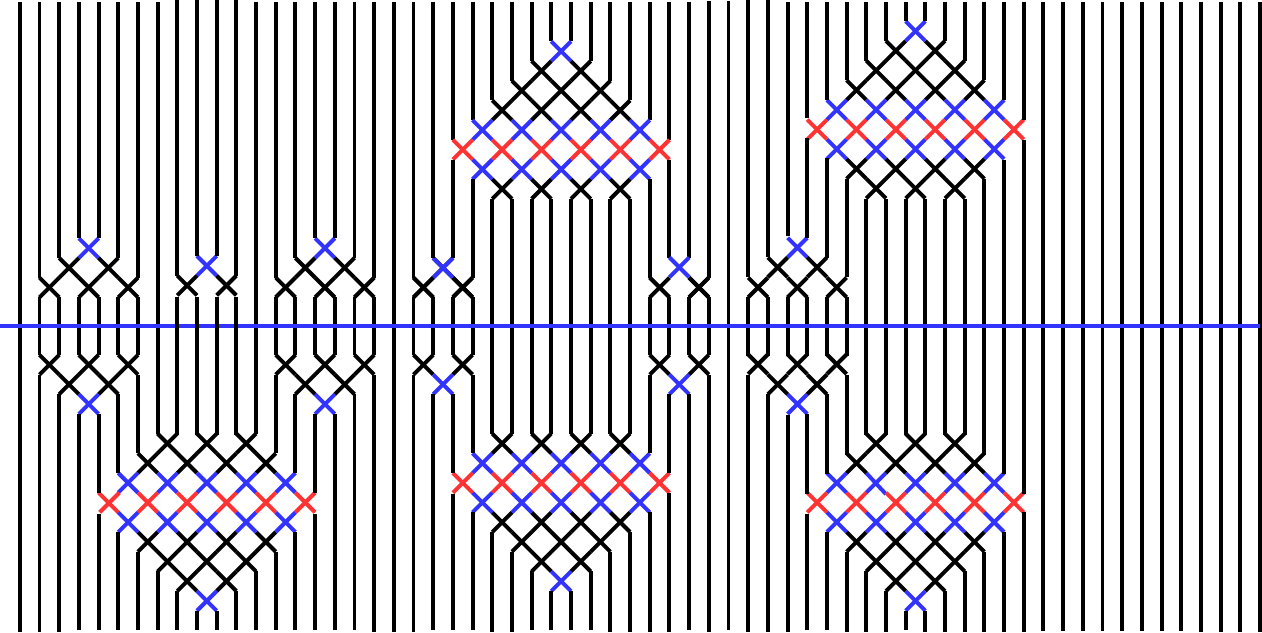}}
   \, \, \, \, \, \, \, \, \, \, \, \, \, \, \, \, \, 
   \raisebox{-.5\height}{\includegraphics[scale=0.7]{codeST.pdf}}
\end{equation}

\medskip

We must show that the elements  $ m_{\s \T}^{\bmu}$  can be written in terms of the elements of $ G(\blambda)$.
We will do so by pairing the elements of the columns of the corresponding $ c(\s, \T)$. 

\medskip
Note that the residue sequence for the middle blue horizontal of \ref{KLRmstXA} 
is
{\color{black}{$ \bi^{{\bmu}}$}}.
The idea is to apply Lemma \ref{lemma muere y a la izquerda} 
and therefore it is of importance to resolve the 
columns from the right to the left.

\medskip

Let us first consider columns containing pairs $ \{ H_{i}^{\ast}, H_{i} \} $,
starting with the rightmost of these columns. Thus in the above case we consider first 
$ \{ H_{6}^{\ast}, H_{6} \} $. We now use relation \eqref{KLRquadratic} to undo all the
crossings in $  H_{i}^{\ast} $ and $ H_{i}  $, arriving at a 
diagram like \ref{KLRmstX}. Here we use an overline on the two dots to denote 
that the result is a difference of two equal diagrams
but each with \emph{one} dot in the indicated place.
Note that the residue sequence for the middle line has now changed, and correspondingly we have 
changed the color from blue to red
{\color{black}{and green}}
around the two dots. In the above case, the new middle residue sequence is
$ \bi^{\T_1} $ where $ \T_1 = \T^{\bmu} H_6$, {\color{black}{that is $ \T_1 $ is obtained from
    $ \T^{\bmu} $ by replacing $ \partial_{in}(h_6) $ with $ \partial_{out}(h_6) $.
    In the leftmost diagram of Figure \ref{fig:cutdiagramsA}, we have indicated $ P_{\T_1} $, using the same colors
    red and green. On the leftmost dot, given by $ y_{40} $ in the above example, 
    we can now apply Lemma \ref{lemma muere y a la izquerda},
    with $ \T = \T_1 $ and $  \bnu $ as indicated in Figure \ref{fig:cutdiagramsA}. 
    We conclude from the Lemma that the corresponding diagram is zero. 
}}


\medskip
Thus in the above case \ref{KLRmstX} only the second term dot with $ y_{41} $
stays. We now repeat this process for all the other pairs of the form $ \{ H_{i}^{\ast}, H_{i} \} $,
from the right to the left. For example in the case \ref{KLRmstX} we arrive 
at the diagram \ref{KLRmstXBlockA}.
We have indicated the path intervals for $ \blambda$ on the top of
the diagrams \ref{KLRmstX} and \ref{KLRmstXBlockA}.
Note that each $  H_{i} $ (resp. $H_{i}^{\ast},  U_{i}^{\prime} $ and $U_{i}^{\prime  \ast}$) `intersects' both of
the path intervals $ B_{i} $ and $ B_{i+1} $ and that the dots of \ref{KLRmstXBlockA} are all situated at the beginning of a path interval.

\begin{equation}\label{KLRmstX}
m_{\s \T}^{\bmu}= \, \, \, \, \, \,
\raisebox{-.5\height}{\includegraphics[scale=0.7]{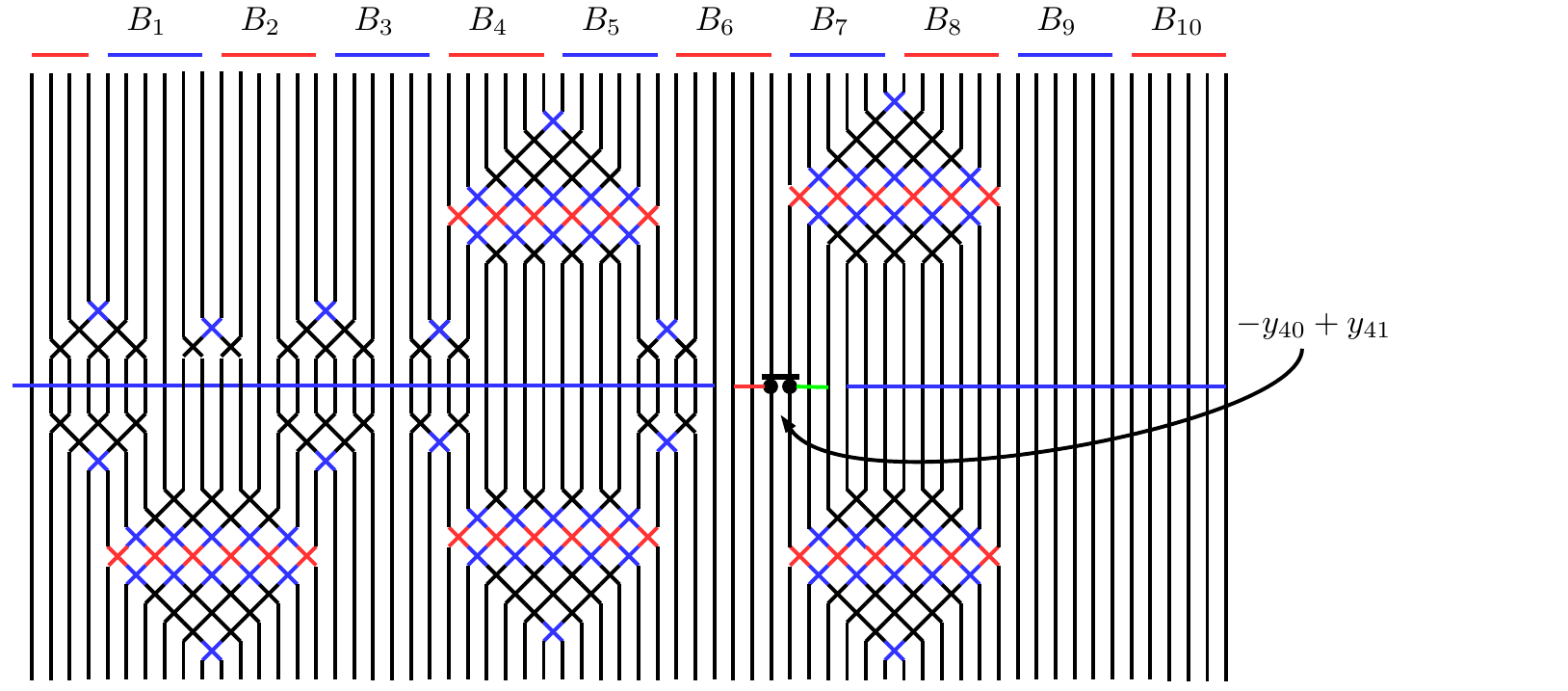}}
\end{equation}

\begin{equation}\label{KLRmstXBlockA}
\! \! \!\! \! \!  \! \!\! \! \!
m_{\s \T}^{\bmu}=   \, \, \, \, \, \,
\raisebox{-.5\height}{\includegraphics[scale=0.7]{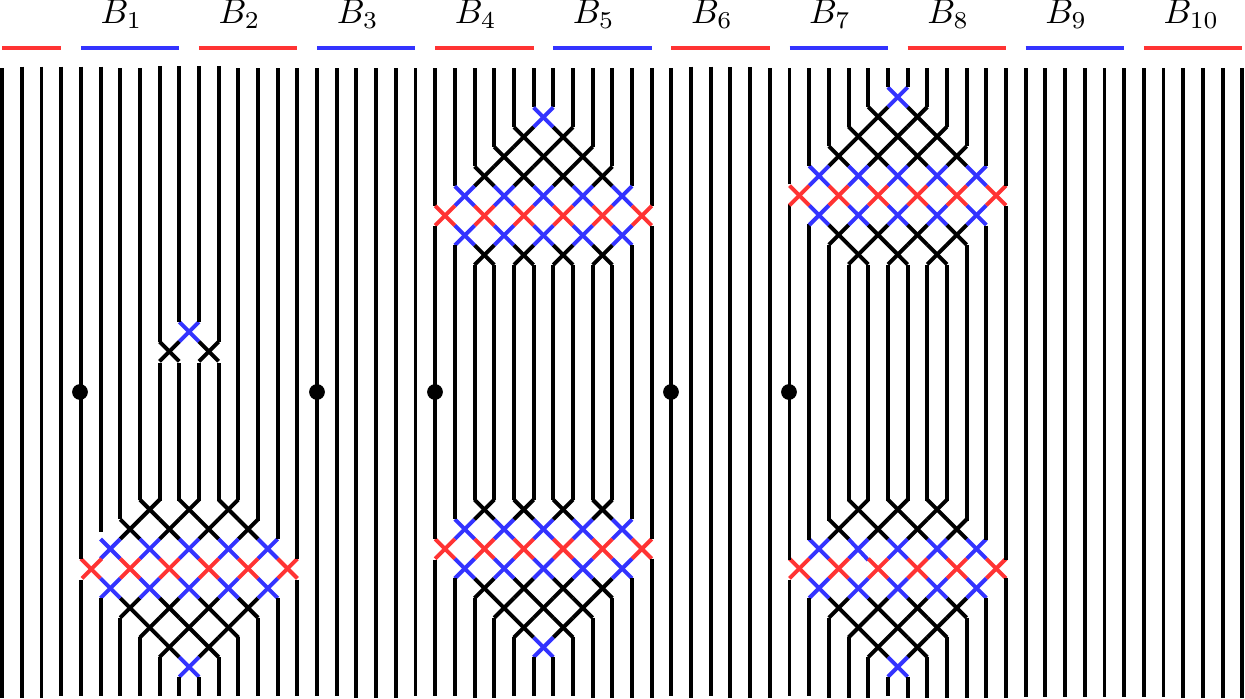}}
    \, \, \,      \, \, \,      \, \, \,      \, \, \,  
\end{equation}

\medskip
Next we treat the pairs of the form $ \{ U_{i}^{\prime \ast}, H_{i} \} $ or
$ \{  H_{i}^{\ast}, U_{i}^{\prime  }\} $. By the combinatorial remarks made earlier,
each appearing $ H_i $-term (resp. $H_{i}^{\ast}$-term) fits perfectly with
the corresponding $ U_{i}^{\prime \ast } $-term (resp. $U_{i}^{\prime}$-term) to form a diamond. 
We then move the $ H_i $-term up (resp. the $H_{i}^{\ast}$-term down) to form
this diamond. Note that this process does not involve any other terms since 
the $ H_i $-terms (resp. the $H_{i}^{\ast}$-terms) are distant from the surrounding
dots. In the above case \ref{KLRmstXBlockA} we get the following diagram.

\medskip 
\begin{equation}\label{KLRmstXBlock}
m_{\s \T}^{\bmu}=   \, \, \, \,
\raisebox{-.5\height}{\includegraphics[scale=0.7]{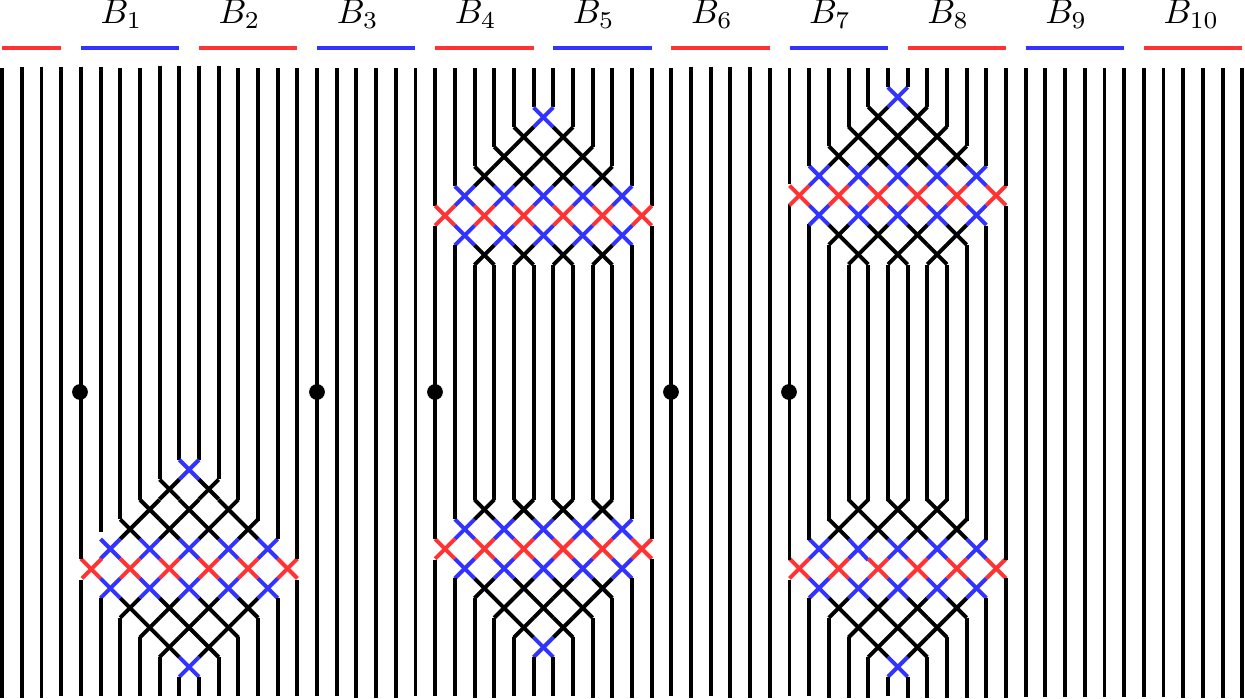}}
    \, \, \,      \, \, \,      \, \, \,      \, \, \,  
\end{equation}

\medskip
We are only left with columns containing pairs of the form $ \{ U_{i}^{\prime \ast}, U_{i}^{\prime} \} $.
By the previous step there is now a dot between the top $  U_{i}^{\prime \ast}$ and the bottom $ U_{i}^{\prime}  $,
at the left end of the `line segment' between them, see \ref{KLRmstXBlock}. We show that this
kind of configuration $ C_i $ is equal to diamond
$ \psi_{U_i}  $. In fact, the arguments we employ for this
have already appeared in the literature, see for example \cite{LiPl}.
Let us give the details corresponding to $ i= 7$ in \ref{KLRmstXBlock}; the general case is done the same way.
Using relation \ref{KLRquadraticA} to undo the black double crosses, next relation \ref{KLRquadratic}
to undo the last blue cross and finally \ref{KLRquadraticA} on the red double cross, we have the 
following series of identities. 
\begin{equation}\label{lastresolve}
C_7    \, \,
\raisebox{-.45\height}{\includegraphics[scale=0.6]{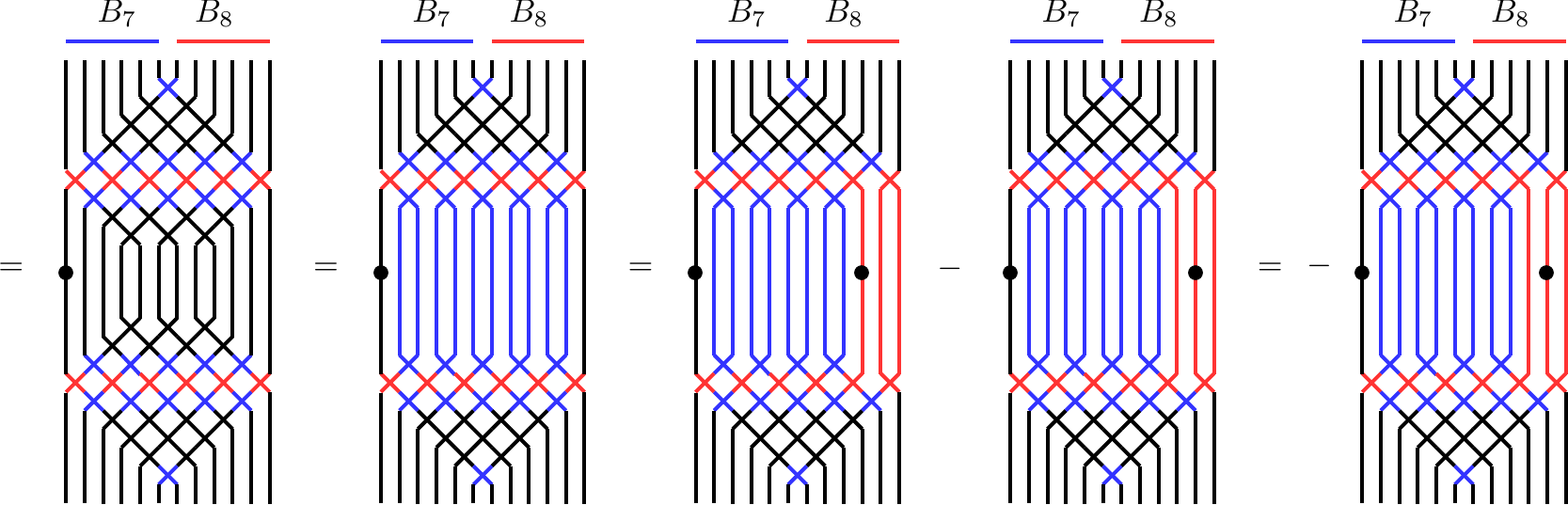}}
    \, \, \,      \, \, \,      \, \, \,      \, \, \,  
\end{equation}
But this process can be repeated on all the blue double crosses and so we have via Lemma 
\ref{doublecross} that 
\begin{equation}\label{lastresolveA}
C_7 \,
\raisebox{-.42\height}{\includegraphics[scale=0.7]{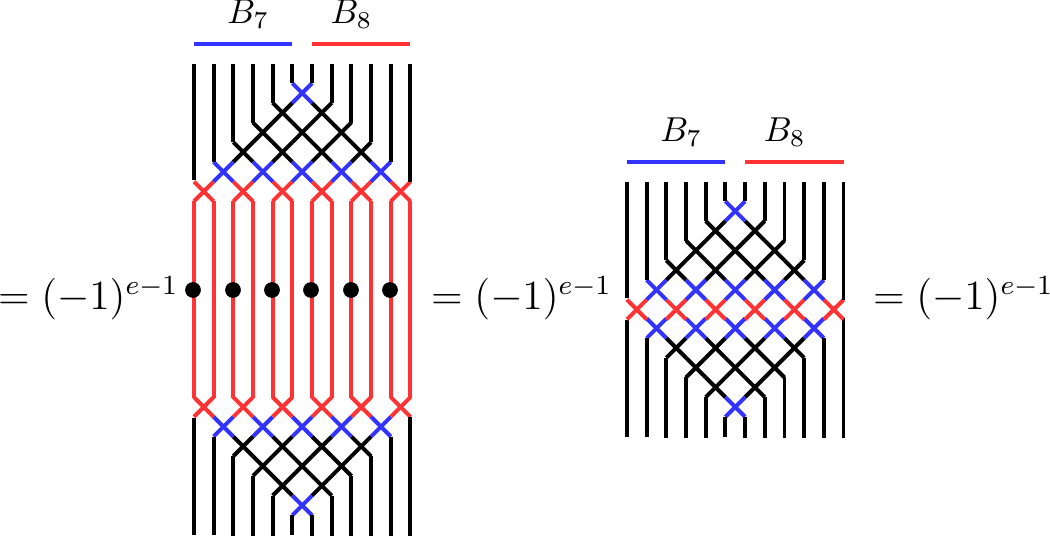}} \,
 \psi_{U_7}.
\end{equation}

The same procedure can be carried out for the other columns of the form $ \{ U_{i}^{\prime \ast}, U_{i}^{\prime} \} $.
In the above case there is only one such column, corresponding to $ i =4 $ and so get finally that


\begin{equation}\label{KLRlast}
m_{\s \T}^{\bmu}= {\color{black}{\pm}}   \, \,
\raisebox{-.45\height}{\includegraphics[scale=0.7]{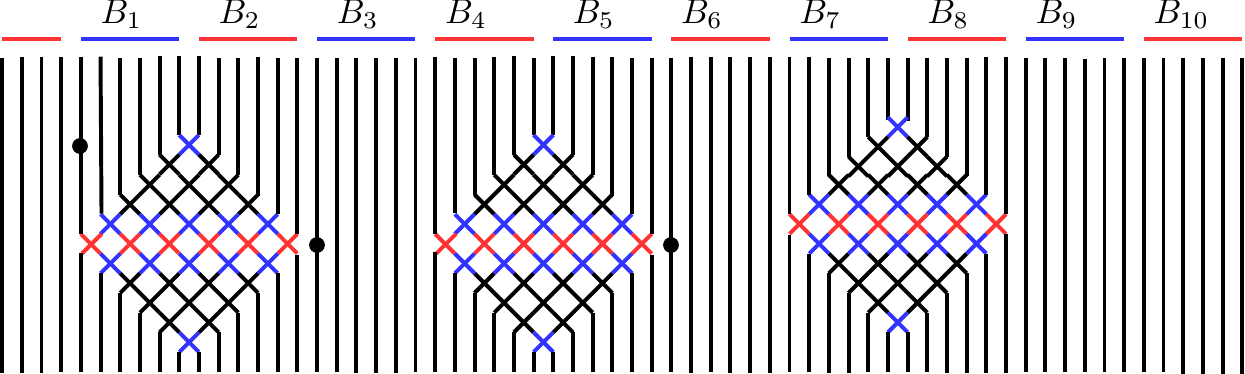}}
    \, \, \,      \, \, \,      \, \, \,      \, \, \,  
\end{equation}
In other words, since multiplication in $ \B$ is from top to bottom, we have that 

\begin{equation}
 m_{\s \T}^{\bmu} ={\color{black}{\pm}} \, y_5 U_1^{\blambda} y_{17} U_4^{\blambda} y_{35} U_7^{\blambda}.
\end{equation}
All appearing factors of $  m_{\s \T}^{\bmu} $ belong to $ G(\blambda) $ and so we have proved the Theorem.
\end{proof}

Let us point out some remarks concerning  Theorem \ref{bigtheorem} and its proof. First of all, we already saw that only a few of the $ y_i$'s are needed to generate $ \trunc$.
Let us make this more precise. Choose any $ k $ in the $i$'th path interval $ B_i$. Then we define 
\begin{equation}\label{blocki}
\YL_i:= y_k e(\bi^\blambda) \in \trunc .
\end{equation}
Note that by Lemma \ref{blockcancel}, we have that $ \YL_i $ is independent of the choice
of $ k$. Moreover, it follows immediately from Theorem \ref{bigtheorem} that $\trunc $ is generated by the set

\begin{equation}\label{gene}
\{ U_j^\blambda  \, |\, 1\leq j < K  \}	 \cup \{ \YL_i \, |\, 1\leq i \leq K \}.
\end{equation}

Secondly we remark that the proof of Theorem \ref{bigtheorem} gives rise to an algorithm for writing the above
$ m_{\s \T}^\bmu $ in terms of the generators in \ref{gene}. Although the algorithm itself is not necessary for what follows, for  the sake of completeness we prefer to establish it formally.

\begin{algorithm}\rm  \label{algorithm matrix}
  Let $\bmu \in \OnePar $ and let $\s,  \T \in \std_{\blambda} (\bmu )$ be central tableaux.
  Let $ c(\s, \T) $ be the matrix associated with $m_{\s\T}^\bmu$.

	\begin{description}
	    \item[Step 0.] Add an empty column to the right of $c(\s, \T)$.  
		\item[Step 1.] For each column in $ c(\s, \T) $ containing $ \{ U_{i}^{\prime \ast}, H_{i} \} $ (resp.
                  $ \{  H_{i}^{\ast}, U_{i}^{\prime  }\}) $ we remove $ H_{i} $ (resp. $   H_{i}^{\ast}$) from $ c(\s, \T) $ and
                  replace $ U_{i}^{\prime \ast}$ (resp. $ U_{i}^{\prime  }) $ in $ c(\s, \T) $ by $ U_i$.
                \item[Step 2.] Working from the right two the left, for each column
                  in $ c(\s, \T) $ containing $ \{ H_{i}^{ \ast}, H_{i} \} $
                  we remove $  H_{i}^{ \ast} $ and $ H_{i}  $ from $ c(\s, \T) $ and write $ Y_{i+1 } $ in one of the two
                  middle boxes of the following column, one to the right. 
                \item[Step 3.] Each column in $ c(\s, \T) $ containing $ \{ U_{i}^{\prime \ast}, U_{i}^{\prime } \} $
                  will now also contain $ Y_{i } $. We replace these three ingredients of that column
                  by one $ U_i$ which is placed in one
                  of the two middle boxes of the column.
                \item[Step 4.] Replacing each $ U_i $ by $ U_i^\blambda $ and each $ Y_i $ by $ \YL_i$ 
                  we form the product of all appearing elements of $ c(\s, \T) $, starting with the top
                  line, then the two middle lines and finally the bottom line. This product is
                  ${\color{black}{\pm}}  m_{\s \T}^\bmu $.
        \end{description}
\end{algorithm}
Let us give an example to illustrate how the algorithm works. Suppose that $ \s $ and $ \T$ are
central tableaux and that $ c({\s ,\T})$ is as follows.

\begin{equation}\label{codeALG}
   \, \, \, \, \, \, \,
c({\s ,\T}) = \raisebox{-.5\height}{\includegraphics[scale=1]{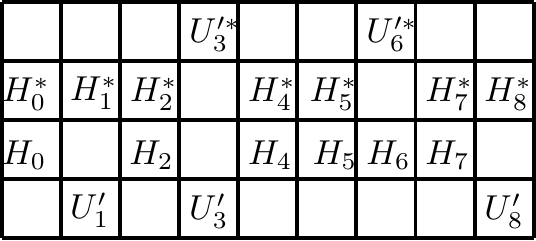}}
\end{equation}

Then going through the algorithm we get

\begin{equation}\label{codeALGA}
   c({\s ,\T}) \rightarrow \raisebox{-.5\height}{\includegraphics[scale=1]{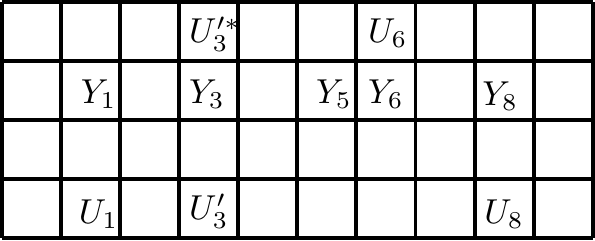}}
\rightarrow \raisebox{-.5\height}{\includegraphics[scale=1]{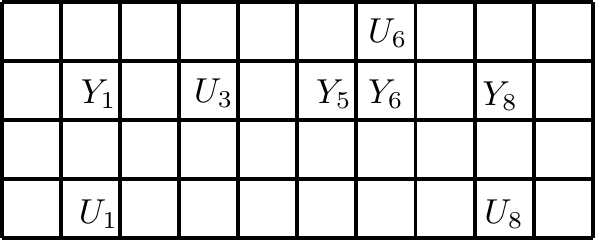}}
\end{equation}
and so we conclude that 

\begin{equation}
 m_{\s \T}^\bmu ={\color{black}{\pm}} U_6^\blambda \YL_1 U_3^\blambda \YL_5 \YL_6 \YL_8 U_1^\blambda U_8^\blambda .
\end{equation}  

\medskip 

Our next step is to show that actually only $ \YL_1 $ is needed in order to generate $ \trunc$.
Let us first prove the following result. 
\begin{lemma} \label{lemma second reduction of ys}
For all $1\leq i <K$ we have 
\begin{equation}\label{jugada-clave}
  \mathcal{Y}_{i+1}^\blambda U_i^\blambda =U_i^\blambda \mathcal{Y}_i^\blambda+
  (-1)^{e}(\mathcal{Y}_i^\blambda-\mathcal{Y}_{i+1}^\blambda).
    \end{equation}
\end{lemma}
\begin{proof}

Let us first recall the following relations valid in $ \B$, see Lemma 5.16 of \cite{LiPl}.
\begin{equation}\label{trenzadesataA}
 \raisebox{-.5\height}{\includegraphics[scale=1]{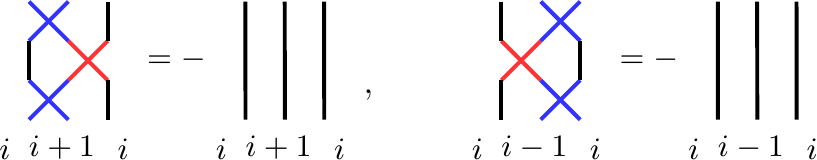}}
\end{equation}
They are a consequence of the braid relation \ref{KLRBraid} together with Lemma
\ref{previousLemma}.

\medskip
Let us now show the Lemma for $ i = 1$, since the general case is treated the same way.
We take $ e= 6 $. Then we have that via repeated applications of relation \ref{punto-arriba})
that

\begin{equation}\label{trenzadesataB}
  \mathcal{Y}_{2}^\blambda U_1^\blambda = \raisebox{-.5\height}{\includegraphics[scale=0.8]{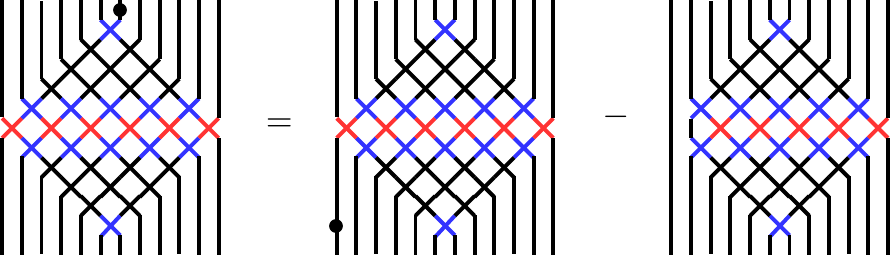}}
\end{equation}
The first diagram is here $ U_1^\blambda \mathcal{Y}_{1}^\blambda  $ so let us focus on the second diagram.
Using the first relation in \ref{trenzadesataA} repeatedly we get that it is equal to

\begin{equation}\label{trenzadesataC}
  \raisebox{-.5\height}{\includegraphics[scale=0.8]{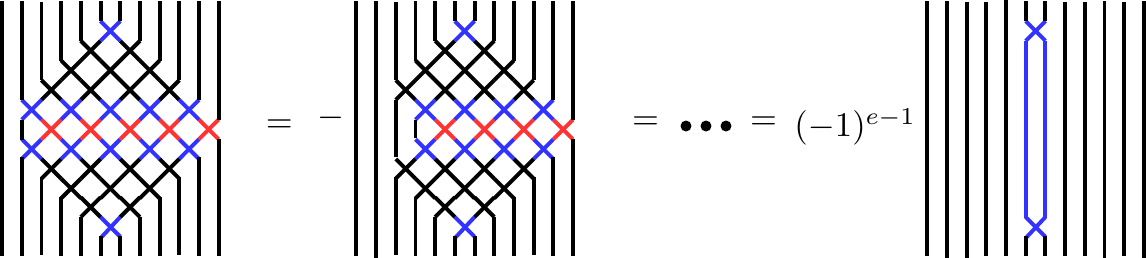}} =
(-1)^{e-1}  (\mathcal{Y}_{1}^\blambda -  \mathcal{Y}_{2}^\blambda) 
\end{equation}
where we used the quadratic relation \ref{KLRquadraticA} for the last step.
Combining \ref{trenzadesataB} and \ref{trenzadesataC}, we then get \eqref{jugada-clave}.
\end{proof}

Let us recall the commutation relations between the $ U_i^\blambda$'s, see Proposition 5.18 of \cite{LiPl}.
\begin{theorem}\label{TemperleyLieb}
  The subset $ \{ U_i^\blambda \mid i = 1, \ldots K-1 \} $ of $ \trunc $ verifies the
  Temperley-Lieb relations, or to be more precise
	\begin{align}
\label{eqUno}	&   (U_i^\blambda)^2   	 =  (-1)^{e-1} 2 U_i^\blambda , &   &  \mbox{if }  1\leq i < K; \\
\label{eqDos}	& U_i^\blambda  U_j^\blambda U_i^\blambda   =  U_i^\blambda , &  & \mbox{if } |i-j|=1 ; \\
\label{eqTres} & U_i^\blambda  U_i^\blambda 	  =  U_j^\blambda  U_i^\blambda  ,  &  &   \mbox{if } |i-j|>1. 
	\end{align}
  \end{theorem}

With this at our disposal we can now prove, as promised, that $ \mathcal{Y}_1^{\blambda} $ is the only
$ \mathcal{Y}_i^{\blambda} $ which is needed in order to generate $ \trunc$. 

\begin{theorem}  \label{theorem generating set without superfluous generators singular }
Suppose that $\blambda$ is singular. Then, the set
	\begin{equation}\label{defineG1}
          G_1(\blambda) := \{ U_i^\blambda  \, |\, 1\leq i < K   \}	 \cup \{ \mathcal{Y}_1^{\blambda} \}
	\end{equation}
	generates $\trunc$ as an $\F$-algebra.
\end{theorem}
\begin{proof}
Recall that $ e(\bi^\blambda) $ is the identity element of $ \trunc$, for simplicity we denote it by $ 1$.
Let us define 
\begin{equation}
S_i^\blambda :=  U_i^\blambda + (-1)^{e}.
\end{equation}
Then from Theorem \ref{TemperleyLieb} we get that
\begin{equation}  \label{SiSquare}
	(S_i^\blambda)^2   	 =  1. 
\end{equation}
On the other hand, we notice that using the notation introduced above, the relation \eqref{jugada-clave} becomes
\begin{equation}\label{jugada-clave-versionS}
	\mathcal{Y}_{i+1}^\blambda S_i^\blambda  = S_i^\blambda \mathcal{Y}_{i}^\blambda. 
\end{equation}
Finally, by combining \eqref{SiSquare} and \eqref{jugada-clave-versionS} we obtain 
\begin{equation} \label{jugada-clave-JM}
	\mathcal{Y}_{i+1}^\blambda   = S_i^\blambda \mathcal{Y}_{i}^\blambda  S_i^\blambda 
\end{equation}
and the result follows.
\end{proof}

We are now in position to prove the main result of this section.

\begin{theorem}\label{theoremsing}
Suppose that $\blambda$ is singular. Then, there is an isomorphism $f: \mathbb{NB}_{K} \rightarrow \trunc $ given by 
\begin{equation}
 \UU_0 \mapsto \YL_1 \,\,\,   \mbox{ and } \,\,\,  \UU_i \mapsto (-1)^{e}U_i^\blambda  \mbox{   for }  1\leq i < K.
\end{equation}
\end{theorem}
\begin{proof}
  In view of Theorem \ref{blob diagram realization} and the Pascal triangle description of the
  cellular basis for $ \trunc $, the two algebras have the same dimension.
  Hence, we only have to show that $ f $ is well defined since, by  Theorem \ref{theorem generating set without superfluous generators singular }, it will
  automatically be surjective.

  \medskip
  Let us therefore check that $  f( \UU_0) $ and the $  f(\UU_i ) $'s verify the relations for $ \mathbb{NB}_{K}$.
  The Temperley-Lieb relations \ref{eq one}, \ref{eq two} and \ref{eq three} are clearly satisfied by
  Theorem \ref{TemperleyLieb} whereas the relation $ (\YL_1)^2 = 0 $ follows from
  relation \ref{dot-al-inicio} and \ref{KLRquadratic}. Hence we are only left with checking relation
  \eqref{eq four}. It corresponds to $U_1^\blambda \YL_1  U_1^\blambda  = 0 $ which via
  Lemma \ref{lemma second reduction of ys} and \ref{eqUno} is equivalent to the relation
  \begin{equation}\label{checkrel}
  (\YL_1  + \YL_2) U_1^\blambda  = 0.
  \end{equation}
  For this we first write $ {\color{black}{(-1)^{e-1}}}U_1^\blambda $ in the following form
\begin{equation}\label{checkquad}
{\color{black}{(-1)^{e-1}}}U_1^\blambda= \raisebox{-.43\height}{\includegraphics[scale=0.7]{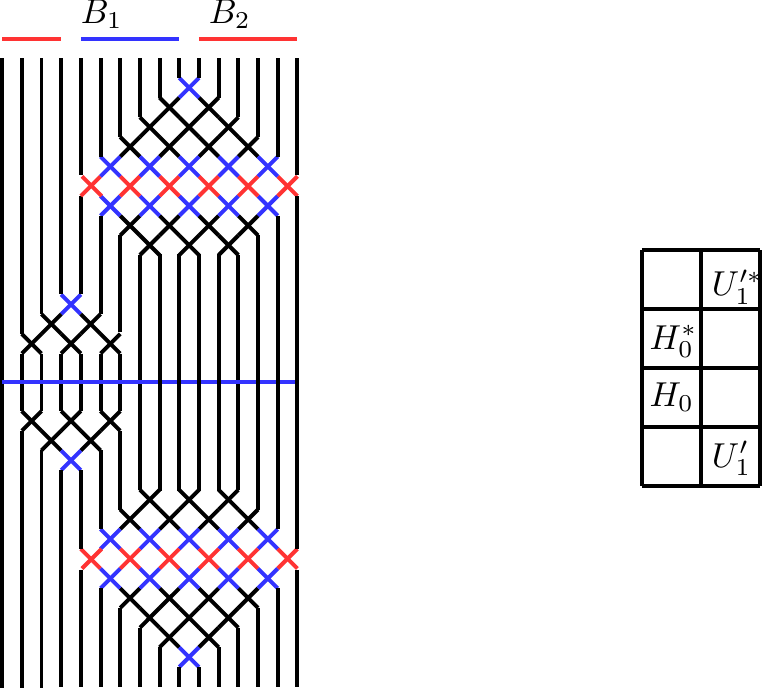}}
\end{equation}
We have here used $ e=6 $ as in the examples of the proof of Theorem \ref{bigtheorem}.
The middle blue horizontal line has the same meaning as in \ref{KLRmstX}; 
its residue sequence is $ \bi^{\bmu}$ {\color{black}{for the corresponding $ \bmu$}}.  Using this we get
\begin{equation}\label{checkquadX}
{\color{black}{(-1)^{e-1}}}\YL_1 U_1^\blambda \, \, \raisebox{-.43\height}{\includegraphics[scale=0.6]{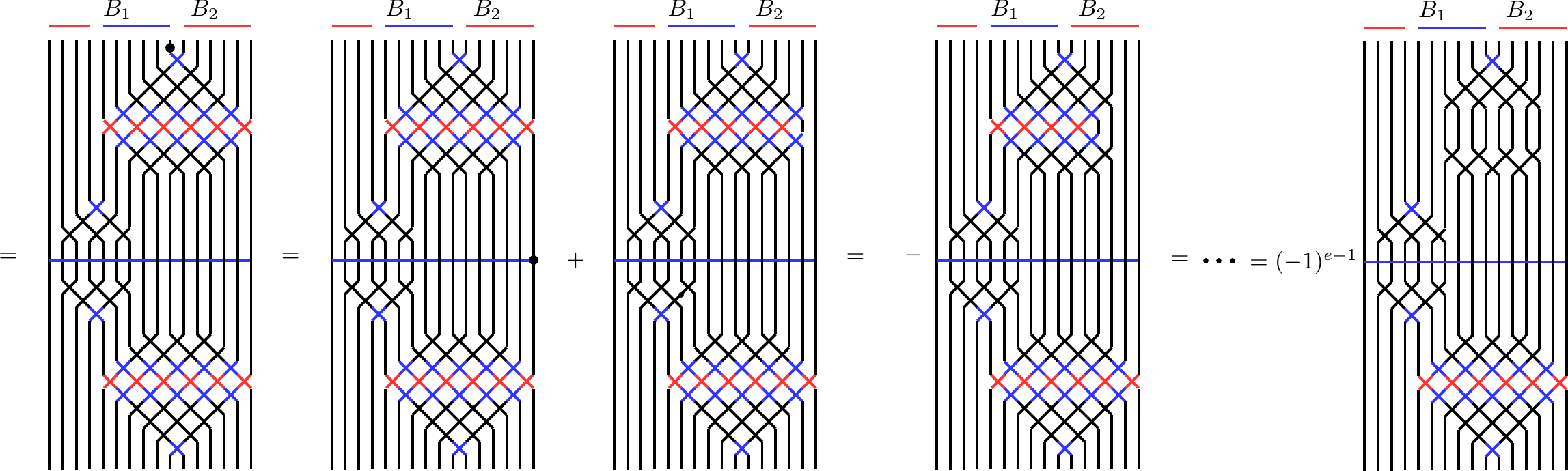}}
\end{equation}
where the first equality comes from relation \eqref{punto-arriba}, the second from
Lemma \ref{lemma muere y a la izquerda}
and the other equalities from \ref{trenzadesataA}. 
On the other hand, for ${\color{black}{(-1)^{e-1}}}\YL_2 U_1^\blambda $  we have almost the same expansion
with only a sign change coming from relation \ref{punto-arriba}: 
\begin{equation}\label{checkquadXX}
{\color{black}{(-1)^{e-1}}}\YL_2 U_1^\blambda   \, \,   \raisebox{-.43\height}{\includegraphics[scale=0.6]{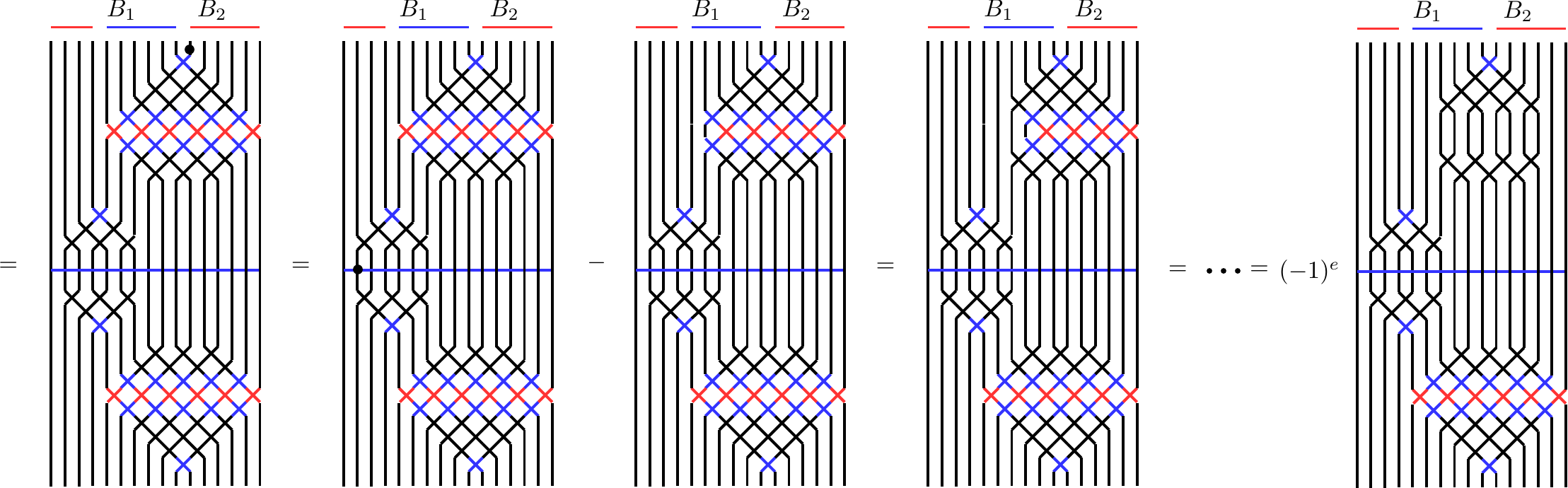}}
\end{equation}
Comparing \ref{checkquadX} and \ref{checkquadXX} we see that \eqref{checkrel} holds. The
Theorem is proved.
\end{proof}

\begin{remark}\label{remarkconcerningJM} \rm
  Using \eqref{jugada-clave-JM} and {\color{black}{\eqref{SiSquare}}} we extend $ (\YL_1)^2 = 0 $ 
  to 
  $ (\YL_i)^2 = 0 $ to all $ i$. Thus 
  the isomorphism $ \varphi:  \mathbb{NB}_n \cong A_w$ gives us a proof of
Lemma \ref{JMA}. The recursive formula for the $ \YY_i$'s is given by 
{\color{black}{\ref{jugada-clave-JM}}}.
  \end{remark} 


\section{A presentation for $\trunc$ for $\blambda$ regular}
In this section we consider the case where $\blambda $ is regular, in other words we assume that $R>0$,
see Definition \ref{definition epsilon}.  We define $ \trunc :=e( \bi^{\blambda } ) \B e( \bi^{\blambda}  ) $
just as in the singular case
but, as we shall see, the regular case is slightly more complicated than the  singular case since 
we need an extra generator. Recall first the function $ f=f_{n,m} $ from \ref{functionblocks}
which was used to define the full path interval in the singular case, see \ref{DEFfullBlock}.
Let $ K $ be as in Definition \ref{definition epsilon}.
Then in the regular case there is an extra \emph{non-full} path interval $ B_{last} $ defined as follows
\begin{equation}\label{lastBlock}
B_{last} := [f(K+1)+1, f(K+1)+2, \ldots, f(K+1)+R] =  [f(K+1)+1, f(K+1)+2, \ldots,n]. 
\end{equation}
For example in the situation described in Figure \ref{pascalAAA}, we have $ n=25, e=5, m=2 $
and so $ K=4, R =2$ and therefore
\begin{equation}\label{lastBlock}
B_1=[4,5,6,7,8], \, B_2=[9,10,11,12,13], \, B_3=[14,15,16,17,18], \, B_4=[19,20,21,22,23], B_{last} := [24,25].
\end{equation}

\begin{equation}\label{PASCALINOUT}
  \raisebox{-.43\height}{\includegraphics[scale=0.8]{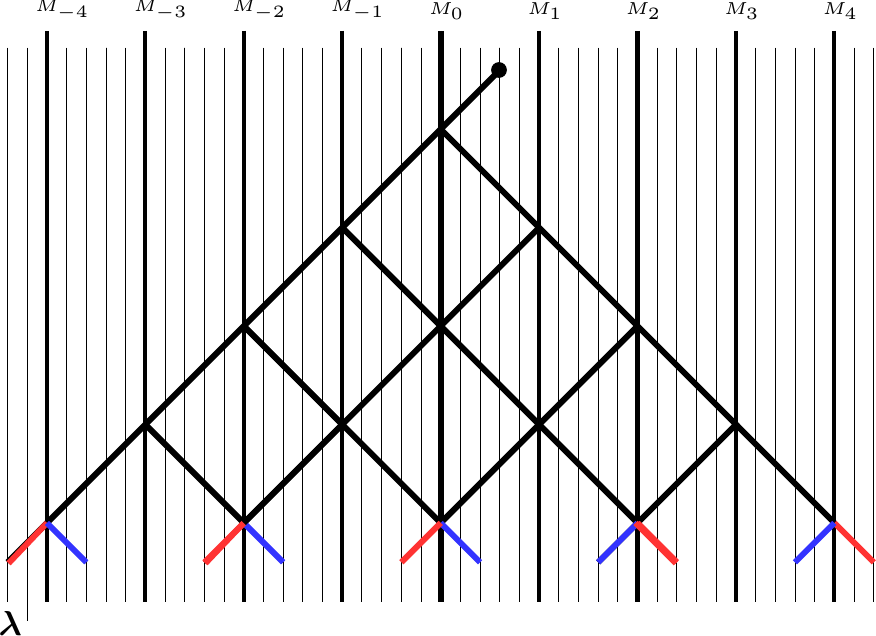}}
\end{equation}

Let $\overline{n}:= n - R $ and let $ \overline{\blambda} := ( 1^{ \overline{n}}, 1^0) \in \OneParSing$. We notice that 
\begin{equation} \label{nbar}
	\bar n = f(K+1). 
\end{equation}
It is clear from the definitions that  $ \overline{\blambda} $ is singular.
On the other hand, any $ \overline{\s} \in \std (\bi^{\bar \blambda }) $ gives rise to two tableaux
$ \overline{\s}(I) $ and $ \overline{\s}(O) $, in
$ \std (\bi^\blambda) $, as follows.
The tableau $ \overline{\s}(I) $ (resp. $ \overline{\s}(O) $) is defined as the unique
tableau $ \T \in \std (\bi^\blambda)$ whose path $ P_{\T } $ 
coincides with $  P_{ \overline{\s} } $ on the restriction to $ [1,2, \ldots, \overline{n}] $ and whose restriction to
$ B_{last} $ is a straight line that moves $ P_{\T } $ closer to (resp. further away from) the central vertical axis of the Pascal triangle.
We say that $ \T $ is an \emph{inner tableau} (resp. an \emph{outer tableau}) 
if it is of the form
$ \T =  \overline{ \s}(I)  $ (resp. $ \T =   \overline{\s}(O)$)  for some
$  \overline{\s} \in \std (\bi^{\bar \blambda}) $. 
It is easy to see that 
any tableau $ \T $ in $ \std (\bi^\blambda) $
is of the form $ \T =  \overline{\s}(I) $ or $ \T =  \overline{\s}(O) $ for a
unique $ \overline{ \s}\in \std ( \bi^{\bar \blambda} )$.

In \ref{PASCALINOUT} we have indicated with blue the restriction to
$ B_{last}$ of the paths corresponding to inner tableaux, 
and with red the restriction to $ B_{last} $ of the paths corresponding to outer tableaux.
Note that $ P_{\blambda} $ is always the path of an outer tableau.

\medskip

Let $ \bi^{last} \in \II^{ R } $ be the restriction to $ B_{last}$ of the residue sequence for $ \bi^{\blambda} $
and let $ e(\bi^{last}) $ be the corresponding idempotent diagram, consisting of $  R  $ vertical lines
with residue sequence $ \bi^{last}$.  
For $ x \in \mathbb{B}_{\bar n} $ we define the element $\iota (x):= x \wedge e( \bi^{last}) \in \B $ as the horizontal concatenation of $ x$ 
with $ e(\bi^{last}) $ on the right. We notice that   
\begin{equation} \label{iota property one}
	\iota (xy) = xy \wedge e( \bi^{last}) =( x \wedge e( \bi^{last}) )(y\wedge e( \bi^{last}) )= \iota (x) \iota (y), 
\end{equation}
for all $x,y\in   \mathbb{B}_{\bar n}$. Furthermore, 
\begin{equation} \label{iota property two}
	\iota (e(\bi^{\bar{\blambda}} ))= e(\bi^{\bar{\blambda}} ) \wedge e( \bi^{last})= e(\bi^\blambda). 
\end{equation}
    {\color{black}{We shall shortly prove that
$m_{ \s \T}^{\bmu} =  \iota( m_{ \overline{\s} \overline{\T}}^{\overline{\bmu}})$. Combining this with}}
    \eqref{iota property one} and \eqref{iota property two} we conclude that {\color{black}{there is an
    algebra inclusion}} 
\begin{equation}
	\iota ( \mathbb{B}_{\bar n} (\bar \blambda)) \subset  \trunc . 
\end{equation}
We define $ U_i^\blambda:=  \iota(U_i^{ \bar \blambda}) \in \trunc $ and $ \mathcal{Y}_j^{\blambda}:= \iota(\mathcal{Y}_j^{\bar \blambda}) \in \trunc$, for $1 \leq i <K$ and $1\leq j \leq K$.  


\medskip \noindent
It turns out that the outer tableaux are easier to handle than the inner tableaux.

\begin{lemma}\label{lemma outer elements}
  Let $\blambda $ be regular and suppose that $ \s =  \overline{\s}(O) $ and $ \T =  \overline{\T}(O) $
  are outer tableaux in $\std_{\blambda} (\bmu)$. Let $ \overline{\bmu}$ be the shape of $\overline{\s}$ and $\overline{\T }$.
  Then we have that 
\begin{equation}\label{thefirststatement}
m_{ \s \T}^{\bmu} =  \iota( m_{ \overline{\s} \overline{\T}}^{\overline{\bmu}}).
\end{equation}
Consequently, $ m_{ \s \T}^{\bmu} $ belongs to the subalgebra of $ \trunc $ generated by
$  \{ U_i^\blambda  \, |\, 1\leq i < K   \}$ and $  \mathcal{Y}_1^{\blambda} $.
\end{lemma}
\begin{proof}
  Using Theorem \ref{theorem generating set without superfluous generators singular } we see that the second statement
  follows from the first statement \ref{thefirststatement}. In order to prove the first statement
  we note that since $ \s $ and $ \T $ are outer tableaux we have that
\begin{equation}\label{outerfact}
d(\s) = d(\overline{\s}) \, \, \, \mbox{and} \, \, \, d(\T) = d(\overline{\T}).
\end{equation}
Here are examples illustrating \ref{outerfact}
\begin{equation}\label{cutdiagramsREG}
\raisebox{-.5\height}{\includegraphics[scale=0.5]{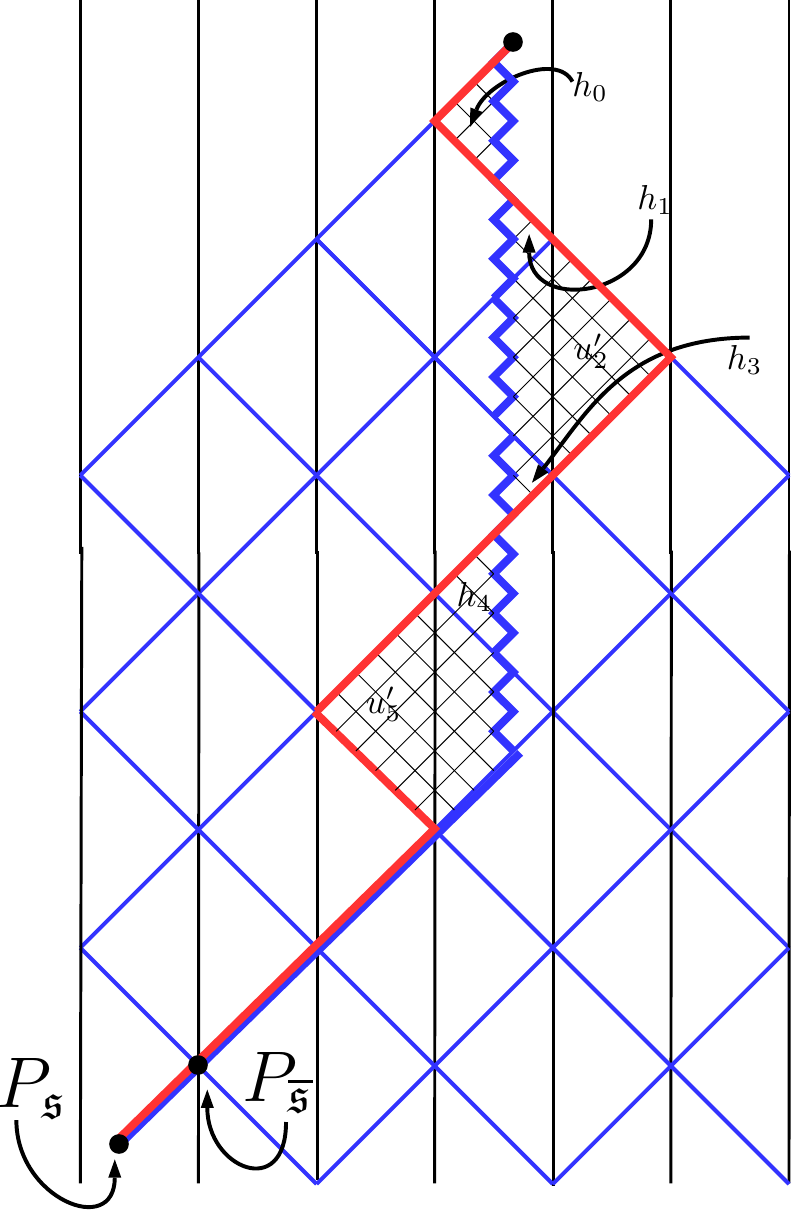}}
 \, \, \, \, \, \, \, \, \,  \, \, \, \, \, \, \, \, \,  \, \, \, \, \, \, \, \, \,  \, \, \, \, \, \, \, \, \, 
\raisebox{-.5\height}{\includegraphics[scale=0.5]{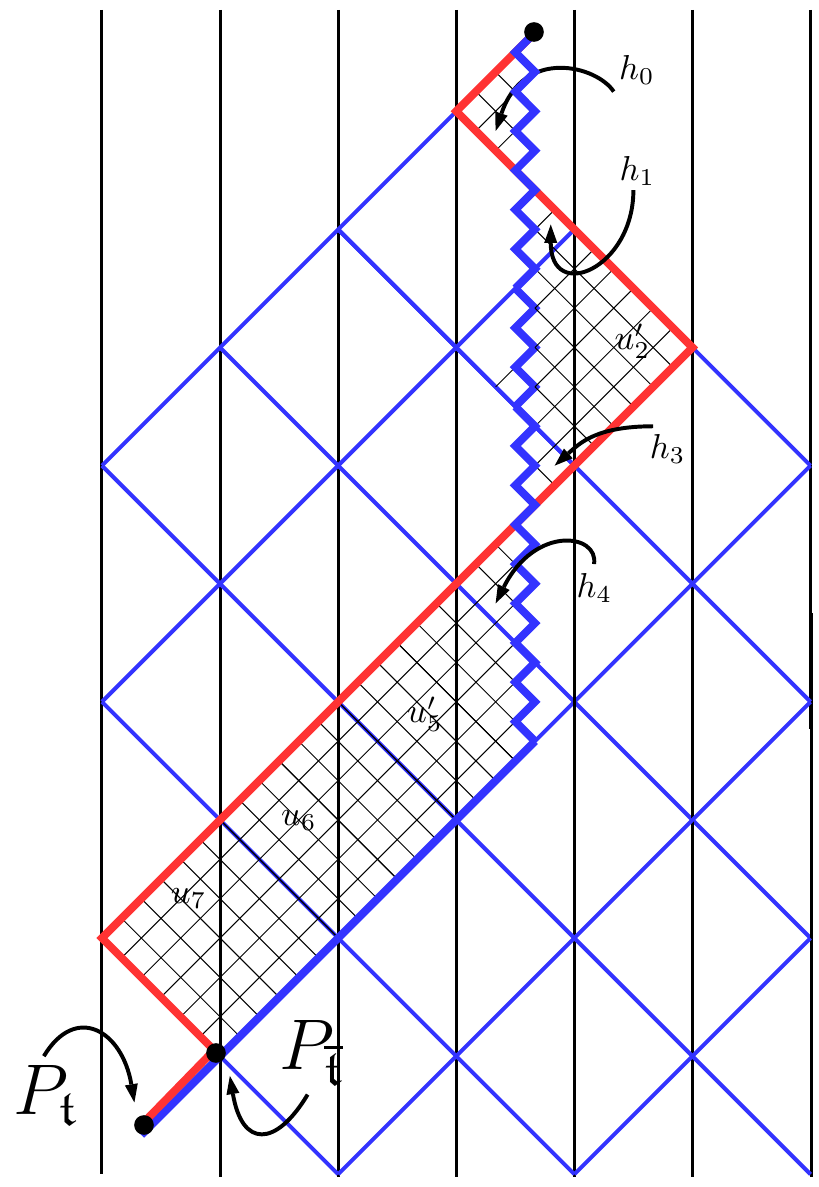}}
\end{equation}
On the other hand we have that $ e(\bi^\bmu) = \iota(e(\bi^{\overline{\bmu}}))$ and so
we obtain 
\begin{equation}
	\iota( m_{ \overline{\s} \overline{\T}}^{\overline{\bmu}}) = \iota ( \psi_{d(\bar \s )}^\ast e(\bi^{\bar \bmu } ) \psi_{d( \bar \T)} ) = \iota ( \psi_{d(\bar \s )}^\ast )  \iota (  e(\bi^{\bar \bmu } ) )  \iota ( \psi_{d( \bar \T)} )  = \psi_{d( \s )}^\ast e(\bi^{ \bmu }) \psi_{d(  \T)}  = m_{ \s \T}^{\bmu}.  
\end{equation}
\end{proof}

Suppose now that $ \s = \overline{\s}(I) \in \std_{\blambda} (\bmu) $ is an inner tableau.
Then $ d(\s ) $ and $ d(\overline{\s}) $ are different but still closely related.
Let $ a_{\s} $ be the region of the Pascal triangle bounded by $ P_{\s} $ and  $P_{\bmu}$ and
let  $ a_{\overline{\s}} $ be the region bounded by $ P_{\overline{\s}} $ and  $P_{\overline{\bmu}}$, where $\bar \bmu$ denotes the shape of $\bar \s$.
Then $ a_{\s} = a_{\overline{\s}} \, \cupdot \,s_{\bmu} $ where $ s_{\bmu} $ is the region bounded by $P_\bmu$ and $P_{ \T^{\bar \bmu  }(I)}$, see \ref{cutdiagramsREGSOUtER} for two examples in which we have
indicated $ s_{\bmu} $ with the color red.
Note that $ s_{\bmu} $ only depends on $ \bmu $ and not on $\s$, which is the reason for our notation.
When applying Algorithm \ref{algorithm paths}
there is an independence 
between the regions $ a_{\overline{\s}} $ and $ s_{\bmu}$. Indeed, 
let 
$  A_{\overline{\s}} \in \Si_n $ be the element obtained by filling in $  a_{\overline{\s}} $ as in
the algorithm, and let similarly
$  S_{\bmu} \in \Si_n $ be the element obtained by filling in $  s_{\bmu} $.
Then we have that
\begin{equation}\label{thenwehavethat}
  d(\s) = S_{\bmu}  A_{\overline{\s}}  .
\end{equation}

\begin{equation}\label{cutdiagramsREGSOUtER}
\raisebox{-.5\height}{\includegraphics[scale=0.5]{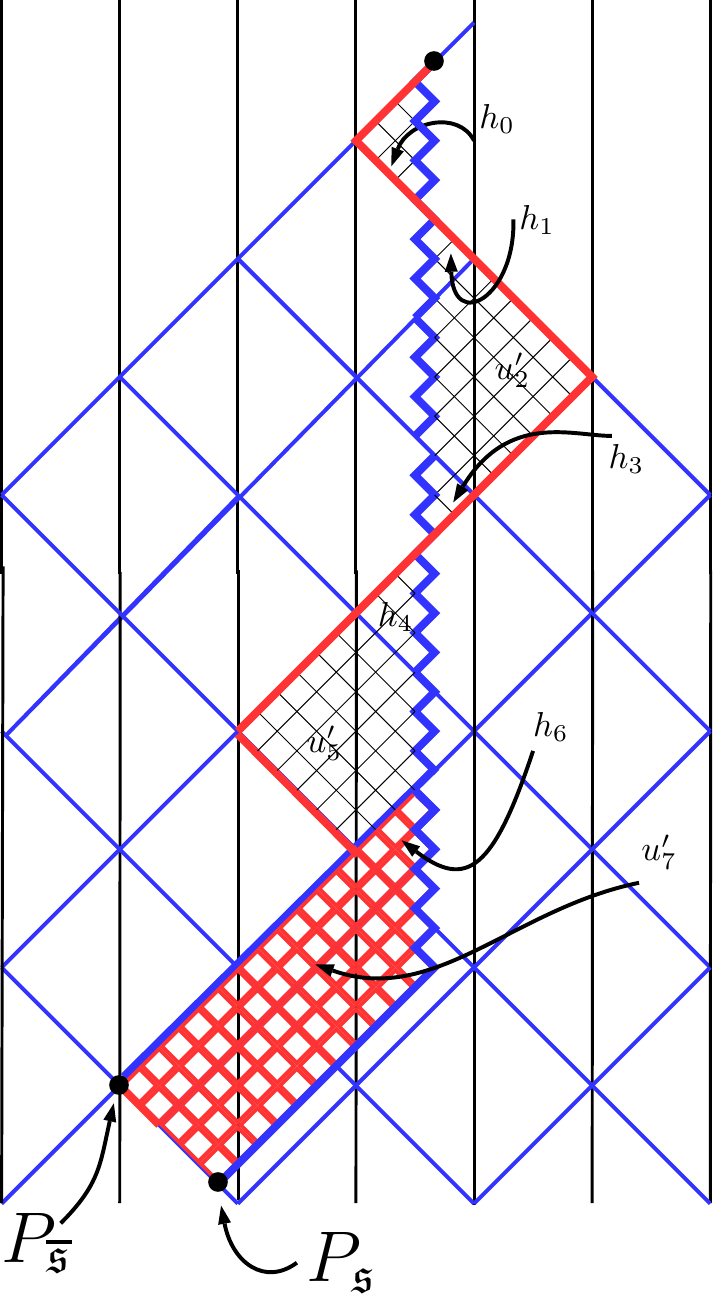}}
 \, \, \, \, \, \, \, \, \,  \, \, \, \, \, \, \, \, \,  \, \, \, \, \, \, \, \, \,  \, \, \, \, \, \, \, \, \, 
\raisebox{-.5\height}{\includegraphics[scale=0.5]{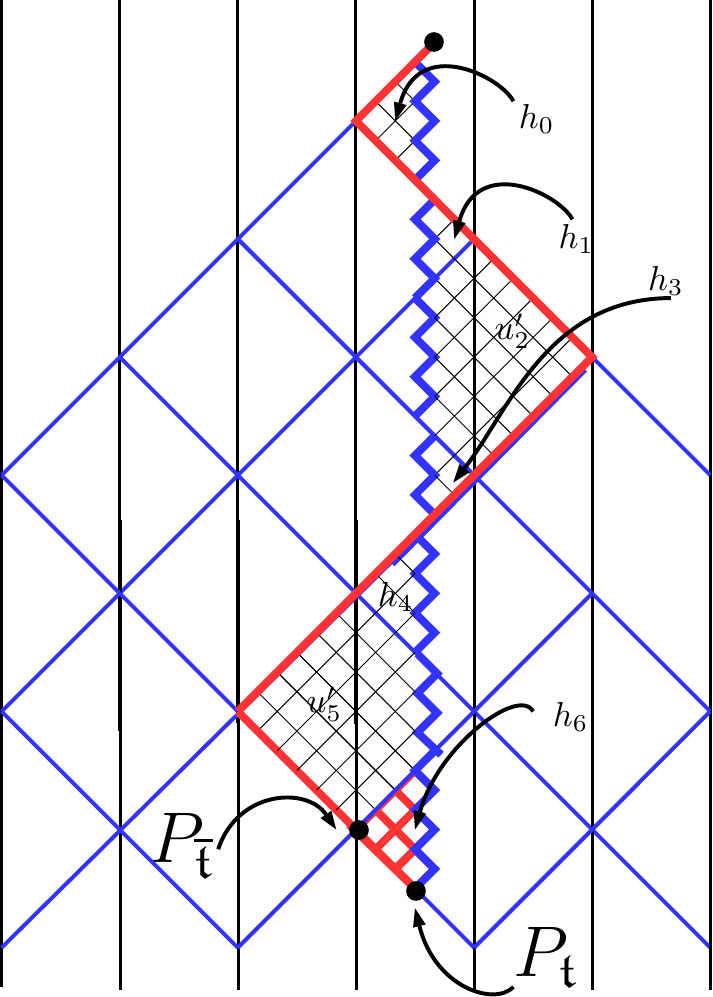}}
\end{equation}

\begin{definition}
Let $ \s = \overline{\s}(I) $ be an inner tableau. We say that   $\s$ is central if $\bar \s$ is central. 
\end{definition}

We can now prove the following Lemma.

\begin{lemma} \label{lemma outer elements IN}
  Let $ \s = \overline{\s}(I) $ and $ \T = \overline{\T}(I) $ be central inner tableaux in $\std_{\blambda}(\bmu)$.  Let  $\bar \bmu$ be the shape of $\bar \s$ and $\bar \T$. Then, we have
  \begin{equation}\label{lemmaouterelements In}
   m^{\bmu}_{ \s \T}  
    = \pm 
    \left\{  \begin{array}{cc}
      (y_{ \bar n +1} - y_{ \bar n}) \iota (m^{\overline{\bmu}}_{\overline{ \s} \overline{\T}})
=  \iota (m^{\overline{\bmu}}_{\overline{ \s} \overline{\T}}) (y_{ \bar n +1} - y_{ \bar n}),
      & \mbox{ if } \bmu \not\in {\cal A}^0; \\ & \\
 y_{ \bar n +1} \iota ( m^{\overline{\bmu}}_{\overline{ \s} \overline{\T}} )=
 \iota ( m^{\overline{\bmu}}_{\overline{ \s} \overline{\T}} ) y_{ \bar n +1}, 
& \mbox{ if } \bmu \in {\cal A}^0. 
\end{array} \right.
    \end{equation}  
\end{lemma}

\begin{proof}
  The proof is a calculation similar to the ones done in Lemma \ref{lemma second reduction of ys}
  and Theorem \ref{theoremsing}. Our general strategy is to first focus on the crosses that come from the region $s_{\bmu}$. Let us prove the first formula in \ref{lemmaouterelements In}.
  {\color{black}{Thus we assume that we are in the case where
$\bmu$ does not belong to the fundamental alcove}}. This case is a bit easier since, as we will see below, the crosses associated to the $s_{\bmu}$ region can be eliminated without altering the other parts of the diagram. 
  We illustrate the computation in the case where $ \s $ is given by  
  the first diagram of \ref{cutdiagramsREGSOUtER} and where $ \T = \s$. For these choices we calculate as
  follows, using the defining relations in 
  $ \B$ together with \ref{trenzadesataA}.
  \begin{equation}   m_{ \s \T} \,\, \mbox{\large$=$} \,\,
      \raisebox{-.42\height}{\includegraphics[scale=0.6]{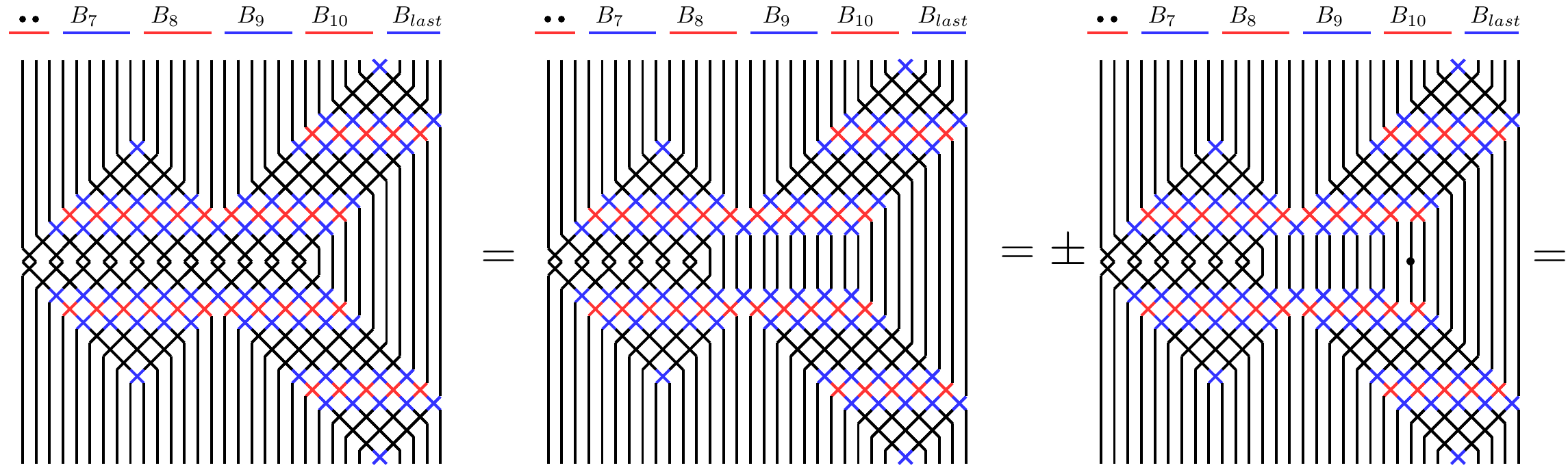}} 
  \end{equation}
    \begin{equation} 
      \raisebox{-.42\height}{\includegraphics[scale=0.6]{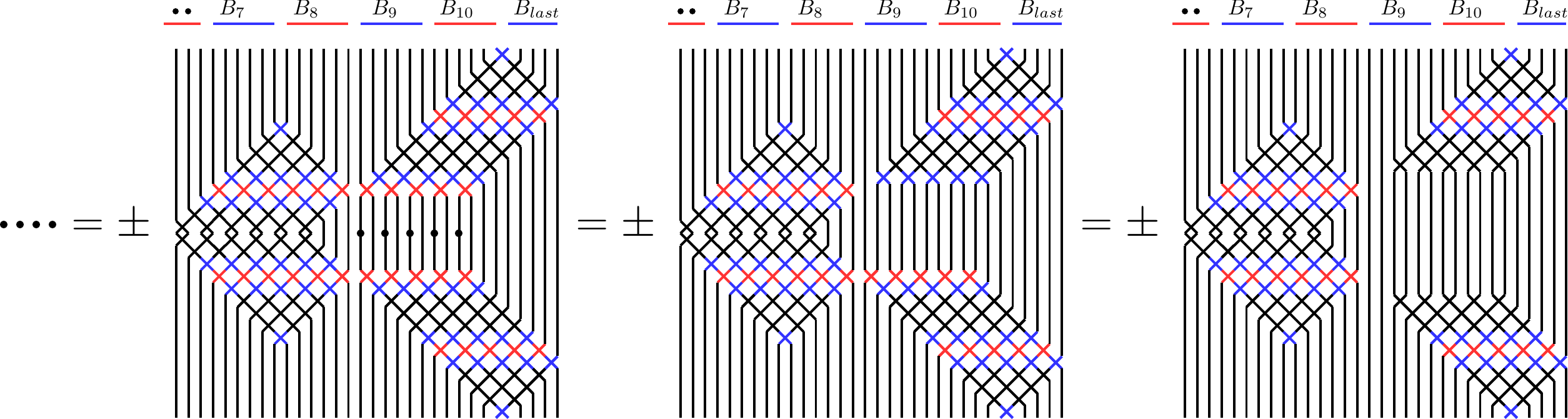}} 
  \end{equation}
    \begin{equation} 
      \raisebox{-.42\height}{\includegraphics[scale=0.6]{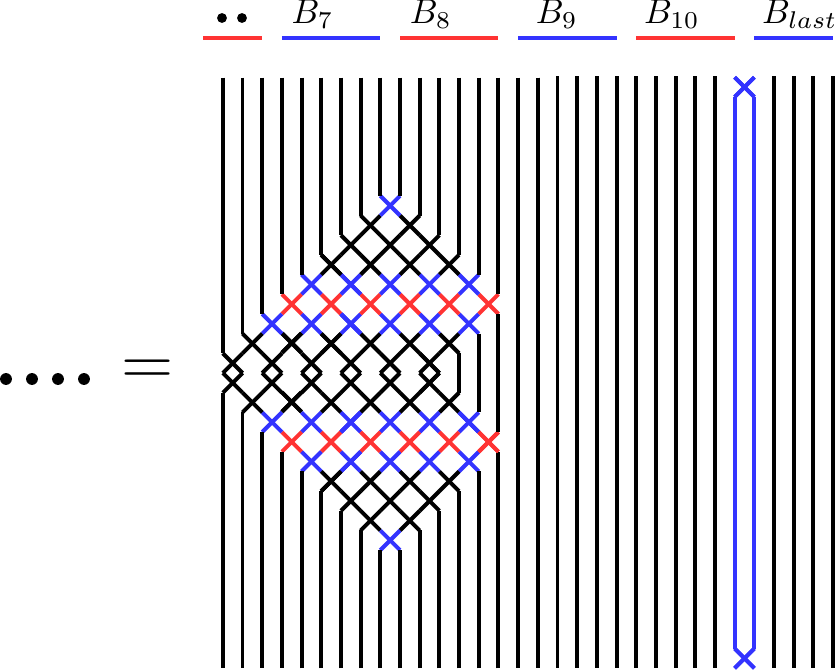}} \, \,  \mbox{\large  $=$ }
            \iota ( m^{\overline{\bmu}}_{\overline{ \s} \overline{\T}} ) \psi_{\bar n}^2 \mbox{\large  $=$ }
            (y_{\bar n+1} - y_{\bar n}) \iota ( m^{\overline{\bmu}}_{\overline{ \s} \overline{\T}} )
            \mbox{\large$=$} \,\, \iota ( m^{\overline{\bmu}}_{\overline{ \s} \overline{\T}} )
 (y_{\bar n +1} - y_{\bar n})
    \end{equation}
as claimed. The general case is done the same way.

\medskip
Let us now prove the second formula in \ref{lemmaouterelements In}, {\color{black}{corresponding to
  the case where $ \bmu $ belongs to the fundamental alcove.}}
In this case $ s_{ \bmu} $ is as small as possible, as for example in 
the second diagram of \ref{cutdiagramsREGSOUtER}.
The proof is essentially the same as the proof of the first formula with the only difference being
the vanishing of the factor $ y_{\bar n} $ which is due to Lemma \ref{lemma muere y a la izquerda}.
Let us do the calculation in the case where $ \s $ is given by the
second diagram of \ref{cutdiagramsREGSOUtER}, and $ \T = \s $.
We have then 

\begin{equation}\label{result}
 m_{ \s \T}^{\bmu}    =  \raisebox{-.42\height}{\includegraphics[scale=0.6]{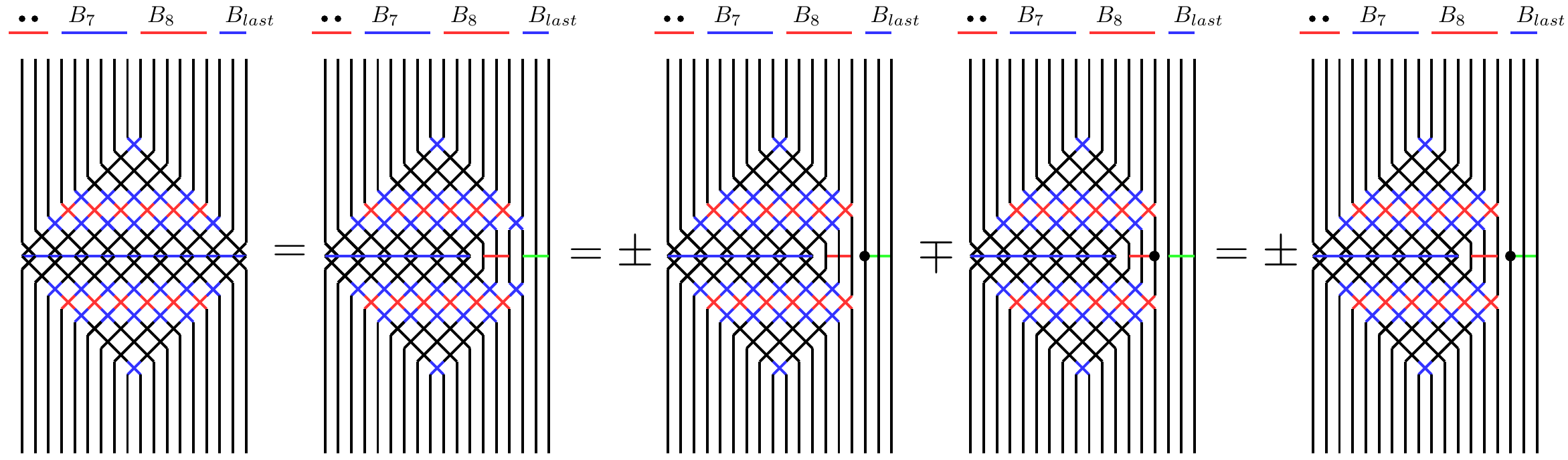}} 
\end{equation}  
where the blue horizontal, {\color{black}{red and green lines have the same meaning as in \ref{KLRmstX}.
    The fact that the fourth diagram of \ref{result} vanishes is
    shown using Lemma \ref{lemma muere y a la izquerda},
    arguing the same way as two paragraphs above \ref{KLRmstXBlock}, in the proof of Theorem \ref{bigtheorem}.}} This proves the Lemma.

\end{proof}

Suppose that $ i $ in any element of $ B_{last}$. Then we extend
the definition in \ref{blocki} by setting 
\begin{equation}
  \YL_{K+1}:= y_i e(\bi^\blambda) \in \trunc .
\end{equation}
We get from Lemma \ref{blockcancel} that $   \YL_{K+1}$ is independent of
the choice of $ i$.

\begin{corollary}\label{genA}
Let $ G_1(\blambda) $ be as in Theorem \ref{theorem generating set without superfluous generators singular }. Then the set 
  \begin{equation}
G_2(\blambda) := G_1(\blambda) \cup \{   \YL_{K+1} \} 
  \end{equation}    
generates $ \trunc$.
\end{corollary}

\begin{proof}
Let $\trunc '$ be the subalgebra of $\trunc$ generated by $G_2 (\blambda)$.
	Let $\s , \T \in \std_{\blambda} (\bmu) $. We need to show that $m_{\s\T}^\bmu \in \trunc '$. If $\s , \T$ are outer tableaux then the result follows by a combination of Theorem \ref{theorem generating set without superfluous generators singular } and Lemma \ref{lemma outer elements}. Suppose now that $\s$ and $\T$ are inner tableaux. If both tableaux are central then the result follows by combining Theorem \ref{theorem generating set without superfluous generators singular } and Lemma   \ref{lemma outer elements IN}. Otherwise,  the same argument given in the proof of Lemma \ref{lemma reduction to central}  allows us to conclude that there exist central standard tableaux $\s_1, \T_1 \in \std_{\blambda} (\bmu )$ and monomials $M_{\s}$ and $M_{\T}$ in the generators $\{ U_1^\blambda , \ldots , U_{K-1}^\blambda \}$ such that
	\begin{equation}
		m_{\s\T}^\bmu  = M_\s  m_{\s_1 \T_1}^\bmu M_\T, 
	\end{equation}
	and the result follows in this case as well. 
\end{proof}

\begin{corollary}\label{central}
$  \YL_{K+1} $ is a central element of $ \trunc$. 
\end{corollary}  
\begin{proof}  
This follows from Corollary \ref{genA} once we notice that $\YL_{K+1}$ commutes with all the elements of $G_1(\blambda)$. 
\end{proof}  

\begin{lemma}\label{square}
We have that $  (\YL_{K+1})^2 =0  $.
\end{lemma}  
\begin{proof}  
  For $ i= 1, 2 \ldots, K+1$ we introduce the following elements of $ \trunc$
\begin{equation}
  {\mathcal L }^{\blambda}_i := \YL_{i} - \YL_{i-1} 
  \end{equation}
with the convention that $ \YL_{0} :=0$.
Then in Theorem 6.9 of \cite{Esp-Pl} it was shown that these elements $ {\mathcal L }^{\blambda}_i$ satisfy the
JM-relations of Lemma \ref{JMB}. On the other hand we have that
\begin{equation}
\YL_{K+1} =   {\mathcal L }^{\blambda}_{K+1}  +   {\mathcal L }^{\blambda}_K  + \ldots+   {\mathcal L }^{\blambda}_1
\end{equation}
and so the calculation done in \eqref{JorgeDavidA} shows that $ (\YL_{K+1})^2 = 0 $, as claimed.
The Lemma is proved.
\end{proof}  

We can now establish the connection between the extended nil-blob algebra and $\trunc$. 

\begin{theorem}  \label{main theorem regular}
  Suppose that $\blambda $ is regular. Then the
  assignment $ \mathbb{U}_0\mapsto \mathcal{Y}_1^\blambda$, $\mathbb{J}_{K}\mapsto \YL_{K+1} $
  and  $\mathbb{U}_i\mapsto (-1)^e U_i^\blambda$ for all $1\leq i < K$, induces an $\mathbb{F}$-algebra
    isomorphism between $\widetilde{ \mathbb{NB}}_{K}$ and $\trunc  $.
\end{theorem}
\begin{proof}
  Combining Theorem \ref{theoremsing}, Corollary \ref{central} and Lemma \ref{square}
  we get that the assignment of the Theorem defines an algebra homomorphism, which is 
  surjective in view of Corollary \ref{genA}. The two algebras have the same dimension
  $ 2 \binom{2K}{K}  $, and hence the Theorem is proved.
\end{proof}  


\begin{theorem}
	Let $\blambda $ be a regular bipartition. Suppose that $\blambda$ is located in the alcove $\mathcal{A}_w$. Then, $\tilde{A}_w \cong \trunc $ as $\mathbb{F}$-algebras. 
\end{theorem}

\begin{proof}
	This is an immediate consequence of Corollary \ref{corollary short presentation} and Theorem \ref{main theorem regular}.
\end{proof}

\bibliographystyle{myalpha} 
\bibliography{mybibfile}

\sc
diego.lobos@pucv.cl, Universidad de Talca/Pontificia Universidad Cat\'olica de Valparaiso, Chile. \newline
dplaza@inst-mat.utalca.cl, Universidad de Talca, Chile. \newline
steen@inst-mat.utalca.cl, Universidad de Talca, Chile.


\end{document}